\documentclass[12pt]{article}
\usepackage{amsmath, amsfonts, amsthm}
\usepackage{amssymb,mathrsfs, verbatim}
\usepackage{graphicx,amstext}
\usepackage[mathcal]{euscript}
\usepackage{mathtools}

\newcommand{\clg}[1]{{\mathcal{#1}}}

\newcommand{\PM}{\mathbb P}
\newcommand{\R}{\mathbb R}
\newcommand{\Z}{\mathbb Z}

\newcommand{\ve}{\varepsilon}
\newcommand{\vp}{\varphi}

\newcommand{\sgn}{\text{sgn}}

\newcommand \loc    {\text{loc}}

\newcommand{\Medint}{\mkern12mu\mbox{\vrule height4pt
         depth-3.2pt
          width5pt}\mkern-16.5mu\int\nolimits}

\newtheorem{theorem}{Theorem}[section]
\newtheorem{proposition}[theorem]{Proposition}
\newtheorem{remark}[theorem]{Remark}
\newtheorem{lemma}[theorem]{Lemma}

\newtheorem{definition}[theorem]{Definition}
\newtheorem{example}[theorem]{Example}

\begin{document}
\title{Homogenization of Schr\"odinger equations.
\\
Extended Effective Mass Theorems 
\\
for non-crystalline matter}

\author{Vernny Ccajma$^1$, Wladimir Neves$^{1}$, Jean Silva$^2$}

\date{}

\maketitle

\footnotetext[1]{ Instituto de Matem\'atica, Universidade Federal
do Rio de Janeiro, C.P. 68530, Cidade Universit\'aria 21945-970,
Rio de Janeiro, Brazil. E-mail: {\sl
v.ccajma@gmail.com, wladimir@im.ufrj.br.}}
\footnotetext[2]{ Departamento de Matem\'atica, Universidade Federal
de Minas Gerais. E-mail: {\sl jcsmat@mat.ufmg.br}}

\textit{Key words and phrases. Homogenization theory,
stochastic Schr\"odinger equations, initial value problem, extended effective mass theorems.}

%
%
\begin{abstract} This paper concerns the homogenization of Schr\"odinger equations for non-crystalline matter, 
that is to say the coefficients are given by the composition 
of stationary functions with stochastic deformations. 
Two rigorous results of so-called effective mass theorems in solid state physics are obtained: 
a general abstract result (beyond the classical stationary 
ergodic setting), and one for quasi-perfect materials (i.e. the disorder in the 
non-crystalline matter is limited). 
The former relies on the double-scale 
limits and the wave function is spanned on the Bloch basis. Therefore, we have extended 
the Bloch Theory which was restrict until now to crystals (periodic setting). The second 
result relies on the Perturbation Theory and a special case of stochastic deformations, 
namely stochastic perturbation of the identity.  
\end{abstract}

\maketitle



\section{Introduction}
\label{Intro}

In this work we study the homogenization (assymptotic 
limit as $\varepsilon \to 0$) of the anisotropic
Schr\"odinger equation in the following Cauchy problem
\begin{equation}
\label{jhjkhkjhkj765675233}
	\left\{
	\begin{aligned}
&	i\displaystyle\frac{\partial u_\varepsilon}{\partial t} - {\rm div} {\big(A( \Phi^{-1} {\big( \frac{x}{\varepsilon}, \omega \big)},\omega) \nabla u_\varepsilon \big)} 
        + \frac{1}{\varepsilon^2} V( \Phi^{-1} {\big( \displaystyle\frac{x}{\varepsilon}, \omega \big)},\omega) \; u_\varepsilon
\\[5pt]
	&\hspace{100pt}	+ U( \Phi^{-1} {\big( \frac{x}{\varepsilon}, \omega \big)},\omega) \; u_\varepsilon = 0, 
		\quad \text{in $\mathbb{R}^{n+1}_T \! \times \! \Omega$}, 
\\[3pt]
		& u_\varepsilon= u_\varepsilon^0, \quad \text{in $\mathbb{R}^n \! \times \! \Omega$},
	\end{aligned}
	\right.
\end{equation} 
where $\mathbb{R}^{n+1}_T := (0,T) \times \mathbb{R}^n$, for any real number $T> 0$, $\Omega$ is a probability space, 
and the unknown function $u_\varepsilon(t,x,\omega)$ is complex-value.  

\medskip
The coefficients in \eqref{jhjkhkjhkj765675233}, that is the matrix-value 
function $A$, the real-value (potencial) functions $V$, $U$ 
are random perturbations of stationary functions accomplished by 
stochastic diffeomorphisms $\Phi: \R^n \times \Omega \to \R^n$, (called stochastic deformations).
The stationarity property of random 
functions will be precisely defined in Section \ref{628739yhf}, also the definition of 
stochastic deformations which were introduced by 
X. Blanc, C. Le Bris, P.-L. Lions (see \cite{BlancLeBrisLions1,BlancLeBrisLions2}).
In that paper they consider the homogenization 
problem of an elliptic operator whose coefficients are periodic or stationary functions perturbed by 
stochastic deformations. 

\medskip
In particular, we assume that  $A = (A_{k \ell})$, $V$ and $U$ are measurable and
bounded functions, i.e. for $k, \ell= 1,\ldots,n$
\begin{equation}
\label{ASSUM1}
     A_{k \ell}, \; V, \; U \in L^\infty(\mathbb{R}^n \times \Omega).
\end{equation}
Moreover, the matrix $A$ is symmetric and 
uniformly positive defined, that is, there exists $a_0> 0$, such that, for a.a. 
$(y, \omega) \in \mathbb{R}^n \times \Omega$, and each $\xi \in \R^n$
\begin{equation}
\label{ASSUM2}
	\sum_{k,\ell=1}^n A_{k\ell}(y,\omega)\,  \xi_k \, \xi_\ell \geqslant a_0 {\vert \xi \vert}^2.
\end{equation}

\medskip
This paper is the second part of the Project initiated with  
T. Andrade, W. Neves, J. Silva \cite{AndradeNevesSilva}
(Homogenization of Liouville Equations 
beyond stationary ergodic setting)
concerning the study of moving electrons in non-crystalline matter, 
which justify the form considered for the coefficients 
in \eqref{jhjkhkjhkj765675233}. 
We recall that crystalline materials, also called perfect materials, are described 
by periodic functions. Thus any homogenization result for
Schr\"odinger equations with periodic coefficients is restrict to 
crystalline matter. 
Moreover, perfect materials are rare in Nature, there 
exist much more non-crystalline than crystalline materials.   
For instance, there exists a huge class called quasi-perfect materials
(see Section \ref{6775765ff0090sds}, also \cite{AndradeNevesSilva}), which are closer to 
perfect ones. Indeed, the concept of stochastic deformations 
are very suitable to describe interstitial defects in 
materials science
(see Cances, Le Bris \cite{CancesLeBris}, 
and Myers \cite{Myers}).

\medskip
One remarks that, the homogenization of the Schr\"odinger equation
in \eqref{jhjkhkjhkj765675233},
when the stochastic deformation $\Phi(y,\omega)$ 
is the identity mapping and the coefficients are periodic, 
were studied by Allaire, Piatnitski \cite{AllairePiatnitski}. 
Notably, that paper presents the discussion about the 
differences between the scaling considered in \eqref{jhjkhkjhkj765675233}
and the one called semi-classical limit. 
We are not going to rephrase
this point here, and address the reader to Chapter 4 in \cite{BensoussanLionsPapanicolaou}
for more general considerations about that. It should be mentioned that, to the best of
our knowledge the present work is the first to study the homogenization 
of the Schr\"odinger equations beyond the periodic setting, applying the double-scale 
limits and the wave function is spanned on the Bloch basis. 

\medskip
Last but not least, one observes that the initial data $u_\varepsilon^0$ 
shall be considered well-prepared, see equation \eqref{WellPreparedness}. 
This assumption is fundamental for the abstract homogenization result 
established in Theorem \ref{876427463tggfdhgdfgkkjjlmk}, where the 
limit function obtained from $u_\varepsilon$
satisfies a simpler Schr\"odinger equation, called the effective mass equation,
with effective constant coefficients, namely matrix $A^*$, and potential $V^*$.
This homogenization procedure is well known in solid state physics as 
Effective Mass Theorems, see Section \ref{HomoSchEqu}. 

\medskip
Finally, we stress Section \ref{6775765ff0090sds} which is 
related to the homogenization of the Schr\"odinger equation for quasi-perfect materials,
and it is also an important part of this paper. 
Indeed, a very special case occurs in situations where
the amount of randomness is small, more specifically the disorder in the 
material is limited. In particular, this section is interesting  
for numerical applications, where specific computationally efficient techniques 
already designed to deal with the homogenization of the Schr\"odinger equation in the periodic setting, 
can be employed to treat the case of quasi-perfect materials. 

\subsection{Contextualization}

Let us briefly recall that the homogenization's problem for \eqref{jhjkhkjhkj765675233} has been treated for the periodic case 
($A_{\rm per}(y)$, $V_{\rm per}(y)$, $U_{\rm per}(y)$), 
and  $\Phi(y,\omega) = y$ by some authors. Besides the paper by G. Allaire, A.Piatnitski \cite{AllairePiatnitski} already mentioned,
we address the following papers for the case of $A_{\rm per}= I_{n \times n}$, i.e. isotropic 
Schr\"odinger equation in \eqref{jhjkhkjhkj765675233}: G. Allaire, M.Vanninathan \cite{AllaireVanninathan}, L. Barletti, N. Ben Abdallah \cite{BarlettiBenAbdallah}, 
V. Chabu, C. Fermanian-Kammerer, F. Marcia  \cite{ChabuFermanianMarcia}, and we observe that this list is by no means exhaustive. 
In \cite{AllaireVanninathan}, the authors study a semiconductors model excited by an external potencial $U_{\rm per}(t,x)$, which depends on 
the time $t$ and macroscopic variable $x$. 
In  \cite{BarlettiBenAbdallah} the authors
treat the homogenization's problem when the external 
potential $U_{\rm per}(x,y)$ depends also on the macroscopic variable $x$. 
Finally, in \cite{ChabuFermanianMarcia}
it was considered an external potential $U_{\rm per}(t,x)$ 
which model the effects of impurities on the otherwise perfect matter. 

\medskip
All the references cited above treat the homogenization's problem 
for \eqref{jhjkhkjhkj765675233}, studying the spectrum of the associated 
Bloch spectral cell equation, that is, for each $ \theta \in \mathbb{R}^n$,
find the eigenvalue-eigenfunction pair $(\lambda,\psi)$, satisfying  
\begin{equation}
\label{8756trg}
\left\{
\begin{aligned}
L_{\rm per}(\theta) {\big[ \psi \big]}&= \lambda \, \psi, 
\quad \text{in $[0,1)^n$},
\\[5pt]
\psi(y)&\not= 0, \quad \text{periodic function},
\end{aligned}
\right.
\end{equation}
where $L_{\rm per}(\theta)$ is the Hamiltonian given by 
$$
L_{\rm per}(\theta){\big[ f \big]}= -{\big( {\rm div}_{\! y} + 2i\pi \theta \big)} {\big[ A_{\rm per}(y) {( \nabla_{\!\! y} 
+ 2i \pi \theta)} f \big]} + V_{\rm per}(y) f.
$$
The above eigenvalue problem is precisely stated (in the more general context studied in this paper) in 
Section \ref{877853467yd56rtfe5rtfgeds76ytged}. 
Here, concerning the periodic setting mathematical solutions to \eqref{8756trg}, we address the reader to 
C. H. Wilcox \cite{Wilcox}, (see in particular Section 2: A discussion of related literature). Then,
once this eigenvalue problem is resolved, the goal is to pass to the limit as $\varepsilon \to 0$. One remarks that, 
there does not exist an uniform estimate in $H^1(\R^n)$ for the family of solutions $\{u_\varepsilon\}$ of \eqref{jhjkhkjhkj765675233}, 
due to the scale $\varepsilon^{-2}$ multiplying the 
internal potential $V_{\rm per}(y)$. To accomplish the desired asymptotic limit, under this lack of compactness,
a nice strategy is to use the two-scale convergence, for instance see the proof of Theorem 3.2 in \cite{AllairePiatnitski}. 

\medskip
Let us now focus on the stochastic setting studied here, where the coefficients of
the Schr\"odinger equation in \eqref{jhjkhkjhkj765675233} are the composition of
stationary functions with stochastic deformations. First, we establish an analogously 
Bloch spectral cell equation assuming that the solution of 
equation \eqref{jhjkhkjhkj765675233} is given by a plane wave. Then, 
the stochastic spectral Bloch cell equation
\eqref{92347828454trfhfd4rfghjls}
is obtained applying the
asymptotic expansion WKB method.
More specifically, the Hamiltonian in \eqref{92347828454trfhfd4rfghjls}
is given by
$$
   L^\Phi(\theta)\big[ F \big]\! \! = -\big( {\rm div}_{\! z} + 2i\pi \theta \big){\left[ A{( \Phi^{-1}(z,\omega),\omega)} {\big( \nabla_{\!\! z} + 2i\pi\theta \big)} F \right]} 
   + V( \Phi^{-1}(z,\omega),\omega) F, 
$$
for each $F(z,\omega)= f\left( \Phi^{-1}(z,\omega),\omega \right)$, where $f(y,\omega)$ is a stationary function. 

\medskip
To follow, we consider the spectrum of the operator $L^\Phi(\theta)$, for each $\theta \in \R^n$ fixed. 
First, we follow the techniques 
applied for the periodic setting, and consider a characterization of the Rellich-Kondrachov Theorem on groups
given in \cite{VCWNJS}. One observes that, $\omega \in \Omega$ 
can not be treat as a fixed parameter. 
Then we obtain an abstract theorem, see Theorem \ref{876427463tggfdhgdfgkkjjlmk},
where it is given the assymptotic 
limit as $\varepsilon \to 0$ of the Cauchy problem \eqref{jhjkhkjhkj765675233}. 
For that, we apply the two-scale convergence (stochastic setting), developed in
Section \ref{pud63656bg254v2v5}, which is beyond the classical stationary 
ergodic setting. The main difference here with the earlier stochastic extensions of the periodic setting is 
that, the test functions used are random 
perturbations of stationary functions accomplished by 
the stochastic deformations. These compositions are beyond the stationary class, thus we have a 
lack of the stationarity property in this kind of test functions (see the introduction section in \cite{AndradeNevesSilva} 
for a deep discussion about this subject). It was introduced a compactification argument that, preserves the ergodic 
nature of the setting involved and allow us to overcome these difficulties.

The second applied strategy here to study 
the spectrum of the operator $L^\Phi(\theta)$ is  
the Perturbation Theory. Therefore, taking the advantage of the well known spectrum for $L_{\rm per}(\theta)$, 
we consider that the coefficients of
the Schr\"odinger equation in \eqref{jhjkhkjhkj765675233} are the composition of the periodic 
functions $A_{\rm per}$, $V_{\rm per}$ and $U_{\rm per}$ with a special case of stochastic deformations, 
namely stochastic perturbation of the identity (see Definition \ref{37285gdhddddddddddd}), 
which is given by 
$$
   \Phi_\eta(y,\omega) := y + \eta  \, Z(y,\omega) + \mathrm{O}(\eta^2),
$$
where $Z$ is some stochastic deformation and $\eta \in (0,1)$. This concept was introduced 
by X. Blanc, C. Le Bris, P.-L. Lions \cite{BlancLeBrisLions2}, and applied to 
evolutionary equations in T. Andrade, W. Neves, J. Silva \cite{AndradeNevesSilva}. 

\section{Preliminaries and Background}
\label{PrelmBackg}

This section introduces the basement theory, which will be used through the paper. 
To begin we fix some notations, and collect some preliminary results. The material which is well-known or a direct extension 
of existing work are giving without proofs, otherwise we present them.

\medskip
We denote by $\mathbb{G}$ the group $\mathbb{Z}^n$ (or $\mathbb{R}^n$), with $n \in \mathbb{N}$.
The set $[0,1)^n$
denotes the unit cube, which is also called the unitary cell and will be used  
as the reference period for periodic functions.
The symbol $\left\lfloor x \right\rfloor$ denotes the 
unique number in $\mathbb{Z}^n$, such that $x - \left\lfloor x \right\rfloor \in [0,1)^n$.
Let $H$ be a complex Hilbert space, we denote by $\mathcal{B}(H)$ 
the Banach space of linear bounded operators from $H$ to $H$.

\medskip
Let $U \subset \R^{n}$ be an open set, $p \geqslant  1$, and $s \in \mathbb{R}$.
We denote by 
$L^p(U)$ the set of (real or complex) $p-$summable functions
with respect to the Lebesgue measure (vector ones should be understood
componentwise). Given a Lebesgue measurable set
$E \subset \R^n$, 
$|E|$ denotes its $n-$dimensional Lebesgue measure.
Moreover, we will use the standard notations for the 
Sobolev spaces $W^{s,p}(U)$ and $H^{s}(U)\equiv W^{s,2}(U)$. 

\subsection{Anisotropic Schr\"odinger equations}
\label{SchEq}

Here we present the well-posedness for the solutions of the Schr\"odinger equation, 
and some properties of them. 
Most of the material can be found in
Cazenave, Haraux \cite{CazenaveHaraux}.

\medskip
First, let us consider the following Cauchy problem, which is driven by a linear anisotropic
Schr\"odinger equation, that is
\begin{equation}
\label{87644343}
   \left\{
   \begin{aligned}
   &i \; \partial_t u(t,x) - {\rm div} \big(A(x) \nabla u(t,x) \big)
   + V(x) \, u(t,x) = 0 \quad  \text{in $\mathbb{R}^{n+1}_T$}, 
   \\[5pt]
   & u(0,x)=u_0(x) \quad \text{in $\mathbb{R}^n$},
   \end{aligned}
  \right.
\end{equation}
where the unknown $u(t,x)$ is a complex value function, and $u_0$
is a given initial datum. The coefficient $A(x)$ is a symmetric real $n \times n$-matrix 
value function, and the potential $V(x)$ is a real function. We always assume that
\begin{equation}
\label{CONDITAV}
     A(x), V(x) \quad \text{are measurable bounded functions}. 
\end{equation}
One recalls that, a matrix $A$ is called (uniformly) coercive, when, there exists $a_0> 0$, 
such that, for each $\xi \in \mathbb{R}^n$, and almost all $x \in \mathbb{R}^n$,
$A(x) \xi \cdot \xi \geqslant a_0 \vert \xi \vert^2$. 

\medskip
The following definition tell us in which sense a complex function $u(t,x)$ is a mild solution to \eqref{87644343}.
\begin{definition}
\label{MildSol}
Let $A, V$ be coefficients satisfying \eqref{CONDITAV}. 
Given $u_0 \in H^1(\mathbb{R}^n)$, a function 
$$
   u \in C( [0,T]; H^1(\mathbb{R}^n)) \cap C^1((0,T); H^{-1}(\mathbb{R}^n))
$$
is called a mild solution to the Cauchy problem \eqref{87644343}, when for each $t \in (0,T)$, it follows that
\begin{equation}
\label{DEFSOLSCH}
   i \partial_t u(t) -{\rm div} \big(A \nabla u(t) \big) + V u(t) = 0 \quad \text{in $H^{-1}(\mathbb{R}^n)$}, 
\end{equation}
and $u(0)= u_0$ in $H^1(\mathbb{R}^n)$.
\end{definition}	

Then, we state the following 
\begin{proposition}
\label{PROPEUSCHEQ}
Let $A$ be a coercive matriz value function, $V$ a potential and 
$u_0 \in H^1(\mathbb{R}^n)$  a given initial data. 
Assume that $A, V$ satisfy \eqref{CONDITAV}. Then, there exist a unique 
mild solution of the Cauchy problem \eqref{87644343}. 
\end{proposition}

\begin{proof}
The proof follows applying Lemma 4.1.5 and Corollary 4.1.2 in \cite{CazenaveHaraux}.
\end{proof}

\medskip
\begin{remark}
\label{REMCOSTCOEFF}
It is very important in the homogenization procedure of the Schr\"odinger equation, 
when the coefficients $A$ and $V$ in \eqref{87644343} are constants,
the matrix $A$ is not necessarily coercive, and the initial data 
$u_0 \in L^2(\mathbb{R}^n)$. Then, a function $u \in L^2(\mathbb{R}^{n+1}_T)$ is called a
weak solution to \eqref{87644343}, if it satisfies 
$$
   i \partial_t u - {\rm tr}(A D^2 u) + V u = 0 \quad \text{in distribution sense}.
$$
Since $A, V$ are constant, we may apply the Fourier Transform, and
obtain the existence of a unique solution $u \in H^1((0,T); L^2(\mathbb{R}^{n}))$. 
Therefore, the solution $u \in  C([0,T]; L^2(\mathbb{R}^{n}))$ after being 
redefined in a set of measure zero, and we have $u(0)= u_0$ in $L^2(\mathbb{R}^n)$.
\end{remark}

\subsection{Stochastic configuration}
\label{628739yhf}

Here we present the stochastic context, which will be used thoroughly in the paper. 
To begin, let $(\Omega, \mathcal{F}, \mathbb{P})$ be a probability space. For each random variable  
$f$ in $L^1(\Omega; \PM)$, ($L^1(\Omega)$ for short), 
we denote its expectation value by
$$
    \mathbb{E}[f]= \int_\Omega f(\omega) \ d\PM(\omega).
$$

A mapping $\tau: \mathbb{G} \times \Omega \to \Omega$ is said a $n-$dimensional dynamical 
system if:
\begin{enumerate}
\item[(i)](Group Property) $\tau(0,\cdot)=id_{\Omega}$ and $\tau(x+y,\omega)=\tau(x,\tau(y,\omega))$ for all $x,y \in  \mathbb{G}$ 
and $\omega\in\Omega$.
\item[(ii)](Invariance) The mappings $\tau(x,\cdot):\Omega\to \Omega$ are $\PM$-measure preserving, that is, for each $x \in  \mathbb{G}$ and 
every $E\in \mathcal{F}$, we have 
$$
\tau(x,E)\in \mathcal{F},\qquad \PM(\tau(x,E))=\PM(E).
$$
\end{enumerate}
For simplicity, we shall use $\tau(k)\omega$ to denote $\tau(k,\omega)$. Moreover, it is usual to say that 
$\tau(k)$ is a discrete (continuous) dynamical system if $k \in \Z^n$ ($k \in \R^n$), but we only stress 
this when it is not obvious from the context. 

\medskip
A measurable function $f$ on $\Omega$ is called $\tau$-invariant, if for each $k \in \mathbb{G}$ 
$$
    f(\tau(k) \omega)= f(\omega) \quad \text{for almost all $\omega \in \Omega$}. 
$$
Hence a measurable set $E \in \mathcal{F}$ is $\tau$-invariant, if its characteristic function $\chi_E$ is $\tau$-invariant. 
In fact, it is a straightforward to show that, a $\tau$-invariant set $E$ can be equivalently defined by 
$
    \tau(k) E= E \quad \text{for each $k \in \mathbb{G}$}.
$
Moreover, we say that the dynamical system $\tau$ is ergodic, when
all $\tau$-invariant sets $E$ have measure $\PM(E)$ of either zero or one. 
Equivalently, we may characterize an ergodic dynamical system
in terms of invariant functions. Indeed, a dynamical system is ergodic if 
each $\tau$- invariant function is constant almost everywhere, that is to say 
$$
    \Big( f(\tau(k) \omega)= f(\omega) \quad \text{for each $k \in \mathbb{G}$ and a.e. $\omega \in \Omega$} \Big) 
    \Rightarrow \text{ $f(\cdot)= const.$ a.e.}.  
$$


\begin{example}
\label{EXTJING}
Let $(\Omega_0,\mathscr{F}_0,\mathbb{P}_0)$ be a probability space. 
For $m \in \mathbb{N}$ fixed, we consider the set $S= \{0,1,2,\ldots,m\}$ 
and the real numbers $p_0, p_1, p_2, \ldots, p_m$ in $(0,1)$, such that 
$\sum_{\ell= 0}^m p_\ell=1$. 
If $ \{X_k:\Omega_0 \to S \}_{k\in\mathbb{Z}^n}$
is a family of random variables, then it is induced a probability measure from it on the measurable space  
$\big( S^{\mathbb{Z}^n}, \bigotimes_{k\in\mathbb{Z}^n}2^{S} \big)$. Indeed, we may define the probability 
measure
$
\mathbb{P}(E):= \mathbb{P}_0{\left\{ X \in E \right\}}, \;\; E \in \bigotimes_{k\in\mathbb{Z}^n}2^{S},
$
where the mapping $X: \Omega_0 \to S^{\mathbb{Z}^n}$ is given by 
$X(\omega_0)= (X_k(\omega_0))_{k\in\mathbb{Z}^n}$.

\medskip
Now, we denote for convenience $\Omega= S^{\mathbb{Z}^n}$
and $\mathscr{F}= \bigotimes_{k\in\mathbb{Z}^n}2^{S}$, 
that is, $\mathscr{F}= \sigma(\mathscr{A})$, 
where $\mathscr{A}$ is the algebra given by the finite union of sets (cylinders of finite base) 
of the form 
$    \prod_{k \in \mathbb{Z}^n} E_k$, 
where $E_k \in 2^S$ is different from $S$ for a finite number of indices $k$. Additionally we assume that,
the family 	$\{X_k\}_{k\in\mathbb{Z}^n}$ is independent, and for each $k \in \mathbb{Z}^n$, we have 
\begin{equation}
\label{243}
	\mathbb{P}_0{\{ X_k=0 \}}= p_0, \,\, \mathbb{P}_0{\{ X_k=1 \}}= p_1, \,\, \ldots, \,\, \mathbb{P}_0{\{ X_k=m \}}= p_m.
\end{equation}
Then, we may define an ergodic dynamical system $\tau: \mathbb{Z}^n \times \Omega \to \Omega$, by
$( \tau (\ell) \omega)(k):= \omega(k + \ell)$, for any $k,\ell \in \mathbb{Z}^n$,
where $\omega= (\omega(k))_{k \in \mathbb{Z}^n}$.
\end{example}

\medskip
Now, let $(\Gamma, \mathcal{G}, \mathbb{Q})$ be a given probability space. We say that a
measurable function $g: \R^n \times\Gamma \to \R$ is stationary, if for any finite set 
consisting of points $x_1,\ldots,x_j\in \R^n$, and any $k \in \mathbb{G}$, the distribution of the random vector 
$
   \big(g(x_1+k,\cdot),\cdots,g(x_j+k,\cdot)\big)
$
is independent of $k$. Further, subjecting the stationary function $g$ to some natural conditions
it can be showed that, there exists other probability space $(\Omega, \mathcal{F}, \PM)$,  a $n-$dimensional dynamical system 
$\tau: \mathbb{G} \times \Omega \to \Omega$ and a measurable function $f: \R^n \times \Omega \to \R$ satisfying 
\begin{itemize}
\item For all $x \in \R^n$, $k \in \mathbb{G}$ and $\PM-$almost every $\omega \in \Omega$ 
\begin{equation}
\label{Stationary}
   f(x+k, \omega)= f(x, \tau(k) \omega).
\end{equation} 

\item For each $x \in \R^n$ the random variables $g(x,\cdot)$ and $f(x,\cdot)$ have the same 
law. We recall that, the equality almost surely implies 
equality in law, but the converse is not true. 

\end{itemize} 

One remarks that, the set of stationary functions forms an algebra, and
also is stable by limit process. 
For instance, the product of two
stationaries functions is a stationary one, and the derivative of a 
stationary function is stationary. 
Moreover, the stationarity concept is the most general extension of the 
notions of periodicity and almost periodicity for a function to have some "self-averaging" behaviour. 
	
\begin{example}
\label{8563249tyudh}
Under the conditions of Example \ref{EXTJING}, we take $m= 1$, and
consider the following functions, $\varphi_0 = 0$ and $\varphi_1$ be a Lipschitz vector field,  
such that, 
$\varphi_1$ is periodic, ${\rm supp} \, \varphi_1 \subset (0,1)^n$. 
Consequently, the function 
\begin{equation*}
    f(y,\omega) := \varphi_{\omega({\lfloor y \rfloor})} (y), \;\; (y,\omega) \in \mathbb{R}^n \times \! \Omega
\end{equation*}
satisfies,  ${ f(y,\cdot) }$ is ${ \mathscr{F} }$-measurable, ${ f(\cdot,\omega) }$ is continuous, and 
for each  $k \in \mathbb{Z}^n$, 
$$
    f(y+k,\omega)= f(y,\tau(k)\omega).
$$
Therefore, ${ f }$ is a stationary function. 
\end{example}

\medskip
Next we present the precise definition of the stochastic deformation as presented in \cite{AndradeNevesSilva}.
\begin{definition}
\label{GradPhiStationary}
A mapping $\Phi: \R^n \times \Omega \to \R^n, (y,\omega) \mapsto z= \Phi(y,\omega)$, is called a stochastic deformation (for short $\Phi_\omega$), when satisfies:
\begin{itemize}
\item[i)] For $\mathbb{P}-$almost every $\omega \in \Omega$, $\Phi(\cdot,\omega)$ is a bi--Lipschitz diffeomorphism.

\item[ii)] There exists $\nu> 0$, such that
$$
\underset{\omega \in \Omega, \, y \in \R^n}{\rm ess \, inf} 
\big({\rm det} \big(\nabla \Phi(y,\omega)\big)\big) \geq \nu.
$$
\item[iii)] There exists a $M> 0$, such that
$$
 \underset{\omega \in \Omega, \, y \in \R^n}{\rm ess \, sup}\big(|\nabla \Phi(y,\omega)|\big) \leq M< \infty.
$$
\item[iv)]
The gradient of $\Phi$, i.e. $\nabla\Phi(y,\omega)$, is stationary in the sense~\eqref{Stationary}.
\end{itemize}
\end{definition}

Now, let us present an interesting example of stochastic deformations. 
\begin{example}
\label{EXAMPLE8}
Under the conditions of Example \ref{8563249tyudh}, let us consider (for $\eta> 0$)
the following map 
$$ 
    \Phi(y,\omega):= y + \eta \, \varphi_{\omega({\lfloor y \rfloor})} (y), \;\; (y,\omega) \in \mathbb{R}^n \times \Omega.
$$
Then,  $\nabla_{\!\! y} \Phi(y,\omega) = I_{\mathbb{R}^{n\times n}} + \eta \, \nabla \varphi_{\omega({\lfloor y \rfloor})} (y)$,
and for $\eta$ sufficiently small all the conditions in the Definition \ref{GradPhiStationary} are satisfied.
Then, ${ \Phi }$ is a stochastic deformation.
\end{example}

Given a stochastic deformation $\Phi$, 
let us consider the following spaces
\begin{equation}
	\mathcal{L}_\Phi := {\big\{F(z,\omega)= f( \Phi^{-1} (z, \omega), \omega); f \in L^2_{\rm loc}(\mathbb{R}^n; L^2(\Omega)) \;\; \text{stationary} \big\}}
\end{equation}
and
\begin{equation}
\label{SPACEHPHI}
		\mathcal{H}_\Phi := {\big\{F(z,\omega)= f( \Phi^{-1} (z, \omega), \omega); \; f\in H^1_{\rm loc}(\mathbb{R}^n; L^2(\Omega)) \;\; \text{stationary} \big\}}
\end{equation}
which are Hilbert spaces, endowed respectively with the inner products  
$$
\begin{aligned}
  {\langle F, G \rangle}_{\mathcal{L}_\Phi}&:= \int_\Omega \int_{\Phi([0,1)^n,\omega)} \!\! F(z, \omega) \, \overline{ G(z, \omega) } \, dz \, d\mathbb{P}(\omega),
\\[5pt]
{\langle F, G \rangle}_{\mathcal{H}_\Phi}&:= \int_\Omega \int_{\Phi([0,1)^n,\omega)} \!\! F(z, \omega) \, \overline{ G(z, \omega) } \, dz \, d\mathbb{P}(\omega)
\\
&\; \quad +\int_\Omega \int_{\Phi([0,1)^n,\omega)} \!\! \nabla_{\!\! z} F(z, \omega) \cdot \overline{ \nabla_{\!\! z} G(z, \omega) } \, dz \, d\mathbb{P}(\omega). 
\end{aligned}
$$
\begin{remark}
\label{REMFPHI}
Under the above notations, 
when $\Phi= Id$ we denote $\mathcal{L}_\Phi$ and $\mathcal{H}_\Phi$ by  $\mathcal{L}$ and $\mathcal{H}$ respectively.
Moroever, a function $F \in \clg{H}_\Phi$ if, and only if, $F \circ \Phi \in \clg{H}$, and 
there exist constants $C_1, C_2> 0$, such that 
$$
   C_1 \|F \circ \Phi \|_{\clg{H}} \leq \|F \|_{\clg{H}_\Phi} \leq C_2 \|F \circ \Phi \|_{\clg{H}}.
$$
Analogously, $F \in \clg{L}_\Phi$ if, and only if, $F \circ \Phi \in \clg{L}$, and 
there exist constants $C_1, C_2> 0$, such that 
$$
   C_1 \|F \circ \Phi \|_{\clg{L}} \leq \|F \|_{\clg{L}_\Phi} \leq C_2 \|F \circ \Phi \|_{\clg{L}}.
$$  
Indeed, let us show the former equivalence.
Applying a change of variables, we obtain 
$$
\begin{aligned}
    \|F\|^2_{\clg{H}_\Phi}&= \int_\Omega \int_{\Phi([0,1)^n,\omega)} \!\! |F(z, \omega)|^2 \, dz \, d\mathbb{P}(\omega)
   +\int_\Omega \int_{\Phi([0,1)^n,\omega)} \!\! |\nabla_{\!\! z} F(z, \omega)|^2 \, dz \, d\mathbb{P}(\omega)
\\[5pt]   
   &= \int_\Omega \int_{[0,1)^n} \!\! |f(y, \omega)|^2  \det [\nabla \Phi(y,\omega)] \,  dy \, d\mathbb{P}(\omega)
\\[5pt]   
   &\quad +\int_\Omega \int_{[0,1)^n} \!\! | [\nabla \Phi(y,\omega)]^{-1} \nabla_{\!\! z} f(y, \omega)|^2  \det [\nabla \Phi(y,\omega)] \, dy \, d\mathbb{P}(\omega).
\end{aligned}
$$
The equivalence follows from the properties of the stochastic deformation $\Phi$. 
\end{remark}

\subsubsection{Ergodic theorems}
\label{ErgThm}

We begin this section with the concept of mean value, which is 
in connection with the notion of stationarity. 
A function $f \in L^1_{\loc}(\R^n)$ is said to possess a mean value if the 
sequence $\{f(\cdot/\ve){\}}_{\ve>0}$ converges in the duality with $L^{\infty}$ and compactly supported 
functions to a constant $M(f)$. This convergence is equivalent to
\begin{equation}
\label{MeanValue}
\lim_{t\to\infty}\frac1{t^n|A|}\int_{A_t}f(x)\,dx=M(f),
\end{equation}
where $A_t:=\{x\in\R^n\,:\, t^{-1}x\in A\}$, for $t>0$ and any $A \subset \R^n$, with $|A| \ne0$.

\begin{remark}
\label{REMERG}
Unless otherwise stated, we assume that the dynamical system $\tau: \mathbb{G} \times \Omega\to\Omega$ is ergodic 
and we will also use the notation 
$$
   \Medint_{\R^n} f(x) \ dx \quad \text{for $M(f)$}.
$$
\end{remark}

Now, we state the result due to Birkhoff, which connects all the notions 
considered before, see \cite{Krengel}. 

\begin{theorem}[Birkhoff Ergodic Theorem]\label{Birkhoff}
Let $f \in L^1_\loc(\R^n; L^1(\Omega))$ be a stationary random variable. 
Then, for almost every $\widetilde{\omega} \in \Omega$ the function 
$f(\cdot,\widetilde{\omega})$ possesses a mean value in the sense of~\eqref{MeanValue}. Moreover, the mean value 
$M\left(f(\cdot,\widetilde{\omega})\right)$ as a function of $\widetilde{\omega} \in\Omega$ satisfies
for almost every $\widetilde{\omega} \in \Omega$: 

\smallskip
i) Discrete case (i.e. $\tau: \Z^n \times \Omega \to \Omega$);
$$
    \Medint_{\R^n} f(x,\widetilde{\omega}) \ dx= 
   \mathbb{E} \left[\int_{[0,1)^n}  f(y,\cdot)\, dy\right].
$$

ii) Continuous case (i.e. $\tau: \R^n \times \Omega \to \Omega$);
$$
    \Medint_{\R^n} f(x,\widetilde{\omega}) \ dx=  \mathbb{E}\left[ f(0,\cdot) \right].
$$
\end{theorem}

\medskip
The following lemma shows that, the Birkhoff Ergodic Theorem holds if a stationary function is composed 
with a stochastic deformation. 
\begin{lemma}\label{phi2}
Let $\Phi$ be a stochastic deformation and $f \in L^{\infty}_\loc(\R^n; L^1(\Omega))$ be a stationary random variable in the 
sense~\eqref{Stationary}. Then for almost $\widetilde{\omega} \in \Omega$ the function 
$f\left(\Phi^{-1}(\cdot,\widetilde{\omega}),\widetilde{\omega}\right)$ possesses a mean value 
in the sense of~\eqref{MeanValue} and satisfies: 

\smallskip
i) Discrete case;
$$
\text{$\Medint_{\R^n}f\left(\Phi^{-1}(z,\widetilde{\omega}),\widetilde{\omega}\right)\,dz
= \frac{\mathbb{E}\left[\int_{\Phi([0,1)^n, \cdot)} f {\left( \Phi^{-1}\left( z, \cdot \right), \cdot \right)} \, dz \right]}
{\det\left(\mathbb{E}\left[\int_{[0,1)^n} \nabla_{\!\! y} \Phi(y,\cdot) \, dy \right]\right)}$
\quad for a.a. $\widetilde{\omega} \in \Omega$}.
$$

ii) Continuous case; 
$$
\text{$\Medint_{\R^n}f\left(\Phi^{-1}(z,\widetilde{\omega}),\widetilde{\omega}\right)\,dz
= \frac{\mathbb{E}\left[f(0,\cdot)\det\left(\nabla\Phi(0,\cdot)\right)\right]}
{\det\left(\mathbb{E}\left[\nabla \Phi(0,\cdot)\right]\right)}$
\qquad for a.a. $\widetilde{\omega} \in \Omega$}.
$$
\end{lemma}

\begin{proof}
See Blanc, Le Bris, Lions \cite{BlancLeBrisLions1}, also
Andrade, Neves, Silva \cite{AndradeNevesSilva}.
\end{proof}

\subsubsection{Analysis of stationary functions}

In the rest of this paper, unless otherwise explicitly stated, we
assume discrete dynamical systems and therefore, stationary 
functions are considered in this discrete sense.
We begin the analysis of stationary functions with the concept of realization.
\begin{definition}
Let $f: \mathbb{R}^n \! \times \! \Omega \to \mathbb{R}$ be a stationary function. 
For $\omega \in \Omega$ fixed, the function $f(\cdot, \omega)$ is called a realization
of $f$. 
\end{definition}
Due to Theorem \ref{Birkhoff}, almost every realization 
$f(\cdot,\omega)$ possesses a mean value in the sense of~\eqref{MeanValue}.
On the other hand, if $f$ is a stationary function, then the mapping 
$$
   y \in \mathbb{R}^n \mapsto \int_\Omega f(y, \omega) \, d\mathbb{P}(\omega) 
$$
is a $[0,1)^n-$periodic function. 

In fact, it is enough to consider the realizations to study some properties 
of stationary functions. For instance, the following theorem will be used 
more than once through this paper.  
%
\begin{theorem}
\label{987987789879879879}
For $p> 1$, let $u,v \in L^1_{\rm loc}(\mathbb{R}^n; L^p(\Omega))$
be stationary functions. Then, for any $i \in \{1,\ldots,n \}$ fixed, the following sentences are equivalent: 
\begin{equation}
\label{837648726963874}
(A) \quad  \int_{[0,1)^n} \int_\Omega u(y,\omega) \frac{\partial {\zeta}}{\partial y_i} (y, \omega) \, d\mathbb{P}(\omega) \, dy 
      = - \int_{[0,1)^n} \int_\Omega v(y,\omega) \, {\zeta}(y,\omega) \, d\mathbb{P} \, dy, \hspace{20pt}
\end{equation}
for each stationary function $\zeta \in C^1( \mathbb{R}^n; L^q(\Omega))$,
with $1/p + 1/q = 1$. 
\begin{equation}
\label{987978978956743}
(B) \quad  \int_{\mathbb{R}^n} u(y,\omega) \frac{\partial {\varphi}}{\partial y_i} (y) \, dy = - \int_{\mathbb{R}^n} v(y,\omega) \, {\varphi}(y) \, dy,
\hspace{87pt}
\end{equation}
for any $\varphi \in C^1_{\rm c}(\mathbb{R}^n)$, and almost sure $\omega \in \Omega$.
\end{theorem}

\begin{proof}
See Blanc, Le Bris, Lions \cite{BlancLeBrisLions1}. 
\end{proof} 

Similarly to the above theorem, we have the characterization of weak derivatives of stationary functions 
composed with stochastic deformations, given by the following 

\begin{theorem}
\label{648235azwsxqdgfd}
Let $u,v \in L^1_{\rm loc}(\mathbb{R}^n; L^p(\Omega))$
be stationary functions, $(p> 1)$. Then, for any $i \in \{1,\ldots,n \}$ fixed, the following sentences are equivalent: 
\begin{multline*}
(A) \quad \int_\Omega \int_{\Phi([0,1)^n, \omega)} u {\left( \Phi^{-1}\left( z, \omega \right), \omega \right)} \, { \frac{ \partial {\left( \zeta{\left( \Phi^{-1}(z,\omega),\omega \right)} \right)}}{\partial z_k} } \, dz \, d\mathbb{P}(\omega) 
\\[5pt]
= - \int_\Omega \int_{\Phi([0,1)^n, \omega)} v {\left( \Phi^{-1}\left( z, \omega \right), \omega \right)} \, {\zeta{\left( \Phi^{-1}(z,\omega),\omega \right)}} \, dz \, d\mathbb{P}(\omega),
\end{multline*}
for each stationary function $\zeta \in C^1( \mathbb{R}^n; L^q(\Omega))$,
with $1/p + 1/q = 1$. 
\begin{equation*}
(B) \quad \int_{\mathbb{R}^n} u {\left( \Phi^{-1}\left( z, \omega \right), \omega \right)} \, { \frac{\partial \varphi}{\partial z_k}(z) } \, dz 
= - \int_{\mathbb{R}^n} v {\left( \Phi^{-1}\left( z, \omega \right), \omega \right)} \, {\varphi(z)} \, dz, \hspace{12pt}
\end{equation*}
for any $\varphi \in C^1_{\rm c}(\mathbb{R}^n)$, and almost sure $\omega \in \Omega$.
\end{theorem}

\subsection{$\Phi_\omega-$Two-scale Convergence}
\label{pud63656bg254v2v5}

In this subsection, we shall consider the two-scale convergence in a stochastic setting that is beyond of the classical stationary 
ergodic setting. The classical concept of two-scale convergence was introduced by Nguetseng~\cite{Nguetseng} and futher developed by Allaire~\cite{Allaire} 
to deal with periodic problems. 

\medskip
The notion of two-scale convergence has been successfully extended to non-periodic settings in several papers as in~\cite{FridSilvaVersieux,DiazGayte} in the ergodic algebra 
setting and in~\cite{BourgeatMikelicWright} in the stochastic setting. The main difference here with the earlier studies is 
that the test functions used here are random perturbations accomplished by stochastic diffeomorphisms of stationary 
functions. The main difficulty brought by this kind of test function is the lack of the stationarity property (see \cite{AndradeNevesSilva} for a deep discussion about that) which makes us 
unable to use the results described in~\cite{BourgeatMikelicWright} and the lack of a compatible topology with the probability space considered. This is overcome by using a 
compactification argument that preserves the ergodic nature of the setting involved. For this, we will make use of the following lemma, whose simple proof can be found in~\cite{AF}.

\begin{lemma}\label{TopologicalLemma}
Let $X_1,X_2$ be compact spaces, $R_1$ a dense subset of $X_1$ and $W:R_1\to X_2$. Suppose that for all $g\in C(X_2)$ the function $g\circ W$ is the restriction 
to $R_1$ of some (unique) $g_1\in C(X_1)$. Then $W$ can be uniquely extended to a continuous mapping $\underline{W}:X_1\to X_2$. Further, suppose in addition that 
$R_2$ is a dense set of $X_2$, $W$ is a bijection from $R_1$ onto $R_2$ and for all $f\in C(X_1)$, $f\circ W^{-1}$ is the restriction to $R_2$ of some (unique) 
$f_2\in C(X_2)$. Then, $W$ can be uniquely extended to a homeomorphism $\underline{W}:X_1\to X_2$. 
\end{lemma}

Now, we can prove the following result.

\begin{theorem}
\label{Compacification}
Let $\mathbb {S}\subset L^{\infty}(\R^n\times \Omega)$ be a countable set of stationary functions. Then there exists a compact (separable) topological space 
$\widetilde{\Omega}$ and one-to-one function $\delta:\Omega\to \widetilde{\Omega}$ with dense image satisfying the following properties: 
\begin{enumerate}
\item[(i)] The probability space $\Big(\Omega,\mathscr{F},\mathbb{P}\Big)$ and the ergodic dynamical system $\tau:\mathbb{Z}^n\times \Omega\to\Omega$ acting on it 
extends respectively to a Radon probability space $\Big(\widetilde{\Omega},\mathscr{B},\widetilde{\mathbb{P}}\Big)$ and to an ergodic dynamical system 
$\widetilde{\tau}:\mathbb{Z}^n\times \widetilde{\Omega}\to\widetilde{\Omega}$.
\item[(ii)] The stochastic deformation $\Phi:\R^n\times\Omega\to\R^n$ extends to a stochastic deformation 
$\tilde{\Phi}:\R^n\times\widetilde{\Omega}\to\R^n$ satisfying 
$$
\Phi(x,\omega)=\tilde{\Phi}(x,\delta(\omega)),
$$
for a.e. $\omega\in\Omega$.
\item[(iii)] Any function $f\in\mathbb{S}$ extends to a $\tilde{\tau}-$stationary function $\tilde{f}\in L^{\infty}(\widetilde{\Omega}\times \R^n)$ satisfying 
$$
\Medint_{\R^n}f\left(\Phi^{-1}(z,\omega),\omega\right)\,dz=\Medint_{\R^n}\tilde{f}\left(\tilde{\Phi}^{-1}(z,\delta(\omega)),\delta(\omega)\right)\,dz,
$$
for a.e. $\omega\in\Omega$.
\end{enumerate}
\end{theorem}
\begin{proof}
1.  Let $\mathbb{S}$ be the set of the lemma.  Given $f\in \mathbb{S}$, define 
$$
f_j(y,\omega):=\int_{\R^n}f(y+x,\omega)\,\rho_j(x)\,dx,
$$
where $\rho_j$ is the classical approximation of the identity in $\R^n$. Note that for a.e. $y\in\R^n$, we have that $f_j(y,\cdot)\to f(y,\cdot)$ in $L^1(\Omega)$ as $j\to\infty$.  
Define $\mathcal{A}$ as the closed algebra with unity generated by the set 
$$
\Big\{ f_j(y,\cdot);\,j\ge 1,y\in\mathbb{Q}^n,f\in\mathbb{S}\Big\}\cap \Big\{\partial_j \Phi_i(y,\cdot);\, 1\le j,i\le n, y\in\mathbb{Q}^n\Big\}.
$$
Since $[-1,1]$ is a compact set, by the well known Tychonoff's Theorem, the set 
$$
[-1,1]^{\mathcal{A}}:=\Big\{\text{the functions $\gamma:\mathcal{A}\to[-1,1]$}\Big\}
$$
is a compact set in the product topology. Define $\delta:\Omega\to[-1,1]^{\mathcal{A}}$ by 
$$
\delta(\omega)(g):=\left\{\begin{array}{rc}
\frac{g(\omega)}{\|g{\|}_{\infty}},&\mbox{if}\quad g\neq 0,\\
0,&\mbox{if}\quad g=0.
\end{array}\right.
$$
We may assume that the algebra $\mathcal{A}$ distinguishes between points of $\Omega$, that is, given any two distinct points $\omega_1,\omega_2\in\Omega$, there exists 
$g\in\mathcal{A}$ such that $g(\omega_1)\neq g(\omega_2)$. In the case that it is not true we may replace $\Omega$ by its quotient by a trivial equivalence relation, in a standard 
way, and we proceed correspondingly with the $\sigma-$algebra $\mathscr{F}$ and with the probability measure $\mathbb{P}$. Thus, the function $\delta$ is one-to-one. Define 
$$
\widetilde{\Omega}:=\overline{\delta(\Omega)}.
$$
Then, we can see that the set $\Omega$ inherits all topological features of the compact space $\widetilde{\Omega}$ in a natural way which allows us to identify it homeomorphically with 
the image $\delta(\Omega)$.

2.  Define the mapping $i:\mathcal{A}\to C(\delta(\Omega))$ by 
$$
i(g)(\delta(\omega)):=g(\omega).
$$
We claim that there exists a continuous function $\tilde{g}:\widetilde{\Omega}\to \R$ such that 
$$
i(g)=\tilde{g}\,\text{on $\delta(\Omega)$}.
$$
In fact, take $g\in\mathcal{A}$ and $Y:=\overline{g(\Omega)}$. Define the function $f^{*}:C(Y)\to\mathcal{A}$ by 
$$
f^{*}(h):=h\circ g\,\text{(the algebra structure is used!)}
$$
Hence, we can define $f^{**}:[-1,1]^{\mathcal{A}}\to[-1,1]^{C(Y)}$ by 
$
f^{**}(h):=h\circ f^{*}.
$
Note that the function $f^{**}$ is a continuous function. In order to see that, we highlighted that it is known that a function $H$ from a topological space 
to a product space $\otimes_{\alpha\in \mathcal{I}}X_{\alpha}$ is continuous if and only if each component $\pi_{\alpha}\circ H:=H_{\alpha}$ is 
continuous. Hence, if $\alpha\in C(Y)$ then the projection function $f^{**}_{\alpha}$ must satisfy 
\begin{eqnarray*}
&&f^{**}_{\alpha}(h):=\left(\pi_{\alpha}\circ f^{**}\right)(h)=\pi_{\alpha}\circ\left(f^{**}(h)\right)=\pi_{\alpha}\left(h\circ f^{*}\right)\\
&&\qquad=\left(h\circ f^{*}\right)(\alpha)=h\left(f^{*}(\alpha)\right)=h\left(\alpha\circ g\right)=\pi_{\alpha\circ g}(h).
\end{eqnarray*}
Now, consider the function $\tilde{\delta}:Y\to [-1,1]^{C(Y)}$ given by 
$$
\tilde{\delta}(y)(h):=\left\{\begin{array}{rc}
\frac{h(y)}{\|h{\|}_{\infty}},&\mbox{if}\quad h\neq 0,\\
0,&\mbox{if}\quad h=0.
\end{array}\right.
$$
Since the algebra $C(Y)$ has the following property: If $F\subset Y$ is a closed set and $y\notin F$ then 
$f(y)\notin \overline{f(F)}$ for some $f\in C(Y)$, we can conclude that the function $\tilde{\delta}$ is a homeomorphism onto its image. Furthermore, 
given $\omega\in\Omega$ we have that $\left(f^{**}\circ \delta\right)(\omega)=f^{**}\left(\delta(\omega)\right)=\delta(\omega)\circ f^{*}$. Hence, if $0\neq h\in C(Y)$ 
it follows that 
\begin{eqnarray*}
&&\left(f^{**}\circ\delta\right)\big(\omega\big)(h)=\left(\delta(\omega)\circ f^{*}\right)(h)=\delta(\omega)\left(f^{*}(h)\right)\\
&&\quad=\delta(\omega)\left(h\circ g\right)=\frac{\left(h\circ g\right) (\omega)}{\|h\circ g {\|}_{\infty}}=\tilde{\delta}\left(g(\omega)\right)(h).
\end{eqnarray*}
Thus, we see that ${\tilde{\delta}}^{-1}\circ f^{**}=i(g)$. Defining $\tilde{g}:={\tilde{\delta}}^{-1}\circ f^{**}$ we have clearly that 
$i:\mathcal{A}\to C(\widetilde{\Omega})$ is a one-to-one isometry satisfying $i(g)=\tilde{g}$ and our claim is proved. Moreover, it is easy to see that $i(\mathcal{A})$ is 
an algebra of functions over $C(\widetilde{\Omega})$ containing the unity. As before, if $i(\mathcal{A})$ does not separate points of $\widetilde{\Omega}$, then we may replace 
$\widetilde{\Omega}$ by its quotient $\widetilde{\Omega}/\sim$, where 
$$
\tilde{\omega_1}\sim\tilde{\omega_2}\,\Leftrightarrow\, \tilde{g}(\tilde{\omega_1})=\tilde{g}(\tilde{\omega_2})\,\forall g\in\mathcal{A}.
$$
Therefore, we can assume that $i(\mathcal{A})$ separates the points of $\widetilde{\Omega}$. Hence by the Stone's Theorem 
we must have $i(\mathcal{A})=C(\widetilde{\Omega})$.

\medskip
3. Define $\widetilde{\tau}:\mathbb{Z}^n\times\delta(\Omega)\to \delta(\Omega)$ by
$
\widetilde{\tau}_k\left(\delta(\omega)\right):=\delta(\tau_k\omega).
$
It is easy to see that 
$$
\widetilde{\tau}_{k_1+k_2}(\delta(\omega))=\widetilde{\tau}_{k_1}\Big(\widetilde{\tau}_{k_2}(\delta(\omega))\Big),
$$
for all $k_1,k_2\in\mathbb{Z}^n$ and $\omega\in\Omega$. Since $\tilde{g}\circ{\widetilde{\tau}_k}=\widetilde{g\circ{\tau}_k}$ for all $g\in\mathcal{A}$ and 
$k\in\mathbb{Z}^n$, the lemma~\ref{TopologicalLemma} allows us to extend the mapping $\widetilde{\tau}_k$ from $\delta(\Omega)$ to $\widetilde{\Omega}$ satisfying the 
group property $\widetilde{\tau}_{k_1+k_2}=\widetilde{\tau}_{k_1}\circ{\widetilde{\tau}}_{k_2}$. Given a borelian set $\tilde{A}\subset{\widetilde{\Omega}}$ and defining 
$\widetilde{\mathbb{P}}(\tilde{A}):=\mathbb{P}(\delta^{-1}(\tilde{A}\cap \delta(\Omega)))$, we can deduce that 
$\widetilde{\mathbb{P}}\circ {\widetilde{\tau}_k}=\widetilde{\mathbb{P}}$. Thus, the mapping $\widetilde{\tau}_k$ is an ergodic dynamical system over the Radon probability 
space $\Big(\widetilde{\Omega},\mathscr{B},\widetilde{\mathbb{P}}\Big)$. Thus, we have concluded the proof of the item (i). 

4. Now, note that for each $\omega\in\widetilde{\Omega}$ and each integer $j\ge 1$, the function $f_j(\cdot,\omega)$ is uniformly continuous over $\mathbb{Q}^n$. Hence, 
it can be extended uniquely to a function $\widetilde{f_j}(\cdot,\omega)$ defined in $\R^n$ that satisfies   
$$
\limsup_{j,l\to\infty}\int_{[0,1)^n\times \widetilde{\Omega}}|\widetilde{f_j}(y,\omega)-\widetilde{f_l}(y,\omega)|\,d\widetilde{\mathbb{P}}(\omega)\,dy=0.
$$
Therefore, there exists a $\widetilde{\tau}$-stationary function $\widetilde{f}\in L^1_{\loc}\left(\R^n\times \widetilde{\Omega}\right)$, such that 
$\widetilde{f_j}\to \widetilde{f}$ as $j\to\infty$ in $L^1_{\loc}(\R^n\times \widetilde{\Omega})$. Since $\| \widetilde{f_j}{\|}_{\infty}\le \| {f}{\|}_{\infty}$, for all $j\ge 1$ we have 
that $\widetilde{f}\in L^{\infty}(\R^n\times \widetilde{\Omega})$. In the same way, the stochastic deformation $\Phi:\R^n\times \Omega\to \R^n$ extends to a stochastic 
deformation $\tilde{\Phi}:\R^n\times \widetilde{\Omega}\to \R^n$ satisfying $\Phi(y,\omega)=\tilde{\Phi}(y,\delta(\omega))$ for all $\omega\in\Omega$ and 
$$
\Medint_{\R^n}f\left(\Phi^{-1}(z,\omega),\omega\right)\,dz=\Medint_{\R^n}\tilde{f}\left(\tilde{\Phi}^{-1}(z,\delta(\omega)),\delta(\omega)\right)\,dz,
$$
for a.e. $\omega\in\Omega$. This, completes the proof of the theorem \eqref{Compacification}.
\end{proof}

In practice, in the present context, the set $\mathbb{S}$ shall be a countable set generated by the coefficients of our equation $\eqref{jhjkhkjhkj765675233}$  and by the 
eigenfunctions of the spectral equation associated to it. Thus, the Theorem \eqref{Compacification} allow us to suppose without loss of generality that our probability 
space $\Big(\Omega,\mathscr{F},\mathbb{P}\Big)$ is a separable compact space.  Using the Ergodic Theorem, given a stationary function $f\in L^{\infty}(\R^n\times\Omega)$ 
there exists a set of full measure $\Omega_f\subset \Omega$ such that 
\begin{equation}\label{Compacification1}
\Medint_{\R^n}f\left(\Phi^{-1}(z,\tilde{\omega}),\tilde{\omega}\right)\,dz= c_{\Phi}^{-1}\int_{\Omega}\int_{\Phi([0,1)^n,\omega)}f\left(\Phi^{-1}(z,\omega),\omega\right)\,dz\,d\mathbb{P}(\omega),
\end{equation}
for almost all $\tilde{\omega}\in{\Omega}_f$. Due to the separability of the probability compact space $\Big(\Omega,\mathscr{F},\mathbb{P}\Big)$, we can find a set 
$\mathbb{D}\subset C_b(\R^n\times\Omega)$ such that:
\begin{itemize}
\item Each $f\in\mathbb{D}$ is a stationary function. 
\item $\mathbb{D}$ is a countable and dense set in $C_0\big([0,1)^n\times\Omega\big)$.
\end{itemize}
In this case, there exists a set $\Omega_0\subset\Omega$ of full measure such that the equality \eqref{Compacification1} holds for any $\tilde{\omega}\in\Omega_0$ and 
$f\in\mathbb{D}$.

\medskip
Now, we proceed with the definition of the two-scale convergence in this scenario of stochastically deformed. In what follows, the set $O\subset\R^n$ is an open set. 
\begin{definition}
\label{two-scale}
Let $1< p <\infty$ and $v_{\ve}:O\times\Omega\to \mathbb{C}$ be a sequence such that $v_{\ve}(\cdot,\tilde{\omega})\in L^p(O)$. 
The sequence 
$\{v_{\ve}(\cdot,\tilde{\omega}){\}}_{\ve>0}$ is said to $\Phi_\omega-$two-scale converges to a stationary function $V_{\tilde{\omega}}\in L^p\left(O\times [0,1)^n\times\Omega\right)$,
when for a.e. $\tilde{\omega}\in\Omega$ holds the following
$$
\begin{aligned}
&\lim_{\ve\to 0}\int_{O}v_{\ve}(x,\tilde{\omega})\,\varphi(x)\,\Theta\left(\Phi^{-1}\left(\frac{x}{\ve},\tilde{\omega}\right),\tilde{\omega}\right)\,dx
\\
&= c_{\Phi}^{-1}\int_{O\times\Omega}\int_{\Phi([0,1)^n,\omega)}  \!\!\! V_{\tilde{\omega}}\left(x,\Phi^{-1}\left(z,\omega\right),\omega\right)\,\varphi(x)\,
\Theta(\Phi^{-1}(z,\omega),\omega)\,dz\,d{\mathbb{P}(\omega)}\,dx,
\end{aligned}
$$
for all $\varphi\in C_c^{\infty}(O)$ and $\Theta\in L^{q}_{\loc}(\R^n\times\Omega)$ stationary. Here, $p^{-1}+q^{-1}=1$ and 
$c_{\Phi}:=\det\Big(\int_{[0,1)^n\times \Omega}\nabla \Phi(y,\omega)\,d{\mathbb{P}(\omega)}\,dy\Big)$.
\end{definition}
\begin{remark}
From now on, we shall use the notation 
\begin{equation*}
			v_{\varepsilon}(x,\widetilde{\omega}) \; \xrightharpoonup[\varepsilon \to 0]{2-{\rm s}}\; V_{\widetilde{\omega}} {\left(x,\Phi^{-1}(z,\omega),\omega \right)},
		\end{equation*}
to indicate that $v_{\varepsilon}(\cdot,\tilde{\omega})$ $\Phi_\omega-$two-scale converges to $V_{\tilde{\omega}}$.
\end{remark}
The most important result about the two-scale convergence needed in this paper is the following compactness theorem which 
generalize the corresponding one for the 
deterministic case in~\cite{DiazGayte} (see Theorem 4.8)
and the corresponding one for the stochastic case in~\cite{BourgeatMikelicWright} (see Theorem 3.4). 
\begin{theorem}
\label{TwoScale}
Let $1< p <\infty$ and $v_{\ve}:O\times\Omega\to \mathbb{C}$ be a sequence such that
$$
\sup_{\ve>0}\int_{O}|v_{\ve}(x,\tilde{\omega})|^p\,dx<\infty,
$$
for almost all $\tilde{\omega}\in\Omega$. 
Then, for almost all $\tilde{\omega}\in\Omega_0$, there exists a subsequence $\{v_{\ve'}(\cdot,\tilde{\omega}){\}}_{\ve'>0}$, which may depend on $\tilde{\omega}$, and a 
stationary function $V_{\tilde{\omega}}\in L^p(O\times [0,1)^n\times\Omega)$, such that 
$$
v_{\varepsilon}(x,\widetilde{\omega}) \; \xrightharpoonup[\varepsilon' \to 0]{2-{\rm s}}\; V_{\widetilde{\omega}} {\left(x,\Phi^{-1}(z,\omega),\omega \right)}.
$$
\end{theorem}
\begin{proof}
1. We begin by fixing $\tilde{\omega}\in\Omega_0$. Due to our assumption, there exists ${ c(\widetilde{\omega})>0 }$, such that 
for all $\varepsilon >0$
		\begin{equation*}
			{\left\Vert v_\varepsilon(\cdot,\widetilde{\omega}) \right\Vert}_{L^p(O)} \leqslant c(\widetilde{\omega}).
		\end{equation*}
Now, taking $\phi\in \Xi\times \mathbb{D}$ with $\Xi\subset C_c^{\infty}(O)$ dense in $L^q(O)$, we have after applying the H\"older inequality and the Ergodic Theorem, 
\begin{multline}
\label{6742938tyuer}
			\underset{\varepsilon \to 0}{\limsup} {\vert \int_{O} v_\varepsilon(x,\widetilde{\omega}) \, { \phi{\left( x,\Phi^{-1}{\left( \frac{x}{\varepsilon}, \widetilde{\omega} \right)}, \widetilde{\omega} \right)} } dx \vert} 
			\\
		       \leq c(\widetilde{\omega}) {\left[ \underset{\varepsilon \to 0}{\limsup}\int_{O} {\left\vert \phi{\left( x,\Phi^{-1}{\left( \frac{x}{\varepsilon}, \widetilde{\omega} \right)}, \widetilde{\omega} \right)} \right\vert}^q dx \right]}^{1/q} \hspace{100pt} 
		       \\
			= c(\widetilde{\omega}) {\left[ c_\Phi^{-1} \int_{O} \int_\Omega \int_{\Phi([0,1)^n,\omega)} {\left\vert \phi{\left( x,\Phi^{-1}(z,\omega),\omega \right)} \right\vert}^q dz \, d\mathbb{P}(\omega) \, dx \right]}^{1/q}.
		\end{multline}
Thus, the use of the enumerability of the set $\Xi\times \mathbb{D}$ combined with a diagonal argument yields us a subsequence $\{\ve'\}$ (maybe depending of $\tilde{\omega}$) 
such that the functional $\mu:\Xi\times \mathbb{D}\to \mathbb{C}$ given by 
\begin{equation}\label{Two-scale1}
		\langle\mu,\phi\rangle:=	\lim_{\varepsilon^{\prime} \to 0}\int_{O} v_{\varepsilon^{\prime}}(x,\widetilde{\omega}) \, { \phi{\left( x, \Phi^{-1}{\left( \frac{x}{\varepsilon^{\prime}}, \widetilde{\omega} \right)}, \widetilde{\omega} \right)} } dx
		\end{equation}
is well-defined and bounded with respect to the norm $\|\cdot{\|}_q$ defined as 
$$
\|\phi{\|}_q:= \big[ c_\Phi^{-1} \int_{O} \int_\Omega \int_{\Phi([0,1)^n,\omega)} {\left\vert \phi{\left( x,\Phi^{-1}(z,\omega),\omega \right)} \right\vert}^q dz \, d\mathbb{P}(\omega) \, dx \big]^{1/q}
$$
by \eqref{6742938tyuer}. Since the set $\Xi\times \mathbb{D}$ is dense in $L^q\left(O\times[0,1)^n\times\Omega\right)$, we can extend the functional $\mu$ to a bounded 
functional $\tilde{\mu}$ over $L^q\left(O\times[0,1)^n\times\Omega\right)$. Hence, we find $V_{\tilde{\omega}}\in L^p(O\times [0,1)^n\times\Omega)$ which can be extended 
to $O\times\R^n\times\Omega$ in a stationary way by setting 
$$
V_{\tilde{\omega}}(x,y,\omega)=V_{\tilde{\omega}}\left(x,y-\left\lfloor y \right\rfloor, \tau_{\left\lfloor y \right\rfloor}\omega\right),
$$
and satisfying for all $\phi\in L^q\left(O\times[0,1)^n\times\Omega\right)$, 
$$
\langle \tilde{\mu},\phi\rangle\!\!= c_{\Phi}^{-1}\int_{O\times\Omega}\int_{\Phi([0,1)^n,\omega)}
\!\!\! \!V_{\tilde{\omega}}\left(x,\Phi^{-1}\left(z,\omega\right),\omega\right)
\phi\left(x,\Phi^{-1}(z,\omega),\omega\right)dz d\mathbb{P}(\omega) dx.
$$

2. Now, take $\varphi\in C^{\infty}_c(O)$ and $\Theta\in L^{q}_{\loc}(\R^n\times\Omega)$ a $\tau$-stationary function. Since the set 
$\Xi\times \mathbb{D}$ is dense in $L^q\left(O\times[0,1)^n\times\Omega\right)$, we can pick up a sequence 
$\{(\varphi_j,\Theta_j){\}}_{j\ge1}\subset \Xi\times \mathbb{D}$ such that 
$$
\lim_{j\to\infty}(\varphi_j,\Theta_j)=(\varphi,\Theta)\quad\text{in $L^q\Big(O\times[0,1)^n\times\Omega\Big)$}.
$$
Then, observing that 
\begin{eqnarray*}
&&\limsup_{\ve'\to 0}\Big| \int_{O}v_{\ve'}(x,\tilde{\omega})\varphi(x)\Theta\left(\Phi^{-1}\left(\frac{x}{\ve'},\tilde{\omega}\right),\tilde{\omega}\right)\,dx
\\
&&-\int_{O}v_{\ve'}(x,\tilde{\omega})\varphi_j(x)\Theta_j\left(\Phi^{-1}\left(\frac{x}{\ve'},\tilde{\omega}\right),\tilde{\omega}\right)\,dx\Big|
\\
&& \le c \|\varphi-\varphi_j{\|}_{L^q(O)}
[ c_\Phi^{-1} \int_\Omega \int_{\Phi([0,1)^n,\omega)} {\left\vert (\Theta-\Theta_j){\left(\Phi^{-1}(z,\omega),\omega \right)} \right\vert}^q dz \, d\mathbb{P}(\omega)]^{1/q},
\end{eqnarray*}
where $c= c(\tilde{\omega})$ is a positive constant. 
Then, combining the previous equality with the \eqref{Two-scale1}, we concluded the proof of the theorem.
\end{proof}

Let us remember the following space (see Remark \ref{REMFPHI}) 
$$
\mathcal{H}:=\Big\{w\in H^1_{\loc}(\R^n;L^2(\Omega));\,\text{$w$ is a stationary function}\Big\},
$$
which is a Hilbert space with respect to the following inner product 
$$
\begin{aligned}
\langle w,v{\rangle}_{\mathcal{H}}:=\int_{[0,1)^n\times\Omega} \!\! & \nabla_{\!y} w(y,\omega)\cdot \nabla_y v(y,\omega)\,d{\mathbb{P}}(\omega)\,dy
\\
& +\int_{[0,1)^n\times\Omega}. \!\! \!\! w(y,\omega) v(y,\omega)\,d{\mathbb{P}}(\omega)\,dy.
\end{aligned} 
$$
The next lemma will be important in the homogenization's process. 
\begin{lemma}\label{SYM1-5}
Let $O\subset\R^n$ be an open set and assume that $\{u_{\ve}(\cdot,\tilde{\omega}){\}}_{\ve>0}$ and $\{\ve \nabla u_{\ve}(\cdot,\tilde{\omega}){\}}_{\ve>0}$ are 
bounded sequences in $L^2(O)$ and in $L^2(O;\R^n)$ respectively for a.e. $\tilde{\omega}\in\Omega$. Then, for a.e. $\tilde{\omega}\in\Omega$, there exists a 
subsequence $\{\ve'\}$(it may depend on $\tilde{\omega}$) and $u_{\tilde{\omega}}\in L^2(O;\mathcal{H})$, such that 
$$
u_{\ve'}(\cdot,\tilde{\omega})\; \xrightharpoonup[\varepsilon \to 0]{2-{\rm s}}\; u_{\tilde{\omega}},
$$
and 
$$
\ve'\nabla u_{\ve'}(\cdot,\tilde{\omega})\; \xrightharpoonup[\varepsilon \to 0]{2-{\rm s}}\;[\nabla_y\Phi]^{-1}\nabla_y u_{\tilde{\omega}}.
$$
\end{lemma}
\begin{proof}
  Applying the Theorem \ref{TwoScale} for the sequences 
  $${ \{u_\varepsilon(\cdot,\widetilde{\omega})\}_{\varepsilon > 0} }, \quad
{ \{\varepsilon \nabla u_\varepsilon(\cdot,\widetilde{\omega})\}_{\varepsilon > 0} }$$ 
for a.e.  ${ \widetilde{\omega} \in \Omega }$, 
we can find a subsequence $\{ \varepsilon^\prime \}$, and functions 
$${ u_{\widetilde{\omega}} \in  L^2({O} \! \times \! [0,1)^n \! \times \! \Omega) }, \quad
{ V_{\widetilde{\omega}} \in L^2({O} \! \times [0,1)^n \! \times \! \Omega;\R^n)}$$ 
with ${ V_{\widetilde{\omega}} = (v^{(1)}_{\widetilde{\omega}}, \ldots, v^{(n)}_{\widetilde{\omega}}) }$
satisfying for ${ k \in \{1,2, \ldots, n\} }$, 
\begin{equation}\label{9867986876410}
			u_{\varepsilon^\prime}(\cdot,\widetilde{\omega}) \; \xrightharpoonup[\varepsilon^\prime \to 0]{2-{\rm s}} \; u_{\widetilde{\omega}},
		\end{equation}
		and
		\begin{equation}\label{7869876874}
			\varepsilon^\prime \frac{\partial u_{\varepsilon^\prime}}{\partial x_k} \; \xrightharpoonup[\varepsilon^\prime \to 0]{2-{\rm s}} \; v^{(k)}_{\widetilde{\omega}}.
		\end{equation}		
Hence for each ${ k\in \{1,\ldots,n\} }$ and performing an integration by parts we have 
		\begin{multline*}
			\int_{{O}} \varepsilon^\prime \frac{\partial u_{\varepsilon^\prime}}{\partial x_k} (x,\widetilde{\omega}) \, {\varphi(x) \, \Theta{\left( \Phi^{-1}\left( \frac{x}{\varepsilon^\prime}, \widetilde{\omega} \right), \widetilde{\omega} \right)}} dx
			\\
			\hspace{-4cm}=  -\varepsilon^\prime\int_{{O}} u_{\varepsilon^\prime} (x,\widetilde{\omega}) \, {\frac{\partial \varphi}{\partial x_k}(x) \, \Theta {\left( \Phi^{-1}\left( \frac{x}{\varepsilon^\prime}, \widetilde{\omega} \right), \widetilde{\omega} \right)}} dx 
			\\
			\quad \quad -\int_{{O}} u_{\varepsilon^\prime} (x,\widetilde{\omega}) \, {\varphi(x) \, {[\nabla_{\!\! y}\Phi]}^{-1}  {\left( \Phi^{-1}{\left( \frac{x}{\varepsilon^\prime}, \widetilde{\omega} \right)}, \widetilde{\omega} \right)}  \, \nabla_{\!\! y} \Theta{\left( \Phi^{-1}\left( \frac{x}{\varepsilon^\prime}, \widetilde{\omega} \right), \widetilde{\omega} \right)} \cdotp e_k} \, dx,
		\end{multline*}
for every $\varphi\in C^{\infty}_c(O)$ and $\Theta \in C_c^{\infty}\big([0,1)^n; L^{\infty}(\Omega)\big)$ extended in a stationary way to $\R^n$.  Then, using the relations \eqref{9867986876410}-\eqref{7869876874} and a density argument in the space of the test functions, we arrive after letting $\ve'\to 0$

\begin{multline*}
			\int_\Omega \int_{\Phi([0,1)^n, \omega)} v^{(k)}_{\widetilde{\omega}} {\left( x, \Phi^{-1}\left( z,\omega \right),\omega \right)} \, { \Theta{\left( \Phi^{-1}(z,\omega),\omega \right)} } \, d\mathbb{P} \, dz
			\\
			= - \int_\Omega \int_{\Phi([0,1)^n, \omega)} u_{\widetilde{\omega}} \left( x, \Phi^{-1}\left( z,\omega \right),\omega \right) { \frac{\partial }{\partial z_k} {\left( \Theta{\left( \Phi^{-1}(z,\omega),\omega \right)} \right)} } \,  d\mathbb{P} \, dz ,
		\end{multline*}
for a.e. $x \in {O} $ and for any $\Theta \in C_c^{\infty}\big([0,1)^n; L^{\infty}(\Omega)\big)$.

Hence­ applying Theorem \ref{648235azwsxqdgfd}, we obtain 
\begin{equation*}
	\int_{\mathbb{R}^n} v^{(k)}_{\widetilde{\omega}} {\left( x, \Phi^{-1}\left( z,\omega \right),\omega \right)} 
	\, { \varphi(z) } \, dz \, = \, -\int_{\mathbb{R}^n} u_{\widetilde{\omega}} 
	{\left( x, \Phi^{-1}\left( z,\omega \right),\omega \right)} \, { \frac{\partial \varphi}{\partial z_k}(z) } \, dz ,
\end{equation*}
for all ${ \varphi \in C_{\rm c}^\infty(\mathbb{R}^n) }$ and a.e.  ${ \omega \in \Omega }$. This completes the proof of our lemma. 
\end{proof}

	
\subsection{Bloch waves analysis. The WKB method}
\label{877853467yd56rtfe5rtfgeds76ytged}

Here we formally obtain the Bloch spectral cell equation
(stochastic setting).
To this end, we apply the 
asymptotic  Wentzel-Kramers-Brillouin (WKB for short) expansion method, that is, 
we assume that the solution of equation \eqref{jhjkhkjhkj765675233} is given by a
plane wave, see Definition \ref{92347828454trfhfd4rfghjls}.
More precisely, for each $\varepsilon> 0$ let us assume that, the solution $u_\varepsilon(t,x,\omega)$
of the equation \eqref{jhjkhkjhkj765675233} has the 
following asymptotic expansion 
\begin{equation}
\label{ansatz}
   u_{\varepsilon}(t,x,\omega)= e^{2\pi i S_{\varepsilon}(t,x)} \sum_{k=1}^{\infty}\varepsilon^k 
   u_k\Big(t,x,\Phi^{-1}\left(\frac{x}{\varepsilon},\omega\right),\omega\Big),
\end{equation}
where the functions $u_k(t,x,y,\omega)$ are conveniently stationary in $y$, and $S_{\varepsilon}$ 
is a real valued function to be established a posteriori
(not necessarily a polynomial in $\varepsilon$), which take part of the modulated plane wave \eqref{ansatz} 
from $e^{2\pi i S_{\varepsilon}(t,x)}$. 

The spatial derivative of the above ansatz \eqref{ansatz} is  
$$
\begin{aligned}
\nabla u_\varepsilon&(t,x,\omega)= e^{2i\pi S_\varepsilon(t,x)} \big(2i\pi \nabla S_\varepsilon(t,x)\sum_{k=0}^\infty \varepsilon^k\,
u_k \big(t,x,\Phi^{-1}\left(\frac{x}{\varepsilon},\omega\right),\omega\big)
\\[5pt]
&\qquad + \sum_{k=0}^\infty \varepsilon^k\Big\{\left(\partial_x u_k\right) \Big(t,x,\Phi^{-1}\left(\frac{x}{\varepsilon},\omega\right),\omega\Big)
\\[5pt]
&\qquad + \frac{1}{\varepsilon}(\nabla\Phi)^{-1}\left(\Phi^{-1}\left(\frac{x}{\varepsilon},\omega\right),\omega\right)
\left(\partial_y u_k\right)\Big(t,x,\Phi^{-1}\left(\frac{x}{\varepsilon},\omega\right),\omega\Big)\Big\}\Big)
\\[5pt]
&=e^{2i\pi S_\varepsilon(t,x)} \Big(\sum_{k=0}^\infty \varepsilon^k\left(\frac{{\nabla}_z}{\varepsilon} + 2i\pi \nabla S_\varepsilon(t,x)\right)
u_k\Big(t,x,\Phi^{-1}(\frac{x}{\varepsilon},\omega),\omega\Big)
\\[5pt]
&\qquad +\sum_{k=0}^\infty \varepsilon^k\left(\nabla_x u_k   \right)
\Big(t,x,\Phi^{-1}\left(\frac{x}{\varepsilon},\omega\right),\omega\Big)\Big).
\end{aligned}
$$

Now, computing the second derivatives of the expansion~\eqref{ansatz} and writing as a cascade of the power of $\varepsilon$, we have 
\begin{equation}
\label{ansatz2}
\begin{aligned}
e^{-2i\pi S_\varepsilon(t,x)} & {\rm div}{\big( A {( \Phi^{-1}{\left( \frac{x}{\varepsilon},\omega \right)},\omega)} 
\nabla u_\varepsilon(t,x,\omega) \big)}
\\
&=\frac{1}{\varepsilon^2}\Big( {\rm div}_{\! z} + 2i\pi \varepsilon \nabla S_\varepsilon(t,x) \Big)
\Big( A{\left( \Phi^{-1}(\cdot,\omega),\omega \right)} \Big( \nabla_{\!\! z} + 2i\pi \varepsilon \nabla S_\varepsilon(t,x) \Big)
\\
&\qquad\qquad\qquad\qquad u_0 ( t,x,\Phi^{-1}(\cdot,\omega),\omega)\Big){\Bigg\rvert}_{z=x/\varepsilon}
\\
&+ \frac{1}{\varepsilon}\Big( {\rm div}_{\! z} + 2i\pi \varepsilon \nabla S_\varepsilon(t,x) \Big)
\Big( A{\left( \Phi^{-1}(\cdot,\omega),\omega \right)} \Big( \nabla_{\!\! z} + 2i\pi \varepsilon \nabla S_\varepsilon(t,x) \Big)
\\
&\qquad\qquad\qquad\qquad u_1 ( t,x,\Phi^{-1}(\cdot,\omega),\omega)\Big){\Big\rvert}_{z=x/\varepsilon}
+ I_\varepsilon, 
\end{aligned}
\end{equation}
where 
\begin{eqnarray}
&& I_\varepsilon= \sum_{k=0}^\infty \varepsilon^k\Big( {\rm div}_{\! z} + 2i\pi \varepsilon \nabla S_\varepsilon(t,x) \Big)
\Big( A{\left( \Phi^{-1}(\cdot,\omega),\omega \right)} \Big( \nabla_{\!\! z} + 2i\pi \varepsilon \nabla S_\varepsilon(t,x) \Big)\nonumber
\\
&&\qquad\qquad u_{k+2} ( t,x,\Phi^{-1}(\cdot,\omega),\omega)\Big){\Big\rvert}_{z=x/\varepsilon}\nonumber
\\
&&+\frac{1}{\varepsilon}\Big( {\rm div}_{\! z} + 2i\pi \varepsilon \nabla S_\varepsilon(t,x) \Big)
\Big( A{\left( \Phi^{-1}(\cdot,\omega),\omega \right)}\nabla_x u_0 ( t,x,\Phi^{-1}(\cdot,\omega),\omega)\Big){\Big\rvert}_{z=x/\varepsilon}\nonumber
\\
&&+\sum_{k=0}^\infty \varepsilon^k\Big( {\rm div}_{\! z} + 2i\pi \varepsilon \nabla S_\varepsilon(t,x) \Big)
\Big( A{\left( \Phi^{-1}(\cdot,\omega),\omega \right)}\nabla_x u_{k+1} ( t,x,\Phi^{-1}(\cdot,\omega),\omega)\Big){\Big\rvert}_{z=x/\varepsilon}\nonumber
\\
&&+ \frac{1}{\varepsilon}{\rm div}_{\! x}\Big(A{\left( \Phi^{-1}(\cdot,\omega),\omega \right)}\Big( \nabla_{\!\! z} + 2i\pi \varepsilon \nabla S_\varepsilon(t,x) \Big)
u_0 ( t,x,\Phi^{-1}(\cdot,\omega),\omega)\Big){\Big\rvert}_{z=x/\varepsilon}\nonumber
\\
&&+ \sum_{k=0}^\infty \varepsilon^k {\rm div}_{\! x}\Big(A{\left( \Phi^{-1}(\cdot,\omega),\omega \right)}\Big( \nabla_{\!\! z} + 2i\pi \varepsilon \nabla S_\varepsilon(t,x) \Big)
u_{k+1} ( t,x,\Phi^{-1}(\cdot,\omega),\omega)\Big){\Big\rvert}_{z=x/\varepsilon}\nonumber
\\
&&+\sum_{k=0}^\infty \varepsilon^k{\rm div}_{\! x}\Big(A{\left( \Phi^{-1}(\cdot,\omega),\omega \right)}\nabla_{\!\! x} 
u_k{\Big( t,x,\Phi^{-1}(\cdot,\omega),\omega\Big)}\Big){\Big\rvert}_{z=x/\varepsilon}.
\end{eqnarray}

Proceeding in the same way with respect to the temporal derivative, we have
\begin{eqnarray}\label{ansatz3}
&&e^{-2i\pi S_\varepsilon(t,x)}\,{\partial}_t u_\varepsilon\nonumber
\\
&&\qquad\qquad=
\frac{1}{\varepsilon^2} \Big(2i\pi \varepsilon^2 {\partial}_{t} S_\varepsilon(t,x) \Big) u_0\Big( t,x,\Phi^{-1}\left(\frac{x}{\varepsilon},\omega\right),\omega\Big)\nonumber
\\
&&\qquad\qquad+\frac{1}{\varepsilon} \Big(2i\pi \varepsilon^2 {\partial}_{t} S_\varepsilon(t,x) \Big) 
u_1\Big( t,x,\Phi^{-1}\left(\frac{x}{\varepsilon},\omega\right),\omega\Big)\nonumber
\\
&&\qquad\qquad+\Big(2i\pi \varepsilon^2 {\partial}_{t} S_\varepsilon(t,x) \Big) \sum_{k=0}^\infty \varepsilon^k
u_{k+2}\Big( t,x,\Phi^{-1}\left(\frac{x}{\varepsilon},\omega\right),\omega\Big)\nonumber
\\
&&\qquad\qquad\qquad\qquad+\sum_{k=0}^\infty \varepsilon^k {\partial}_tu_k\Big( t,x,\Phi^{-1}\left(\frac{x}{\varepsilon},\omega\right),\omega\Big).
\end{eqnarray}
Thus, if we insert the equations \eqref{ansatz2} and \eqref{ansatz3} in \eqref{jhjkhkjhkj765675233} 
and compute the $\varepsilon^{-2}$ order term, we 
arrive at 
$$
  L^\Phi(\varepsilon \nabla S_\varepsilon(t,x)) u_0 {\big( t,x,\Phi^{-1}(\cdot,\omega),\omega\big)}
  = 2 \pi \big( \varepsilon^2 \partial_t S_\varepsilon(t,x)\big)u_0 {\big( t,x,\Phi^{-1}(\cdot,\omega),\omega\big)},
$$
where for each $\theta \in \R^n$, the linear operator $L^\Phi(\theta)$ is defined by
\begin{equation}
\label{EqEsp}
\begin{aligned}
L^\Phi(\theta)[\cdot]:=& -\big( {\rm div}_{\! z} + 2i\pi \theta \big)
\big(A{( \Phi^{-1}(z,\omega),\omega)} {\big( \nabla_{\!\! z} + 2i\pi \theta \big)} [\cdot] \big)
\\ 
&+V {\big(\Phi^{-1}\left(z,\omega\right),\omega\big)} [\cdot].
\end{aligned}
\end{equation}
Therefore, $2 \pi \Big( \varepsilon^2 \partial_t S_\varepsilon(t,x)\Big)$ is an eigenvalue of $ L^\Phi(\varepsilon \nabla S_\varepsilon(t,x))$.
Consequently, if $\lambda(\theta)$ is any eigenvalue of $L^\Phi(\theta)$ (which is sufficiently regular with respect to $\theta$), then
the following (eikonal) Hamilton-Jacobi equation must be satisfied
$$
   2 \pi \varepsilon^2 \partial_t S_\varepsilon(t,x) - \lambda(\varepsilon \nabla S_\varepsilon(t,x))= 0. 
$$
Thus, if we suppose for $t=0$ (companion to \eqref{ansatz}) the modulated plane wave initial data 
\begin{equation}
\label{ansatzID}
   u_{\varepsilon}(0,x,\omega)= e^{2i\pi \frac{\theta \cdot x}{\varepsilon}} \sum_{k=1}^{\infty}\varepsilon^k 
   u_k\Big(0,x,\Phi^{-1}\left(\frac{x}{\varepsilon},\omega\right),\omega\Big),
\end{equation}
then the unique solution for the above Hamilton-Jacobi equation is, for each parameter $\theta \in \R^n$, 
\begin{equation}
\label{SEP}
   S_\varepsilon(t,x)= \frac{\lambda(\theta) \  t}{2 \pi \varepsilon^2} + \frac{\theta \cdot x}{\varepsilon}.
\end{equation}

To sum up, the above expansion, that is
the solution $u_{\varepsilon}$ of the equation \eqref{jhjkhkjhkj765675233} 
with initial data given respectively by \eqref{ansatz} and \eqref{ansatzID}, 
suggests the following 

\begin{definition}[Bloch or
shifted spectral cell equation] 
\label{DEFBLOCHCELL} Let $\Phi$ be a stochastic deformation. 
For any $\theta \in \R^n$ fixed, the following time independent asymptotic equation 
\begin{equation}
\label{92347828454trfhfd4rfghjls}
\left\{
\begin{array}{l}
L^\Phi(\theta) [\Psi(z,\omega)]= \lambda \ \Psi(z,\omega), \hspace{40pt} \text{in $\R^n \times \Omega$}, 
\\[5pt]
\hspace{32pt} \Psi(z, \omega) = \psi {\left( \Phi^{-1} (z, \omega), \omega \right)}, \quad \text{$\psi$ is a stationary function},
			\end{array}
			\right.
		\end{equation}
is called Bloch's spectral cell equation companion to the Schr\"odinger equation in \eqref{jhjkhkjhkj765675233},
where $L^\Phi(\theta)$ is given by \eqref{EqEsp}.
Moreover, each $\theta \in \R^n$ is called a Bloch frequency, $\lambda(\theta)$ is called a Bloch energy and the corresponded 
$\Psi(\theta)$ is called a Bloch wave. If $\Phi$ is well understood in the context, then $L \equiv L^\Phi$. 
\end{definition}

\medskip
The unknown $(\lambda,\Psi)$ in \eqref{92347828454trfhfd4rfghjls}, which is an eigenvalue-eigenfunction pair, is obtained 
by the associated variational formulation, that is 
\begin{equation}
\label{FORMVARIAC}
\begin{aligned}
&\langle L(\theta)[F], G\rangle
\\
&= \int_\Omega \int_{\Phi([0,1)^n,\omega)} \!\!\!\!\!\!\!\!\! A( \Phi^{-1}(z, \omega), \omega) (\nabla_{\!\! z} + 2i\pi\theta)  F(z,\omega) \cdot
 \overline{{( \nabla_{\!\! z} + 2i\pi\theta)} G(z,\omega)} \, dz \, d\mathbb{P}(\omega) 
 \\[5pt]
 &+ \int_\Omega \int_{\Phi([0,1)^n,\omega)} V{( \Phi^{-1}(z, \omega), \omega)} \ F(z,\omega) \, 
 \overline{G(z,\omega)} \, dz \, d\mathbb{P}(\omega). 
\end{aligned}
\end{equation}
Moreover, it is required that for some $\theta^\ast \in \R^n$, $\lambda(\theta^\ast) \in \mathbb{R}$ 
satisfies 
 \begin{equation}
 \label{conds}
	\left\{ \,
	\begin{split}
		& \lambda(\theta^\ast) \; \text{is a simple eigenvalue of} \; L^\Phi(\theta^\ast), \\
		&\theta^\ast \; \text{is a critical point of} \; \lambda(\cdot), \, \text{that is}, \nabla_{\!\! \theta} \lambda(\theta^\ast) = 0. 
	\end{split}
	\right.
\end{equation}

It is not clear the existence of a pair $(\theta^{\ast},\lambda(\theta^{\ast}))$,
in general stochastic environments, satisfying the two above 
conditions. 
However, there are concrete situations  
in the periodic settings where such conditions take place (see for 
instance \cite{AllairePiatnitski,BarlettiBenAbdallah,BensoussanLionsPapanicolaou}).  
Section \ref{6775765ff0090sds} presents realistic stochastic models, 
whose spectral nature is inherited from the periodic ones and may satisfy \eqref{conds}. 

\subsection{Perturbations of bounded operators} 
\label{0239786gfhgdf}

To finish this preliminar section, we consider
the perturbation theory for bounded operators with 
isolated eigenvalues of finite multiplicity.
The precisely result is stated in the following theorem, 
(see for instance Kato \cite{Kato}, Rellich \cite{Rellich}). 
\begin{theorem}
\label{768746hughjg576}
Let $H$ be a Hilbert space, and a sequence 
of operators $\{A_\alpha\}$,
$A_\alpha \in \mathcal{B}(H)$ for each $\alpha \in \mathbb{N}^n$.
Consider the power series of $n-$complex variables $\boldsymbol{z}= (z_1, \ldots, z_n)$
with coefficients $A_\alpha$, which is absolutely convergent in a neighborhood $ \mathcal{O}$ of 
$\boldsymbol{z}=\boldsymbol{0}$. Define, the holomorphic map $A:  \mathcal{O} \to  \mathcal{B}(H)$, 
$$  
  A(\boldsymbol{z}):= \sum_{\alpha\in \mathbb{N}^n} \boldsymbol{z}^\alpha A_\alpha
$$
and assume that, it is symmetric. If $\lambda$ is an eigenvalue 
of $A_0 \equiv A(\boldsymbol{0})$ with finite multiplicity $h$ (and respective eigenvectors $\psi_i$, $i= 1,\ldots,h$), 
then there exist a neighborhood $U \subset  \mathcal{O}$
of ${ \boldsymbol{0} }$, and holomorphic functions
$$
\begin{aligned}
    \boldsymbol{z} \in U &\, \mapsto \, \lambda_1(\boldsymbol{z}), \lambda_2(\boldsymbol{z}), \ldots, \lambda_h(\boldsymbol{z}) \in \mathbb{R},
    \\[5pt]
    \boldsymbol{z} \in U &\, \mapsto \, \psi_1(\boldsymbol{z}), \psi_2(\boldsymbol{z}), \ldots, \psi_h(\boldsymbol{z}) \in H\setminus \{0\},
\end{aligned}
$$
satisfying for each $\boldsymbol{z} \in U$ and $i \in \{1,\ldots,h\}:$
		\begin{itemize}
			\item[$(i)$] $A(\boldsymbol{z}) \psi_i(\boldsymbol{z}) = \lambda_i(\boldsymbol{z}) \psi_i(\boldsymbol{z})$, 
			\item[$(ii)$] ${ \lambda_i(\boldsymbol{z} = \boldsymbol{0}) = \lambda }$, 
			\item[$(iii)$] ${ {\rm dim} {\{w \in H \; ; \; A(\boldsymbol{z}) w = \lambda_i(\boldsymbol{z}) w \}} \leqslant h }$.
		\end{itemize}
Moreover, if there exists $d> 0$ such that 
$
   \sigma(A_0)\cap (\lambda-d, \lambda+d) = {\left\{ \lambda \right\}},
$
then for each $d^\prime\in(0,d)$ there exists a neighborhood $W \subset U$ of $\boldsymbol{0}$, 
such that
\begin{equation}
\label{FINALPERT}
   \sigma(A(\boldsymbol{z})) \cap (\lambda - d^\prime, \lambda + d^\prime) = {\left\{ \lambda_1(\boldsymbol{z}), \ldots, \lambda_h(\boldsymbol{z}) \right\}}, 
   \quad \text{for all $\boldsymbol{z}\in W$.}
\end{equation} 
\end{theorem}
	
\section{\!\! On Schr\"odinger Equations Homogenization}
\label{HomoSchEqu}

In this section, we shall describe the asymptotic behaviour of the family of solutions 
$\{ u_{\varepsilon}{\}_{\varepsilon>0}}$ of the equation~\eqref{jhjkhkjhkj765675233},
this is the content of Theorem~\ref{876427463tggfdhgdfgkkjjlmk} below. It generalizes the similar result of Allaire, Piatnitski \cite{AllairePiatnitski} where they consider the 
similar problem in the periodic setting.  Our scenario is much different from one considered by them. 
Here, the coefficients of equation~\eqref{jhjkhkjhkj765675233} are random perturbations accomplished by 
stochastic diffeomorphisms of stationary functions.  
Since the two-scale convergence technique is the best tool to deal with 
asymptotic analysis of linear operators, we make use of it in analogous way done in~\cite{AllairePiatnitski}. 
Although, the presence of the stochastic deformation in the coefficients
brings out several complications, which we were able to overcome. 

\medskip
To begin, some important and basic a priori estimates of the solution of the Schr\"odinger 
equation \eqref{jhjkhkjhkj765675233} are needed. Then, we have the following  

\begin{lemma}[Energy Estimates]
\label{63457rf2wertgh}
	Assume that the conditions \eqref{ASSUM1}, \eqref{ASSUM2} hold and let 
${ u_\varepsilon }\in C\big([0,T);H^1(\R^n)\big)$ be the solution of the equation \eqref{jhjkhkjhkj765675233} with initial data 
	$u_{\ve}^0$. Then, for all ${ t\in [0,T] }$ and a.e. ${ \omega \in \Omega }$, the following a priori estimates hold:
	\begin{itemize}
		\item[(i)] $($Energy Conservation$.)$ ${ \displaystyle\int_{\mathbb{R}^n} {\vert u_\varepsilon(t,x,\omega) \vert}^2 dx = \int_{\mathbb{R}^n} {\vert u_{\varepsilon}^0(x,\omega) \vert}^2 dx }$.
		\item[(ii)] $( \varepsilon \nabla-$ Estimate$.)$
		\begin{eqnarray*}
		 \int_{\mathbb{R}^n} |\ve\nabla u_{\varepsilon}(t,x,\omega)|^2\, dx
		\le C\int_{\R^n}\Big\{|\ve\nabla u_{\ve}^0(x,\omega)|^2+|u_{\ve}^0(x,\omega)|^2\Big\} \,dx,
	 	\end{eqnarray*}
where ${ C:=C\big(\Lambda,{\Vert A \Vert}_\infty,{\Vert V \Vert}_\infty,{\Vert U \Vert}_\infty} \big) $ is a positive constant which does not depend on $\varepsilon > 0$.
	        	\end{itemize}
\end{lemma}

Its important to remember the followings facts that will be necessary in this section:
\begin{itemize}
\item The initial data of the equation \eqref{jhjkhkjhkj765675233} is assumed to be well-prepared, that is, 
for $(x,\omega) \in \mathbb{R}^n \! \times \! \Omega$, and 
${ \theta^\ast \in \mathbb{R}^n }$
\begin{equation}
\label{WellPreparedness}
	u_\varepsilon^0(x,\omega) = e^{2i\pi \frac{\theta^\ast \cdot x}{\varepsilon}}\psi \left( \Phi^{-1}(x/\ve,\omega),\omega,\theta^\ast \right) 
v^0(x),
\end{equation} 
where {${ v^0 \in C_{\rm c}^\infty(\mathbb{R}^n) }$}, and ${ \psi(\theta^\ast) }$ is an eigenfunction of the cell problem \eqref{92347828454trfhfd4rfghjls}.  
\item Using the Ergodic Theorem, it is easily seen that the sequences 
$${ \{u_{\varepsilon}^0(\cdot,\omega)\}_{\varepsilon >0} } \quad \text{and} \quad 
{ \{\varepsilon \nabla u_{\varepsilon}^0(\cdot,\omega) \}_{\varepsilon > 0} }$$ 
are bounded in ${ L^2(\mathbb{R}^n) }$ and ${ [L^2(\mathbb{R}^n)]^n }$, respectively. 
\end{itemize}

\subsection{The Abstract Theorem.}
\label{ATH}

In the next, we establish an abstract homogenization theorem for Schr\"odinger equations.

\begin{theorem}
\label{876427463tggfdhgdfgkkjjlmk}
Let $\Phi(y,\omega)$ be a stochastic deformation, and $\tau:\Z^n\times \Omega\to \Omega$ an ergodic $n-$dimensional dynamical 
system. 
Assume the conditions \eqref{ASSUM1}, \eqref{ASSUM2}, and 
there exists a Bloch frequence ${ \theta^\ast \! \in \mathbb{R}^n }$ such that \eqref{conds} holds,
with ${ \lambda (\theta^\ast) }$ associated to the eigenfunction 
$\Psi(z,\omega,\theta^\ast)\equiv \psi{\left( \Phi^{-1}(z,\omega),\omega,\theta^\ast \right)}$. Also, assume that the initial data $u_\varepsilon^0$ satisfies
\eqref{WellPreparedness}. If ${ u_\varepsilon }\in C\big([0,T);H^1(\R^n)\big)$  is the solution of~\eqref{jhjkhkjhkj765675233} for each $\varepsilon> 0$ fixed, then the sequence 
${ v_\varepsilon }$ defined by 
		\begin{equation*}
			v_\varepsilon(t,x,\omega) := e^{ -{\left( i \frac{\lambda(\theta^\ast) t}{\varepsilon^2} + 2i\pi \frac{\theta^\ast \! \cdot x}{\varepsilon} \right)} } u_\varepsilon(t,x,\omega), \;\, (t,x) \in \mathbb{R}^{n+1}_T, \; \omega \in \Omega, 
		\end{equation*}
$\Phi_{\omega}-$two-scale converges to the function ${ v(t,x) \, \psi{\left( \Phi^{-1}(z,\omega),\omega, \theta^\ast \right)} }$, and satisfies for a.e.  ${ \omega \in \Omega }$
		\begin{equation*}
			\lim_{\varepsilon \to 0} \iint_{\mathbb{R}^{n+1}_T} \! {\left\vert v_\varepsilon (t,x,\omega) - v(t,x) \, \psi{\left( \Phi^{-1} {\left(\frac{x}{\varepsilon},\omega \right)}, \omega, \theta^\ast \right)} \right\vert}^2 dx \, dt \, = \, 0,
		\end{equation*}
where $v \in C\big([0,T); L^2(\mathbb{R}^n)\big)$  is the unique solution of the homogenized Schr\"odinger equation 
\begin{equation}
\label{HomSchEqu}
\left\{
\begin{aligned}
   & i \displaystyle\frac{\partial v}{\partial t} - {\rm div} {\left( A^\ast \nabla v \right)} + U^\ast v= 0 \, , \;\, \text{in} \;\, \mathbb{R}^{n+1}_T, 
   \\[5pt]
   &	v(0,x) = v^0(x) \, , \;\, x\in \mathbb{R}^n,
\end{aligned}
\right.
\end{equation}
with effective (constant) coefficients: matrix ${ A^\ast = D_\theta^2 \lambda(\theta^\ast) }$, and potential   
		\begin{equation*}
			U^\ast =  c^{-1}_\psi \int_\Omega \int_{\Phi([0,1)^n, \omega)} U{\left( \Phi^{-1} (z,\omega),\omega \right)}  {\left\vert \psi {\left( \Phi^{-1} (z,\omega), \omega, \theta^\ast \right)} \right\vert}^2 dz \, d\mathbb{P}(\omega),
		\end{equation*}
where
		\begin{equation*}
			c_\psi = \int_\Omega \int_{\Phi([0,1)^n, \omega)} {\left\vert \psi {\left( \Phi^{-1} (z,\omega), \omega, \theta^\ast \right)} \right\vert}^2 dz \, d\mathbb{P}(\omega).
		\end{equation*}
\end{theorem}

\begin{proof}
In order to better understand the main difficulties brought by the presence of the stochastic deformation $\Phi$, we split our proof in five steps. 

\medskip
1.({\it\bf A priori estimates and $\Phi_{\omega}-$two-scale convergence.}) 
First, we define 
\begin{equation}
\label{jghkd65454ads3e}
	    	v_\varepsilon (t,x,\widetilde{\omega}) := e^{-{\left( i \frac{\lambda(\theta^\ast) t}{\varepsilon^2} + 2i\pi \frac{\theta^\ast \! \cdot x}{\varepsilon} \right)} } 
		u_\varepsilon(t,x,\widetilde{\omega}), \;\, (t,x,\widetilde{\omega}) \in \mathbb{R}^{n+1}_T \! \times \! \Omega.
\end{equation}
Then, computing the first derivatives with respect to the variable $x$, we get
	  \begin{equation}
	  \label{974967uhghjzpzas}
	    	\varepsilon \nabla u_\varepsilon (t,x,\widetilde{\omega})\, e^{-{\left( i \frac{\lambda(\theta^\ast) t}{\varepsilon^2} 
		+ 2i\pi \frac{\theta^\ast \! \cdot x}{\varepsilon} \right)} } = (\varepsilon \nabla + 2i\pi \theta^\ast) v_\varepsilon (t,x,\widetilde{\omega}).
	    \end{equation}
Applying Lemma~\ref{63457rf2wertgh} yields:
		\begin{itemize}
			\item ${ \displaystyle\int_{\mathbb{R}^n} {\vert v_\varepsilon(t,x,\widetilde{\omega}) \vert}^2 dx = \int_{\mathbb{R}^n} 
			{\vert u_\varepsilon^0(x,\widetilde{\omega}) \vert}^2 dx },$
			\item ${ \displaystyle\int_{\mathbb{R}^n} {\vert \varepsilon \nabla v_\varepsilon(t,x,\widetilde{\omega}) \vert}^2 dx \leq \widetilde{C} {\displaystyle\int_{\mathbb{R}^n} 
\Big( {\vert \varepsilon \nabla u_\varepsilon^0(x,\widetilde{\omega}) \vert}^2 +{\vert u_\varepsilon^0(x,\widetilde{\omega}) 
\vert}^2 \Big) dx} }$
		\end{itemize}
for all ${ t\in[0,T) }$ and a.e. $\widetilde{\omega} \in\Omega$, where the constant ${ \widetilde{C} }$ depends on
$\| A{\|}_{\infty}$, $\|V{\|}_{\infty}$, $\|U{\|}_{\infty}$ and $\theta^\ast$.  
Then, from the uniform boundedness of the sequences 
${ \{u_{\varepsilon}^0(\cdot,\widetilde{\omega})\}_{\varepsilon >0} }$ and ${ \{\varepsilon \nabla u_{\varepsilon}^0(\cdot,\widetilde{\omega}) \}_{\varepsilon > 0} }$,  we deduce that the sequences 
$${ {\{ v_{\varepsilon}(\cdot,\cdot\cdot,\widetilde{\omega}) \}}_{\varepsilon > 0} } \quad \text{and} \quad { {\{ \varepsilon\nabla v_\varepsilon(\cdot, \cdot\cdot, \widetilde{\omega}) \}}_{\varepsilon > 0} }$$
are bounded, respectively, in ${ L^2(\mathbb{R}^{n+1}_T) }$ and ${ {[ L^2(\mathbb{R}^{n+1}_T)]}^n }$ for a.e. ${ \widetilde{\omega} \in \Omega }$. Therefore, applying Lemma 
\ref{SYM1-5}, there exists a subsequence ${ \{\varepsilon^\prime\} }$(which may dependent on $\tilde{\omega}$), and a stationary function 
${ v^\ast_{\widetilde{\omega}} \in L^2(\mathbb{R}^{n+1}_T, \mathcal{H}) }$, for a.e. ${ \widetilde{\omega} \in \Omega }$, such that  
		\begin{equation*}
			v_{\varepsilon^\prime}(t,x,\widetilde{\omega}) \; \xrightharpoonup[\varepsilon^\prime \to 0]{2-{\rm s}}\; v^\ast_{\widetilde{\omega}} {\left( t,x,\Phi^{-1}(z,\omega),\omega \right)},
		\end{equation*}
		and
		\begin{equation*}
			\varepsilon^\prime \frac{\partial v_{\varepsilon^\prime}}{\partial x_k}(t,x,\widetilde{\omega}) \; \xrightharpoonup[\varepsilon^\prime \to 0]{2-{\rm s}} \; 
			\frac{\partial }{\partial z_k} {\big( v^\ast_{\widetilde{\omega}} {\left(t,x,\Phi^{-1}{(z,\omega)},\omega \right)} \big)},                
		\end{equation*}
which means that, for ${ k\in \{1,\ldots,n\} }$, we have 
\begin{equation}
\label{jhjkhjfdasdfghyui}
\begin{aligned}
      \lim_{\varepsilon^\prime \to 0}\iint_{\mathbb{R}^{n+1}_T} & v_{\varepsilon^\prime} \left(t, x,\widetilde{\omega} \right) \, 
			\overline{ \varphi(t,x) \,\Theta\left(\Phi^{-1}{\left(\frac{x}{\varepsilon^\prime},\widetilde{\omega} \right)},\widetilde{\omega} \right) } \, dx \, dt
\\[5pt]
& = c_{\Phi}^{-1} \!\! \displaystyle\iint_{\mathbb{R}^{n+1}_T} \! \int_\Omega \int_{\Phi([0,1)^n, \omega)} \!\!\!\!\!   v^\ast_{\widetilde{\omega}} 
			{\left(t,x,\Phi^{-1}{(z,\omega)},\omega \right)} \,
\\[5pt]			
			 & \hspace{90pt} \times \; \overline{ \varphi(t,x) \,\Theta\left(\Phi^{-1}(z,\omega),\omega \right)}   \, dz \, d\mathbb{P} \, dx \, dt
\end{aligned}
\end{equation}
and 
\begin{equation}
\label{uygkgcxzcxzsw}
\begin{aligned}
               \lim_{\varepsilon^\prime \to 0} \iint_{\mathbb{R}^{n+1}_T} & \varepsilon^\prime \frac{\partial v_{\varepsilon^\prime}}{\partial x_k} \left(t, x, \widetilde{\omega} \right) \, 
               \overline{ \varphi (t,x)\,\Theta\left(\Phi^{-1}{\left(\frac{x}{\varepsilon^\prime},\widetilde{\omega} \right)},\widetilde{\omega} \right) } \, dx \, dt
\\[5pt]
&= c_{\Phi}^{-1} \!\! \displaystyle\iint_{\mathbb{R}^{n+1}_T} \! \int_\Omega \int_{\Phi([0,1)^n, \omega)} \frac{\partial }{\partial z_k} {\left( v^\ast_{\widetilde{\omega}} {\left(t,x,\Phi^{-1}{(z,\omega)},\omega \right)} \right)} \, 
\\[5pt]
& \hspace{90pt} \times \, \overline{ \varphi (t,x)\,\Theta\left(\Phi^{-1}(z,\omega),\omega \right)}  \, dz \, d\mathbb{P} \, dx \, dt,
\end{aligned}
\end{equation}
for all functions $\varphi \in C^\infty_{\rm c}((-\infty, T)  \times  \mathbb{R}^n)$ and $\Theta \in L^{2}_{\loc}\left(\R^n\times\Omega\right)$ stationary. Moreover, the sequence 
${ {\{ v_\varepsilon^0(\cdot, \widetilde{\omega}) \}}_{\varepsilon > 0} }$ defined by, 
		\begin{equation}\label{5t345rte54ew3e2wswqqq1qdecv}
			v_\varepsilon^0(x,\omega):=\psi{\left( \Phi^{-1}{\left( \frac{x}{\varepsilon},\omega \right)},\omega,\theta^\ast \right)} v^0(x), \;\; (x,\omega) \in \mathbb{R}^n \! \times \! \Omega,
		\end{equation}
satisfies
		\begin{equation}\label{766765t6y4rf5tzxcvbsgfhry}
			v_{\varepsilon}^0(\cdot,\widetilde{\omega}) \; \xrightharpoonup[\varepsilon \to 0]{2-{\rm s}}\; v^0(x) \, \psi{\left( \Phi^{-1}(z,\omega),\omega,\theta^\ast \right)},
		\end{equation}
for each stationary function ${ \psi(\theta^\ast) }$.
		
\bigskip
2.({\it\bf The Split Process.}) We consider the following 

\medskip
 \underline {Claim:} 
There exists ${ v_{\widetilde{\omega}} \in L^2(\mathbb{R}^{n+1}_T) }$, such that 
$$
\begin{aligned}
   v^\ast_{\widetilde{\omega}} {\left( t,x,\Phi^{-1}(z,\omega),\omega \right)}&= v_{\widetilde{\omega}}(t,x) \, \psi {\left(\Phi^{-1}(z,\omega),\omega, \theta^\ast \right)} 
\\[5pt]
   &\equiv v_{\widetilde{\omega}} (t,x) \, \Psi(z,\omega, \theta^\ast).
\end{aligned}       
$$
  
Proof of Claim: First, for any $\widetilde{\omega} \in \Omega$ fixed, we take the function 
\begin{equation}
\label{7676567543409hj}
			Z_\varepsilon (t,x,\widetilde{\omega}) = \varepsilon^2 e^{i \frac{\lambda(\theta^\ast) t}{\varepsilon^2} + 2i\pi \frac{\theta^\ast \! \cdot x}{\varepsilon}} \varphi (t, x) \, 
			\Theta {\big( \Phi^{-1}\big( \frac{x}{\varepsilon}, \widetilde{\omega} \big), \widetilde{\omega} \big)}
\end{equation}
as a test function in the associated variational formulation of the equation \eqref{jhjkhkjhkj765675233}, where ${ \varphi \in C^\infty_{\rm c}((-\infty, T) \! \times \! \mathbb{R}^n) }$ 
and $ \Theta\in L^{\infty}\left(\R^n\times\Omega\right)$ stationary, with $\Theta(\cdot,\omega)$ smooth. Therefore, we obtain 
$$
\begin{aligned} 
			&- i \iint_{\mathbb{R}^{n+1}_T} u_\varepsilon(t,x,\widetilde{\omega}) \, \frac{\partial \overline{Z_\varepsilon}}{\partial t} (t,x,\widetilde{\omega}) \, dx \, dt
			+ i \int_{\mathbb{R}^n} u_\varepsilon^0(x,\widetilde{\omega}) \, \overline{Z_\varepsilon}(0,x,\widetilde{\omega}) \, dx
\\[5pt]
			&+ \iint_{\mathbb{R}^{n+1}_T} A {\left(\Phi^{-1}\left( \frac{x}{\varepsilon}, \widetilde{\omega} \right),\widetilde{\omega} \right)} \nabla u_\varepsilon(t,x,\widetilde{\omega}) \cdot \nabla \overline{Z_\varepsilon}(t,x,\widetilde{\omega}) \, dx \, dt
\\[5pt]
			&+ \frac{1}{\varepsilon^2} \iint_{\mathbb{R}^{n+1}_T} V {\left(\Phi^{-1}\left( \frac{x}{\varepsilon}, \widetilde{\omega} \right),\widetilde{\omega} \right)} u_\varepsilon(t,x,\widetilde{\omega}) \, \overline{Z_\varepsilon}(t,x,\widetilde{\omega}) \, dx \, dt
\\[5pt]
			&+  \iint_{\mathbb{R}^{n+1}_T} U {\left( \Phi^{-1}\left( \frac{x}{\varepsilon}, \widetilde{\omega} \right),\widetilde{\omega} \right)} u_\varepsilon(t,x,\widetilde{\omega})
%
			\, \overline{Z_\varepsilon}(t,x,\widetilde{\omega}) \, dx \, dt= 0,
\end{aligned}
$$
and since 		
$$
\begin{aligned}
			\frac{\partial Z_\varepsilon }{\partial t} (t,x,\widetilde{\omega})&= i \lambda(\theta^\ast) \, e^{i \frac{\lambda(\theta^\ast) t}{\varepsilon^2} 
			+ 2i\pi \frac{\theta^\ast \! \cdot x}{\varepsilon}} \, \varphi (t, x) \, 
			\Theta {( \Phi^{-1}( \frac{x}{\varepsilon}, \widetilde{\omega}), \widetilde{\omega})} + \mathrm{O}(\varepsilon^2), 
\\[5pt]
			\nabla Z_\varepsilon (t,x,\widetilde{\omega})&= \varepsilon \, e^{i \frac{\lambda(\theta^\ast) t}{\varepsilon^2} + 2i\pi \frac{\theta^\ast \! \cdot x}{\varepsilon}} \, (\varepsilon \nabla + 2i\pi \theta^\ast) 
			\big( \varphi(t,x) \, \Theta{( \Phi^{-1}{(\frac{x}{\varepsilon},\widetilde{\omega})},\widetilde{\omega})}\big), 
\end{aligned}
$$
it follows that
$$
\begin{aligned}
   &-  \lambda(\theta^\ast) \displaystyle\iint_{\mathbb{R}^{n+1}_T} v_\varepsilon(t,x,\widetilde{\omega}) \,  
   \overline{ \varphi (t, x) \, \Theta {( \Phi^{-1}\big( \frac{x}{\varepsilon}, \widetilde{\omega} \big), \widetilde{\omega})} } \, dx \, dt 
\\[5pt]
  & + \iint_{\mathbb{R}^{n+1}_T} A {( \Phi^{-1}{\big(\frac{x}{\varepsilon},\widetilde{\omega} \big)},\widetilde{\omega})} \, {(\varepsilon \nabla + 2i\pi \theta^\ast)} v_\varepsilon(t,x,\widetilde{\omega})
\\[5pt]
		& \hspace{60pt} \cdot   \overline{ {(\varepsilon \nabla + 2i\pi \theta^\ast)} {\left( \varphi (t,x) \, \Theta {\left(\Phi^{-1}{\left(\frac{x}{\varepsilon},\widetilde{\omega} \right)},\widetilde{\omega} \right)} \right)} } \, dx \, dt 
\\[5pt]
   &+ \iint_{\mathbb{R}^{n+1}_T} V {\left( \Phi^{-1}{\left(\frac{x}{\varepsilon},\widetilde{\omega} \right)},\widetilde{\omega} \right)} \, v_\varepsilon(t,x,\widetilde{\omega})  \, \overline{ \varphi (t,x) \, \Theta {\left(\Phi^{-1}{\left(\frac{x}{\varepsilon},\widetilde{\omega} \right)},\widetilde{\omega} \right)} } \, dx \, dt= \mathrm{O}(\varepsilon^2),
\end{aligned}
$$
where we have used \eqref{jghkd65454ads3e}, \eqref{974967uhghjzpzas}, \eqref{5t345rte54ew3e2wswqqq1qdecv}, and \eqref{7676567543409hj}.
Although, it is more convenient to rewrite as 
\begin{equation}
\label{56tugryfgrffdd}
\begin{aligned}
 &- \lambda(\theta^\ast) \displaystyle\iint_{\mathbb{R}^{n+1}_T} v_\varepsilon(t,x,\widetilde{\omega}) \,  \overline{ \varphi (t,x) \, \Theta {( \Phi^{-1}\big( \frac{x}{\varepsilon}, \widetilde{\omega}\big), \widetilde{\omega})} } \, dx \, dt
 \\[5pt]
&+ \iint_{\mathbb{R}^{n+1}_T} {(\varepsilon \nabla + 2i\pi \theta^\ast)} v_\varepsilon(t,x,\widetilde{\omega}) 
\\[5pt]
& \hspace{20pt} \cdot \overline{ A {( \Phi^{-1}{\big(\frac{x}{\varepsilon},\widetilde{\omega} \big)},\widetilde{\omega})} \, 
{(\varepsilon \nabla + 2i\pi \theta^\ast)} {\big( \varphi(t,x) \, \Theta{( \Phi^{-1}{\big(\frac{x}{\varepsilon},\widetilde{\omega} \big)},\widetilde{\omega})} \big)} } \, dx \, dt
\\[5pt]
&+ \iint_{\mathbb{R}^{n+1}_T} v_\varepsilon(t,x,\widetilde{\omega}) \, \overline{ \varphi (t,x) \, V {( \Phi^{-1}{\big(\frac{x}{\varepsilon},\widetilde{\omega} \big)},\widetilde{\omega})}
 \, \Theta{( \Phi^{-1}{\big(\frac{x}{\varepsilon},\widetilde{\omega} \big)},\widetilde{\omega})} } \, dx \, dt= \mathrm{O}(\varepsilon^2).
\end{aligned}
\end{equation}
Now, making $\ve={ \varepsilon^\prime}$, letting $\ve'\to 0 $ and using the Definition \ref{two-scale}, we have for a.e. ${ \widetilde{\omega} \in \Omega }$,
for all ${ \varphi \in C^\infty_{\rm c}((-\infty, T) \! \times \! \mathbb{R}^n) }$, 
$\Theta \in L^{\infty}\left(\R^n\times\Omega\right)$ stationary and $\Theta(\cdot,\omega)$ smooth, 
%
\begin{equation*}
\begin{aligned}
&- \lambda(\theta^\ast) \, c_\Phi^{-1} \!\!\! \displaystyle\iint_{\mathbb{R}^{n+1}_T} \! \int_\Omega \int_{\Phi([0,1)^n, \omega)} v^\ast_{\widetilde{\omega}} {\left( t,x,\Phi^{-1}(z,\omega),\omega \right)} 
\\[7pt]
& \hspace{90pt} \times \overline{ \varphi (t,x) \, \Theta {\left( \Phi^{-1}(z,\omega),\omega \right)} } \, dz \, d\mathbb{P}(\omega) \, dx \, dt
\\[7pt]
&+ c_\Phi^{-1} \!\!\! \displaystyle\iint_{\mathbb{R}^{n+1}_T} \! \int_\Omega \int_{\Phi([0,1)^n, \omega)} {(\nabla_{\!\! z} + 2i\pi \theta^\ast)} {\left( v^\ast_{\widetilde{\omega}} {\left( t,x,\Phi^{-1}(z,\omega),\omega \right)} \right)}
\\[7pt]
& \cdot \overline{ \varphi (t,x) \, A {\left( \Phi^{-1}(z,\omega),\omega \right)}  {\left[ {(\nabla_{\!\! z} + 2i\pi \theta^\ast)} {\left(  \Theta {\left( \Phi^{-1}(z,\omega),\omega \right)} \right)} \right]} } \, dz \, d\mathbb{P}(\omega) \, dx \, dt 
\\[7pt]
&+ c_\Phi^{-1} \!\!\! \displaystyle\iint_{\mathbb{R}^{n+1}_T} \! \int_\Omega \int_{\Phi([0,1)^n, \omega)} v^\ast_{\widetilde{\omega}} {\left( t,x,\Phi^{-1}(z,\omega),\omega \right)} \, 
\\[7pt]
&\hspace{60pt} \times  \overline{ \varphi (t,x) \, V {\left( \Phi^{-1}(z,\omega),\omega \right)} \, \Theta {\left( \Phi^{-1}(z,\omega),\omega \right)} } \, dz \, d\mathbb{P}(\omega) \, dx \, dt = 0.
\end{aligned}
\end{equation*}
Therefore, due to an argument 
of density in the test functions
(thanks to the topological structure of $\Omega$), 
we can conclude that 
\begin{equation*}
\begin{aligned}
&- \lambda(\theta^\ast) \, \displaystyle\int_\Omega \int_{\Phi([0,1)^n, \omega)} v^\ast_{\widetilde{\omega}} {\left( t,x,\Phi^{-1}(z,\omega),\omega \right)} \,  \overline{ \Theta {\left( \Phi^{-1}(z,\omega),\omega \right)} } \, dz \, d\mathbb{P}(\omega) 
\\[5pt]
&+ \displaystyle\int_\Omega \int_{\Phi([0,1)^n, \omega)} A {\left( \Phi^{-1}(z,\omega),\omega \right)} {(\nabla_{\!\! z} + 2i\pi \theta^\ast)} {\left( v^\ast_{\widetilde{\omega}} {\left( t,x,\Phi^{-1}(z,\omega),\omega \right)} \right)}
\\[5pt]
&\hspace{60pt} \cdot \overline{ {(\nabla_{\!\! z} + 2i\pi \theta^\ast)} {\left( \Theta {\left( \Phi^{-1}(z,\omega),\omega \right)} \right)} } \, dz \, d\mathbb{P}(\omega)
\\[5pt]
&+ \int_\Omega \int_{\Phi([0,1)^n, \omega)} \!\!\! V {( \Phi^{-1}(z,\omega),\omega)} \,  v^\ast_{\widetilde{\omega}} 
{( t,x,\Phi^{-1}(z,\omega),\omega)} \, 
\\[5pt]
&\hspace{120pt} \times  \overline{ \Theta {( \Phi^{-1}(z,\omega),\omega)} } \, dz \, d\mathbb{P}(\omega)= 0,
\end{aligned}
\end{equation*}
for a.e. ${ (t,x) \in \mathbb{R}^{n+1}_T }$ and for all ${ \Theta}$ as above.  Thus, the simplicity of the eigenvalue $\lambda(\theta^\ast)$ 
assures us that for a.e.  $\mathbb{R}^{n+1}_T $, the function 
$${ (z,\omega) \mapsto v^\ast_{\widetilde{\omega}}{\left( t,x,\Phi^{-1}(z,\omega),\omega \right)} }$$ (which belongs to the space ${ \mathcal{H} }$) is parallel to the 
function ${ \Psi(\theta^\ast) }$, i.e., we can find ${ v_{\widetilde{\omega}}(t,x) \in \mathbb{C} }$, such that 
		\begin{eqnarray*}
v^\ast_{\widetilde{\omega}} {\left( t,x,\Phi^{-1}(z,\omega),\omega \right)} &=& v_{\widetilde{\omega}}(t,x) \, \Psi(z,\omega, \theta^\ast)\\
& \equiv & v_{\widetilde{\omega}}(t,x) \, \psi {\left(\Phi^{-1}(z,\omega),\omega, \theta^\ast \right)}.
		\end{eqnarray*}
		
\medskip
Finally, since ${ v^\ast_{\widetilde{\omega}} \in L^2 (\mathbb{R}^{n+1}_T; \mathcal{H}) }$, we conclude that  ${ v_{\widetilde{\omega}}\in L^2(\mathbb{R}^{n+1}_T) }$, which 
completes the proof of our claim.

\medskip
3.({\it\bf Homogenization Process.}) Let  ${ \Lambda_k (\theta^\ast)}$, for any $k\in \{1,\ldots,n\}$, be the function defined by 
		\begin{equation*}
			\Lambda_k(z,\omega,\theta^\ast)=\frac{1}{2i\pi}\frac{\partial \Psi}{\partial \theta_k}(z,\omega,\theta^\ast)=\frac{1}{2i\pi}\frac{\partial \psi}{\partial \theta_k}
			{\left( \Phi^{-1}(z,\omega),\omega,\theta^\ast \right)}, \; (z,\omega) \in \mathbb{R}^n \! \times \! \Omega,
		\end{equation*}
where the function $\Psi(z,\omega,\theta^\ast)=\psi{\left( \Phi^{-1}(z,\omega),\omega,\theta^\ast \right)}$ 
is the eigenfunction of the spectral cell problem~\eqref{92347828454trfhfd4rfghjls}. 
Then, we consider the following test function 
\begin{equation*}
			Z_\varepsilon(t,x,\widetilde{\omega}) = e^{i \frac{\lambda(\theta^\ast) t}{\varepsilon^2} + 2i\pi \frac{\theta^\ast \! \cdot x}{\varepsilon}} {\big( \varphi(t,x) \, 
			\Psi_\varepsilon(x,\widetilde{\omega},\theta^\ast)  +\varepsilon  \sum_{k=1}^n 
			\frac{\partial \varphi}{\partial x_k}(t,x) \, \Lambda_{k,\varepsilon} (x,\widetilde{\omega},\theta^\ast) \big)},
\end{equation*}
where ${ \varphi \in C^\infty_{\rm c}((-\infty, T) \! \times \! \mathbb{R}^n) }$ and 
		\begin{equation*}
			\Psi_\varepsilon(x,\widetilde{\omega}, \theta^\ast) = \Psi{\big( \frac{x}{\varepsilon},\widetilde{\omega},\theta^\ast\big)},
			\quad \;\; \Lambda_{k,\varepsilon}(x,\widetilde{\omega}, \theta^\ast) = \Lambda_k{\big( \frac{x}{\varepsilon},\widetilde{\omega},\theta^\ast\big)}.
		\end{equation*}    
Using the function ${ Z_\varepsilon }$ as test function in the variational formulation of the equation \eqref{jhjkhkjhkj765675233}, 
we obtain		
\begin{equation}
\label{676745459023v}
\begin{aligned}
& \big[ i \displaystyle\int_{\mathbb{R}^n} u_\varepsilon^0(x,\widetilde{\omega}) \, \overline{Z_\varepsilon}(0,x,\widetilde{\omega}) \, dx  
 - i \displaystyle\iint_{\mathbb{R}^{n+1}_T} u_\varepsilon(t,x,\widetilde{\omega}) \, \frac{\partial \overline{Z_\varepsilon}}{\partial t} (t,x,\widetilde{\omega}) \, dx \, dt \big]
\\[5pt]
&+ \big[ \displaystyle\iint_{\mathbb{R}^{n+1}_T} A {\left(\Phi^{-1}{\left( \frac{x}{\varepsilon},\widetilde{\omega} \right)},\widetilde{\omega} \right)} \, 
			\nabla u_\varepsilon (t,x,\widetilde{\omega}) \cdot \nabla \overline{Z_\varepsilon}(t,x,\widetilde{\omega}) \, dx \, dt \big]
\\[5pt]
&			+ \big[ \displaystyle\frac{1}{\varepsilon^2} \iint_{\mathbb{R}^{n+1}_T} V {\left(\Phi^{-1}{\left( \frac{x}{\varepsilon},\widetilde{\omega} \right)},\widetilde{\omega} \right) \, 
			u_\varepsilon(t,x,\widetilde{\omega}) \, \overline{Z_\varepsilon}(t,x,\widetilde{\omega}) \, dx \, dt} 
\\[5pt]
& \hspace{30pt}			+ \iint_{\mathbb{R}^{n+1}_T} U {\left(x, \Phi^{-1}{\left( \frac{x}{\varepsilon},\widetilde{\omega} \right)},\widetilde{\omega} \right)} \, 
			u_\varepsilon(t,x,\widetilde{\omega}) \, \overline{Z_\varepsilon}(t,x,\widetilde{\omega}) \, dx \, dt \big]= 0. 
\end{aligned}
\end{equation}
In order to simplify the manipulation of the above equation, we shall denote by $I_{k}^{\varepsilon}(k=1,2,3)$ the respective term in the $k^{\text{th}}$ brackets, so that we can rewrite the 
equation~\eqref{676745459023v} as $I_{1}^{\varepsilon}+I_{2}^{\varepsilon}+I_{3}^{\varepsilon}=0$.

\medskip		
The analysis of the $I_{1}^{\varepsilon}$ term is triggered  by the following computation 
$$
\begin{aligned}
	&\frac{\partial Z_\varepsilon}{\partial t}(t,x,\widetilde{\omega})=  e^{i \frac{\lambda(\theta^\ast) t}{\varepsilon^2} + 2i\pi \frac{\theta^\ast \! \cdot x}{\varepsilon}} 
	{\Big[ i\frac{\lambda(\theta^\ast)}{\varepsilon^2} { \big( \varphi(t,x) \, \Psi_\varepsilon(x,\widetilde{\omega},\theta^\ast) } }
\\[5pt]
			& \quad + \varepsilon\sum_{k=1}^n \frac{\partial \varphi}{\partial x_k}(t,x) \, \Lambda_{k,\varepsilon} \big)
			+ \frac{\partial \varphi}{\partial t}(t,x) \, \Psi_\varepsilon(x,\widetilde{\omega},\theta^\ast)  
			+ \, \varepsilon \sum_{k=1}^n \frac{\partial^2 \varphi}{\partial t \, \partial x_k}(t,x) \, \Lambda_{k,\varepsilon} \Big],
\end{aligned}
$$
therefore we have
$$
\begin{aligned}
I_{1}^{\varepsilon}&= i \int_{\mathbb{R}^n} v_\varepsilon^0 \, \overline{ {\big( \varphi (0,x)\, \Psi_\varepsilon(x,\widetilde{\omega},\theta^\ast) 
+\varepsilon  \sum_{k=1}^n \frac{\partial \varphi}{\partial x_k}(0,x) \, \Lambda_{k,\varepsilon}(x,\widetilde{\omega},\theta^\ast) \big)} } dx 
\\[5pt]
& - \frac{\lambda(\theta^\ast)}{\varepsilon^2} \iint_{\mathbb{R}^{n+1}_T} v_\varepsilon \, 
\overline{{\big( \varphi(t,x) \, \Psi_\varepsilon(x,\widetilde{\omega},\theta^\ast) +\varepsilon  \sum_{k=1}^n \frac{\partial \varphi}{\partial x_k}(t,x) \, 
\Lambda_{k,\varepsilon}(x,\widetilde{\omega},\theta^\ast) \big)}} dx \, dt 
\\[5pt] 
&- i \iint_{\mathbb{R}^{n+1}_T} v_\varepsilon \, \overline{ {\big( \frac{\partial \varphi}{\partial t}(t,x) \,
\Psi_\varepsilon(x,\widetilde{\omega},\theta^\ast) + \varepsilon \sum_{k=1}^n \frac{\partial^2 \varphi}{\partial t \, \partial x_k}(t,x) 
\Lambda_{k,\varepsilon}(x,\widetilde{\omega},\theta^\ast) \big)} }.
\end{aligned}
$$		
For the analysis of the term $I_{2}^{\varepsilon}$, we need to make the following computations  
$$
\begin{aligned}
\nabla Z_\varepsilon(t,x,\widetilde{\omega}) &= e^{i \frac{\lambda(\theta^\ast) t} {\varepsilon^2} + 2i\pi \frac{\theta^\ast \! \cdot x}{\varepsilon}} {\big[ \nabla \varphi(t,x) \, 
\Psi_\varepsilon(x,\widetilde{\omega},\theta^\ast) + \varphi(t,x) \, \nabla \Psi_\varepsilon(z,\widetilde{\omega},\theta^\ast) }
\\[5pt]
	&+  \varepsilon \sum_{k=1}^n \nabla {\big( \frac{\partial \varphi}{\partial x_k}(t,x) \big)} \, \Lambda_{k,\varepsilon}(x,\widetilde{\omega},\theta^\ast) 
	+ \varepsilon \sum_{k=1}^n \frac{\partial \varphi}{\partial x_k}(t,x) \, \nabla \Lambda_{k,\varepsilon}(x,\widetilde{\omega},\theta^\ast) 
\\[5pt]
	&+ \, 2i\pi \frac{\theta^\ast}{\varepsilon} { {\big( \varphi(t,x) \, \Psi_\varepsilon(x,\widetilde{\omega},\theta^\ast) + \varepsilon  
	\sum_{k=1}^n \frac{\partial \varphi}{\partial x_k}(t,x) \, \Lambda_{k,\varepsilon}(x,\widetilde{\omega},\theta^\ast) \big)} \big]} 
\\[5pt] 
	& = e^{i \frac{\lambda(\theta^\ast) t}{\varepsilon^2} + 2i\pi \frac{\theta^\ast \! \cdot x}{\varepsilon}} {\big[ \varphi(t,x) \, {\big( \nabla 
	+ 2i\pi\frac{\theta^\ast}{\varepsilon} \big)} \Psi_\varepsilon(x,\widetilde{\omega},\theta^\ast) } 
\\[5pt]
	&+ { \, \varepsilon \sum_{k=1}^n \frac{\partial \varphi}{\partial x_k}(t,x) \, {\big( \nabla + 2i\pi\frac{\theta^\ast}{\varepsilon} \big)} 
	\Lambda_{k,\varepsilon}(x,\widetilde{\omega},\theta^\ast) + \nabla \varphi(t,x) \, \Psi_\varepsilon(x,\widetilde{\omega},\theta^\ast) } 
\\[5pt]
	& + \, { \varepsilon \sum_{k=1}^n \nabla {\big( \frac{\partial \varphi}{\partial x_k}(t,x) \big)} \, \Lambda_{k,\varepsilon}(x,\widetilde{\omega},\theta^\ast) \big]},
\end{aligned}
$$	
and from this, we have 
$$
\begin{aligned}
&A {\left(\Phi^{-1}{\left( \frac{x}{\varepsilon},\widetilde{\omega} \right)},\widetilde{\omega} \right)} \nabla u_\varepsilon(t,x,\widetilde{\omega}) \cdot \overline{\nabla Z_\varepsilon}(t,x,\widetilde{\omega})
\\[5pt]
&= A {\left(\Phi^{-1}{\left( \frac{x}{\varepsilon},\widetilde{\omega} \right)},\widetilde{\omega} \right)} {\big[ \nabla u_\varepsilon(t,x,\widetilde{\omega}) \, 
e^{ -\left( {i \frac{\lambda(\theta^\ast) t}{\varepsilon^2} + 2i\pi \frac{\theta^\ast \! \cdot x}{\varepsilon}} \right)} \big]} 
\\[5pt]
&\cdot \big[ \overline{\varphi}(t,x) \, 
( \nabla - 2i\pi\frac{\theta^\ast}{\varepsilon} ) \overline{\Psi_\varepsilon}(x,\widetilde{\omega},\theta^\ast) } 
+  \varepsilon \! \sum_{k=1}^n \frac{\partial \overline{\varphi}}{\partial x_k}(t,x) \, {( \nabla \!\! - 2i\pi\frac{\theta^\ast}{\varepsilon})
\overline{\Lambda_{k,\varepsilon}}(x,\widetilde{\omega},\theta^\ast) 
\\[5pt]
&+ \nabla \overline{\varphi}(t,x) \, \overline{\Psi_\varepsilon}(x,\widetilde{\omega},\theta^\ast) 
+ \varepsilon \sum_{k=1}^n \nabla {\big( \frac{\partial \overline{\varphi}}{\partial x_k}(t,x) \big)} \, 
\overline{\Lambda_{k,\varepsilon}}(x,\widetilde{\omega},\theta^\ast) \big].
\end{aligned}
$$		
Then, from equation \eqref{974967uhghjzpzas} and using the terms
$$
{\big( \nabla + 2i\pi\frac{\theta^\ast}{\varepsilon} \big)} (v_\varepsilon(t,x,\widetilde{\omega}) \, \overline{\varphi}(t,x) ),
\quad 
\big( \nabla + 2i\pi\frac{\theta^\ast}{\varepsilon} \big) (v_\varepsilon(t,x,\widetilde{\omega}) \, \frac{\partial \overline{\varphi}}{\partial x_k}(t,x)), 
$$ 
it follows from the above equation that 
$$
    A {\left(\Phi^{-1}{\left( \frac{x}{\varepsilon},\widetilde{\omega} \right)},\widetilde{\omega} \right)} \nabla u_\varepsilon \cdot \overline{\nabla Z_\varepsilon}
    =\sum_{k=1}^n\left( I_{2,1}^{\varepsilon,k}+I_{2,2}^{\varepsilon,k}+I_{2,3}^{\varepsilon,k}\right)(t,x,\widetilde{\omega}),
$$
where 
$$
\begin{aligned}
 & I_{2,1}^{\varepsilon,k}(t,x,\widetilde{\omega}):=  \ve \,A {(\Phi^{-1}{\big( \frac{x}{\varepsilon},\widetilde{\omega} \big)},\widetilde{\omega} )} 
\big[ {\big( \nabla + 2i\pi\frac{\theta^\ast}{\varepsilon} \big)} {( v_\varepsilon(t,x,\widetilde{\omega}) \, \frac{\partial \overline{\varphi}}{\partial x_k}(t,x) )} \big]
 \\[5pt]
&\quad \cdot  \big[ {\big( \nabla - 2i\pi\frac{\theta^\ast}{\varepsilon} \big)} \overline{\Lambda_{k,\varepsilon}}(x,\widetilde{\omega},\theta^\ast) \big]
\\[5pt]
& - A(\Phi^{-1}{\big( \frac{x}{\varepsilon},\widetilde{\omega} \big)},\widetilde{\omega} )
\big[ v_\varepsilon(t,x,\widetilde{\omega}) \, \frac{\partial \overline{\varphi}}{\partial x_k}(t,x) \, e_k \big] \!
 \cdot \! \big[ {\big( \nabla - 2i\pi\frac{\theta^\ast}{\varepsilon} \big)} \overline{\Psi_\varepsilon}(x,\widetilde{\omega},\theta^\ast)	\big]
\\[5pt]
& + A (\Phi^{-1}{\big( \frac{x}{\varepsilon},\widetilde{\omega} \big)},\widetilde{\omega} )
{\big[ {\big( \nabla + 2i\pi\frac{\theta^\ast}{\varepsilon} \big)} 
( v_\varepsilon(t,x,\widetilde{\omega}) \, \frac{\partial \overline{\varphi}}{\partial x_k}(t,x))} \big] \!
\cdot \! \big[ {e_k} \overline{\Psi_\varepsilon}(x,\widetilde{\omega},\theta^\ast)\big],
\end{aligned}
$$
$$
\begin{aligned}
& I_{2,2}^{\varepsilon,k}(t,x,\widetilde{\omega}):=
 \frac{1}{n}A {(\Phi^{-1}{\big( \frac{x}{\varepsilon},\widetilde{\omega} \big)},\widetilde{\omega})} 
{\big[ {\big( \nabla + 2i\pi\frac{\theta^\ast}{\varepsilon} \big)} (v_\varepsilon(t,x,\widetilde{\omega})  \overline{\varphi}(t,x)) \big]}
\\[5pt]
& \quad\quad  \cdot \big[ {\big( \nabla - 2i\pi\frac{\theta^\ast}{\varepsilon} \big)} \overline{\Psi_\varepsilon}(x,\widetilde{\omega},\theta^\ast) \big]
\\[5pt]
& \quad \quad - A (\Phi^{-1}{\big( \frac{x}{\varepsilon},\widetilde{\omega} \big)},\widetilde{\omega} )
{\big[ v_\varepsilon(t,x,\widetilde{\omega}) \, \nabla \frac{\partial \overline{\varphi}}{\partial x_k}(t,x) \big]} 
\cdot {\big[ e_k \, \overline{\Psi_\varepsilon}(x,\widetilde{\omega},\theta^\ast) \big]}, 
\end{aligned}
$$
and 
$$
\begin{aligned}
&I_{2,3}^{\varepsilon,k}(t,x,\widetilde{\omega}):= A(\Phi^{-1}{\big( \frac{x}{\varepsilon},\widetilde{\omega} \big)},\widetilde{\omega} )
\big[ {\left( \varepsilon \nabla + 2i\pi\theta^\ast \right)} v_\varepsilon(t,x,\widetilde{\omega}) \big]
 \\[5pt]
 &\quad  \cdot \big[  \nabla {\big( \frac{\partial \overline{\varphi}}{\partial x_k}(t,x) \big)} \, \overline{\Lambda_{k,\varepsilon}}(x,\widetilde{\omega},\theta^\ast)  \big]
 \\[5pt]
&- A(\Phi^{-1}{\left( \frac{x}{\varepsilon},\widetilde{\omega} \right)},\widetilde{\omega})
\big[ v_\varepsilon(t,x,\widetilde{\omega}) \, \nabla \frac{\partial \overline{\varphi}}{\partial x_k}(t,x) \big] \! \cdot \!
 \big[ {\left( \varepsilon \nabla - 2i\pi\theta^\ast \right)} \overline{\Lambda_{k,\varepsilon}}(x,\widetilde{\omega},\theta^\ast) \big]. 
\end{aligned}
$$
Thus integrating in $\mathbb{R}^{n+1}_T$ we recover the $I_2^{\varepsilon}$ term, that is 
\begin{eqnarray}\label{HomProc1}
&&I_2^{\varepsilon}=\iint_{\mathbb{R}^{n+1}_T} A {\left(\Phi^{-1}{\left( \frac{x}{\varepsilon},\widetilde{\omega} \right)},\widetilde{\omega} \right)} \nabla u_\varepsilon(t,x,\widetilde{\omega}) \cdot \overline{\nabla Z_\varepsilon}(t,x,\widetilde{\omega}) \, dx \, dt\nonumber\\
&&\qquad\qquad=\sum_{k=1}^n\iint_{\mathbb{R}^{n+1}_T}\left( I_{2,1}^{\varepsilon,k}+I_{2,2}^{\varepsilon,k}+I_{2,3}^{\varepsilon,k}\right)(t,x,\widetilde{\omega})\,dx\,dt. 
\end{eqnarray}

Now, 
we intend to simplify the expression of $I_2^{\varepsilon}$ to a more comely one. For this aim, we shall take 
${ v_\varepsilon(t,\cdot,\widetilde{\omega}) \, \overline{\varphi}(t,\cdot) }$, $t \in (0,T)$, as a test function.
Then, we obtain
$$
\begin{aligned}
& \int_{\mathbb{R}^n} \!\! A(\Phi^{-1}{\big( \frac{x}{\varepsilon},\widetilde{\omega} \big)},\widetilde{\omega})
 [ \big( \nabla \!+ 2i\pi \frac{\theta^\ast}{\varepsilon} \big) {(v_\varepsilon(t,x,\widetilde{\omega}) \, \overline{\varphi})}] \! \cdot \! 
[ \big( \nabla \! - 2i\pi \frac{\theta^\ast}{\varepsilon} \big) \overline{\Psi_\varepsilon}(x,\widetilde{\omega},\theta^\ast)] dx
\\[5pt] 
& \quad = \frac{\lambda(\theta^\ast)}{\varepsilon^2}\int_{\mathbb{R}^n} {(v_\varepsilon(t,x,\widetilde{\omega}) \, 
\overline{\varphi}(t,x))} \, \overline{\Psi_\varepsilon}(x,\widetilde{\omega},\theta^\ast) \, dx 
\\[5pt]
& \quad - \frac{1}{\varepsilon^2}\int_{\mathbb{R}^n} V(\Phi^{-1}{\left( \frac{x}{\varepsilon},\widetilde{\omega} \right)},\widetilde{\omega}) \, 
{(v_\varepsilon(t,x,\widetilde{\omega}) \, \overline{\varphi}(t,x))} \, \overline{\Psi_\varepsilon}(x,\widetilde{\omega},\theta^\ast) \, dx. 
\end{aligned}
$$
Therefore, comparing with $I_{2,2}^{\varepsilon,k}(t,x,\widetilde{\omega})$ term obtained before, we have 
\begin{equation}
\label{HomProc2}
\begin{aligned}
&\iint_{\mathbb{R}^{n+1}_T}I_{2,2}^{\varepsilon,k}(t,x,\widetilde{\omega})\,dx\,dt
\\
& = \frac{\lambda(\theta^\ast)}{n\varepsilon^2}\iint_{\mathbb{R}^{n+1}_T} {(v_\varepsilon(t,x,\widetilde{\omega}) \, \overline{\varphi}(t,x))} \, \overline{\Psi_\varepsilon}(x,\widetilde{\omega},\theta^\ast) \, dx\,dt 
\\
& - \frac{1}{n\varepsilon^2}\iint_{\mathbb{R}^{n+1}_T} V {(\Phi^{-1}{\left( \frac{x}{\varepsilon},\widetilde{\omega} \right)},\widetilde{\omega})} \, 
(v_\varepsilon(t,x,\widetilde{\omega}) \, \overline{\varphi}(t,x)) \, \overline{\Psi_\varepsilon}(x,\widetilde{\omega},\theta^\ast) \, dx\,dt
\\
& - \iint_{\mathbb{R}^{n+1}_T} \!\!\! A(\Phi^{-1}{\left( \frac{x}{\varepsilon},\widetilde{\omega} \right)},\widetilde{\omega}) {[ v_\varepsilon(t,x,\widetilde{\omega}) \, 
\nabla \frac{\partial \overline{\varphi}}{\partial x_k}(t,x) ]} \cdot {[ e_k \, \overline{\Psi_\varepsilon}(x,\widetilde{\omega},\theta^\ast)]}\,dx\,dt. 
\end{aligned}
\end{equation}
Analogously, taking ${ v_\varepsilon(t,\cdot,\widetilde{\omega}) \, \displaystyle\frac{\partial \overline{\varphi}}{\partial x_k}(t,\cdot) }$, $t \in (0,T)$, 
with ${ k\in\{1,\ldots,n\} }$ as a test function,
taking into account that 
$\nabla_{\! \theta} \lambda(\theta^\ast)= 0$ and comparing this
expression with $I_{2,1}^{\varepsilon,k}(t,x,\widetilde{\omega})$, we deduce that
\begin{eqnarray}
\label{HomProc3}
&& \iint_{\mathbb{R}^{n+1}_T}I_{2,1}^{\varepsilon,k}(t,x,\widetilde{\omega})\,dx\,dt\nonumber
\\
&&\quad= \frac{\lambda(\theta^\ast)}{\varepsilon}\iint_{\mathbb{R}^{n+1}_T} 
{( v_\varepsilon(t,x,\widetilde{\omega}) \, \frac{\partial \overline{\varphi}}{\partial x_k}(t,x))} \, \overline{\Lambda_{k,\varepsilon}}(x,\widetilde{\omega},\theta^\ast) \, dx\,dt
\\
&&\quad - \frac{1}{\varepsilon}\iint_{\mathbb{R}^{n+1}_T} V {(\Phi^{-1}{\left( \frac{x}{\varepsilon},\widetilde{\omega} \right)},\widetilde{\omega})} \, 
{( v_\varepsilon(t,x,\widetilde{\omega}) \, \frac{\partial \overline{\varphi}}{\partial x_k}(t,x))} \, 
\overline{\Lambda_{k,\varepsilon}}(x,\widetilde{\omega},\theta^\ast) \, dx\,dt.\nonumber
\end{eqnarray}
Therefore, summing equations \eqref{HomProc2}, \eqref{HomProc3}, we arrive at 
\begin{eqnarray}
\label{HomProc4}
&&\sum_{k=1}^n\iint_{\mathbb{R}^{n+1}_T}\big(I_{2,1}^{\varepsilon,k}+I_{2,2}^{\varepsilon,k}\big)(t,x,\widetilde{\omega})\,dxdt\nonumber
\\
&&\quad \!\!={\frac{\lambda(\theta^\ast)}{\varepsilon^2} \iint_{\mathbb{R}^{n+1}_T} \!\!\! v_\varepsilon(t,x,\widetilde{\omega}) \,
 \overline{\varphi(t,x) \, \Psi_\varepsilon(x,\widetilde{\omega},\theta^\ast) + \varepsilon \! \sum_{k=1}^n\frac{\partial \varphi}{\partial x_k}(t,x) \, 
 \Lambda_{k,\varepsilon}(x,\widetilde{\omega},\theta^\ast)} \, dx dt }\nonumber
 \\[4pt]
&& \quad- \frac{1}{\varepsilon^2} \iint_{\mathbb{R}^{n+1}_T} \! V {\left(\Phi^{-1}{\left( \frac{x}{\varepsilon},\widetilde{\omega} \right)},\widetilde{\omega} \right)} \, 
v_\varepsilon(t,x,\widetilde{\omega})
\\
& & \hspace{3cm} \times \,  \overline{\varphi(t,x) \, \Psi_\varepsilon(x,\widetilde{\omega},\theta^\ast) 
+ \varepsilon \sum_{k=1}^n\frac{\partial \varphi}{\partial x_k}(t,x) \, \Lambda_{k,\varepsilon}(x,\widetilde{\omega},\theta^\ast)} \, dx dt \nonumber
\\[4pt]
& & \quad - \sum_{k=1}^n \iint_{\mathbb{R}^{n+1}_T} \! A {(\Phi^{-1}{\left( \frac{x}{\varepsilon},\widetilde{\omega} \right)},\widetilde{\omega})} 
{[ v_\varepsilon(t,x,\widetilde{\omega}) \, \nabla \frac{\partial \overline{\varphi}}{\partial x_k}(t,x)]} \! \cdot \! {[ e_k \, \overline{\Psi_\varepsilon}(x,\widetilde{\omega},\theta^\ast)]} \, dx dt.
\nonumber
\end{eqnarray}
Moreover, expressing the $I_3^{\varepsilon}$ term as  
$$
\begin{aligned}
& I_3^{\varepsilon}= \iint_{\mathbb{R}^{n+1}_T}\frac{1}{\varepsilon^2}\, V {(\Phi^{-1}{\left( \frac{x}{\varepsilon},\widetilde{\omega} \right)},\widetilde{\omega})}\,
v_{\ve} \, 
\overline{\varphi(t,x) \, \Psi_\varepsilon(x,\widetilde{\omega},\theta^\ast)+\varepsilon  \sum_{k=1}^n \frac{\partial \varphi}{\partial x_k}(t,x)
\Lambda_{k,\varepsilon}}\,dx\,dt
\\
&\quad+\iint_{\mathbb{R}^{n+1}_T}U {( \Phi^{-1}{\left( \frac{x}{\varepsilon},\widetilde{\omega} \right)},\widetilde{\omega})}\, 
v_\varepsilon \, 
\overline{ \varphi(t,x) \, \Psi_\varepsilon(x,\widetilde{\omega},\theta^\ast)+\varepsilon  \sum_{k=1}^n \frac{\partial \varphi}{\partial x_k}(t,x)
\Lambda_{k,\varepsilon}}\,dx\,dt,
\end{aligned}
$$
adding with \eqref{HomProc4} and $I_1^{\varepsilon}$, we obtain
\begin{equation}
\label{HomProc5}
\begin{aligned}
&I_1^{\varepsilon}+\sum_{k=1}^n\iint_{\mathbb{R}^{n+1}_T}\Big(I_{2,1}^{\varepsilon,k}+I_{2,2}^{\varepsilon,k}\Big)(t,x,\widetilde{\omega})\,dx\,dt+I_3^{\varepsilon}\nonumber
\\[5pt]
&={ i \int_{\mathbb{R}^n} \!\!\! v_\varepsilon^0(x,\widetilde{\omega}) \, \overline{ \varphi (0,x) \, \Psi_\varepsilon(x,\widetilde{\omega},\theta^\ast) } dx 
- i \! \iint_{\mathbb{R}^{n+1}_T}\!\! \!\! v_\varepsilon(t,x,\widetilde{\omega}) \, \overline{ \frac{\partial \varphi}{\partial t}(t,x) \, \Psi_\varepsilon(x,\widetilde{\omega},\theta^\ast) } }\nonumber
\\[5pt]
&- \sum_{k,\ell=1}^n\iint_{\mathbb{R}^{n+1}_T} \!\!\!
{ v_\varepsilon(t,x,\widetilde{\omega}) \, e_\ell \, \frac{\partial^2 \overline{\varphi}}{\partial x_\ell \, \partial x_k}(t,x) }
\cdot \overline{ A {(\Phi^{-1}{( \frac{x}{\varepsilon},\widetilde{\omega} )},\widetilde{\omega})} { \; e_k \Psi_\varepsilon (x,\widetilde{\omega},\theta^\ast)} } 
\,dx dt\nonumber
\\[5pt]
&+\iint_{\mathbb{R}^{n+1}_T}U {(\Phi^{-1}{\left( \frac{x}{\varepsilon},\widetilde{\omega} \right)},\widetilde{\omega})}\, 
v_\varepsilon(t,x,\widetilde{\omega})\,\overline{ \varphi(t,x) \, \Psi_\varepsilon(x,\widetilde{\omega},\theta^\ast)}\,dx dt +\, \mathrm{O}(\varepsilon).
\end{aligned}
\end{equation}
For $\ve=\ve'(\widetilde{\omega})$ and due to Step 2, i.e. 
$
v_{\varepsilon^\prime}(t,x,\widetilde{\omega}) \; \xrightharpoonup[\varepsilon^\prime \to 0]{2-{\rm s}}\; v_{\widetilde{\omega}}(t,x) \, \Psi(z,\omega, \theta^\ast),
$
we obtain after letting $\ve'\to 0$, from the previous equation
\begin{equation}
\label{HomProc5}
\begin{aligned}
&\lim_{\ve'\to 0}\Big(I_1^{\varepsilon'}+\sum_{k=1}^n\iint_{\mathbb{R}^{n+1}_T}\Big(I_{2,1}^{\varepsilon',k}
+I_{2,2}^{\varepsilon',k}\Big)(t,x,\widetilde{\omega})\,dx\,dt+I_3^{\varepsilon'}\Big)
\\
&= i \int_{\mathbb{R}^n} {\left( \int_\Omega \int_{\Phi([0,1)^n, \omega)} {\left\vert \Psi(z,\omega,\theta^\ast) \right\vert}^2 dz \, d\mathbb{P}  \right)} v^0(x) \, 
\overline{\varphi}(0,x) \, dx
\\
& - i \iint_{\mathbb{R}^{n+1}_T} \!\! {( \int_\Omega \int_{\Phi([0,1)^n, \omega)} 
{\left\vert \Psi(z,\omega,\theta^\ast) \right\vert}^2 dz \, d\mathbb{P})}  \, v_{\widetilde{\omega}}(t,x) \, \frac{\partial \overline{\varphi}}{\partial t} (t,x) \, dx \, dt 
\\
& - \sum_{k,\ell=1}^n \iint_{\mathbb{R}^{n+1}_T} \!\! {( \int_\Omega \int_{\Phi([0,1)^n, \omega)} \!\! 
A{\left( \Phi^{-1}(z,\omega),\omega \right)} \, {( e_\ell \, \Psi)} \cdot {( e_k \, \overline{\Psi})} \, dz \, d\mathbb{P})}
\\
&  \quad \times \, v_{\widetilde{\omega}}(t,x) \, \frac{\partial^2 \overline{\varphi}}{\partial x_\ell \, \partial x_k}(t,x) \, dx \, dt 
\\
& +\iint_{\mathbb{R}^{n+1}_T} \!\! {( \int_\Omega \int_{\Phi([0,1)^n, \omega)} \!\!\!\!  U {\left( \Phi^{-1}(z,\omega),\omega \right)} 
{\vert \Psi \vert}^2 dz \, d\mathbb{P})} \, v_{\widetilde{\omega}}(t,x) \, \overline{\varphi}(t,x) \, dx \, dt.
\end{aligned}
\end{equation}
Proceeding in the same way with respect to the term $I_{2,3}^{\varepsilon,k}(t,x,\widetilde{\omega})$, we obtain
\begin{eqnarray}
\label{HomProc6}
&&\lim_{\ve'\to 0} \sum_{k=1}^n\iint_{\mathbb{R}^{n+1}_T}I_{2,3}^{\varepsilon',k}(t,x,\widetilde{\omega})\,dx\,dt\nonumber\\
&&\quad =\sum_{k,\ell=1}^n \iint_{\mathbb{R}^{n+1}_T} \!\! \Big( \int_\Omega \int_{\Phi([0,1)^n, \omega)} \!\! A {\left( \Phi^{-1}(z,\omega),\omega \right)} \, 
{\left( {\left( \nabla_{\!\! z} + 2i\pi\theta^\ast \right)} \Psi(z,\omega,\theta^\ast) \right)} \nonumber
\\
&&\hspace{2cm} \cdot {\left(  e_\ell \, \overline{\Lambda_k}(z,\omega,\theta^\ast)  \right)} \, dz \, d\mathbb{P}\Big)
v_{\widetilde{\omega}}(t,x) \, \frac{\partial^2 \overline{\varphi}}{\partial x_\ell \, \partial x_k}(t,x) \, dx \, dt\nonumber
\\
&&\quad-\sum_{k,\ell=1}^n \iint_{\mathbb{R}^{n+1}_T} \!\! \Big( \int_\Omega \int_{\Phi([0,1)^n, \omega)} \!\! A {\left( \Phi^{-1}(z,\omega),\omega \right)}
\, {\left( e_\ell \, \Psi(z,\omega,\theta^\ast) \right)}\\
&&\hspace{2cm} \cdot {\left( \nabla_{\!\! z} - 2i\pi\theta^\ast \right)} \overline{\Lambda_k}(z,\omega,\theta^\ast) \, dz \, d\mathbb{P}\Big)
v_{\widetilde{\omega}}(t,x) \, \frac{\partial^2 \overline{\varphi}}{\partial x_\ell \, \partial x_k}(t,x) \, dx \, dt. \nonumber
\end{eqnarray}
Therefore, since $I_1^{\varepsilon'}+I_2^{\varepsilon'}+I_3^{\varepsilon'}=0$, (see \eqref{676745459023v}), 
combining the two last equations 
we conclude that, the function $v_{\widetilde{\omega}}$ is a distribution solution of the following homogenized Schr\"odinger equation
\begin{equation}
\label{567yt65trftdfxxzxzzxcvbn}
\left\{
\begin{array}{c}
	i\displaystyle\frac{\partial v_{\widetilde{\omega}}}{\partial t}(t,x) - {\rm div} {\big( B^\ast \nabla v_{\widetilde{\omega}}(t,x) \big)} 
	+ U^\ast v_{\widetilde{\omega}}(t,x)= 0, \;\, (t,x) \in \mathbb{R}^{n+1}_T, 
	\\ [5pt]
	v_{\widetilde{\omega}}(0,x)=v^0(x), \;\, x\in\mathbb{R}^n,
\end{array}
\right.
\end{equation}
where the effective tensor 
\begin{eqnarray}
\label{7358586tygfjdshfbvvcc}
&&B_{k,\ell}^{\ast}= \frac{1}{c_\psi} { \int_\Omega \int_{\Phi\left([0,1)^n, \omega\right)}\!\!\!  \big\{A {\left( \Phi^{-1}(z,\omega),\omega \right)} \, {\left( e_\ell \, \Psi(z,\omega,\theta^\ast) \right)}}
\cdot {\left( e_k \, \overline{\Psi}(z,\omega,\theta^\ast) \right)}\nonumber\\
&&\qquad\qquad+A {\left( \Phi^{-1}(z,\omega),\omega \right)} \, {\left( e_\ell \, \Psi(z,\omega,\theta^\ast) \right)}
\cdot {\left( (\nabla_{\!\! z} - 2i\pi\theta^\ast) \overline{\Lambda_k}(z,\omega,\theta^\ast) \right)}\nonumber\\
&&\qquad\qquad\qquad-A {\left( \Phi^{-1}(z,\omega),\omega \right)} \, {\Big( ( \nabla_{\!\! z} + 2i\pi\theta^\ast) \Psi(z,\omega,\theta^\ast) \Big)}\nonumber\\
&&\hspace{6cm}\cdot {\left(  e_\ell \, \overline{\Lambda_k}(z,\omega,\theta^\ast)  \right)}\big\} \, dz \, d\mathbb{P}(\omega),
\end{eqnarray}
for ${ k, \ell \in \{1,\ldots,n\} }$,
and the effective potential 
\begin{equation*}
U^\ast = c_\psi^{-1} \int_\Omega \int_{\Phi([0,1)^n, \omega)} U {\left( \Phi^{-1}(z,\omega), \omega \right)} {\vert \Psi (z,\omega,\theta^\ast)\vert}^2 dz \, d\mathbb{P}(\omega)
\end{equation*}
with 
$$
\begin{aligned}
	c_\psi= \!\! \int_\Omega \int_{\Phi([0,1)^n, \omega)} &\!\!\! {\vert \Psi(z,\omega, \theta^\ast) \vert}^2 dz \, d\mathbb{P}(\omega)
	\\[5pt]
	&\equiv \int_\Omega \int_{\Phi([0,1)^n, \omega)} \!\!\! {\vert \psi{( \Phi^{-1}{( \frac{x}{\varepsilon},\omega)},\omega)} \vert}^2 dz \, d\mathbb{P}(\omega).
\end{aligned}
$$
Moreover, we are allowed to change the tensor $B^\ast$ in the equation \eqref{567yt65trftdfxxzxzzxcvbn} by the 
corresponding symmetric part of it, that is 
$
			A^\ast = \big(B^\ast + (B^\ast)^t\big)/ 2.
$

\medskip
4.({\it\bf The Form of the Matrix $A^{\ast}$.}) Now, we show that the homogenized tensor $A^{\ast}$ is a real value matrix, and it coincides with the hessian matrix of 
the function ${ \theta \mapsto \lambda(\theta) }$ in the point $\theta^{\ast}$. In fact, due to ${ \nabla_{\!\! \theta} \lambda (\theta^\ast) = 0 }$
we can write
\begin{eqnarray}\label{968r6f7tyudstfyusgdjsdxxxzxzxzx}
&&\quad\frac{1}{4\pi^2} \frac{\partial^2 \lambda(\theta^\ast)}{\partial \theta_\ell \, \partial \theta_k}\,c_\psi\nonumber
\\
&&\qquad\qquad = \int_\Omega \int_{\Phi ([0,1)^n, \omega)}\big\{A {\left( \Phi^{-1}(z,\omega),\omega \right)}{\left( e_\ell \, \Psi(z,\omega,\theta^\ast) \right)} \cdot
{\left( e_k \, \overline{\Psi}(z,\omega,\theta^\ast) \right)}\nonumber\\
&&\qquad\qquad\qquad+A {\left( \Phi^{-1}(z,\omega),\omega \right)} \, {\left( e_\ell \, \Psi(z,\omega,\theta^\ast) \right)} \cdot {\left( (\nabla_{\!\! z} - 2i\pi\theta^\ast) \overline{\Lambda_k}(z,\omega,\theta^\ast) \right)}\nonumber\\
&&\qquad\qquad\qquad-A {\left( \Phi^{-1}(z,\omega),\omega \right)} \, {\left[ ( \nabla_{\!\! z} + 2i\pi\theta^\ast) \Psi(z,\omega,\theta^\ast) \right]} \cdot {\left(  e_\ell \, \overline{\Lambda_k}(z,\omega,\theta^\ast)  \right)}\nonumber\\
&&\qquad\qquad\qquad+ 
A {\left( \Phi^{-1}(z,\omega),\omega \right)} \, {\left( e_k \, \Psi(z,\omega,\theta^\ast) \right)} \cdot {\left( e_\ell \, \overline{\Psi}(z,\omega,\theta^\ast) \right)}\nonumber\\
&&\qquad\qquad\qquad+
A {\left( \Phi^{-1}(z,\omega),\omega \right)} \, {\left( e_k \, \Psi(z,\omega,\theta^\ast) \right)} \cdot {\left( (\nabla_{\!\! z} - 2i\pi\theta^\ast) \overline{\Lambda_\ell}(z,\omega,\theta^\ast) \right)}\nonumber\\
&&\qquad\qquad\qquad-
A {\left( \Phi^{-1}(z,\omega),\omega \right)} \, {\left( (\nabla_{\!\! z} + 2i\pi\theta^\ast) \Psi(z,\omega,\theta^\ast) \right)}\\
&&\hspace{8.0cm} \cdot {\left(  e_k \, \overline{\Lambda_\ell}(z,\omega,\theta^\ast)  \right)}\nonumber\big\}\, dz \, d\mathbb{P}(\omega),
\end{eqnarray}
from which we obtain 
\begin{equation*}
	A^\ast = \frac{1}{8\pi^2} \, D^2_{\! \theta} \lambda(\theta^\ast).
\end{equation*}
Therefore, from Remark \ref{REMCOSTCOEFF} we deduce the 
well-posedness of the homogenized 	Schr\"odinger \eqref{567yt65trftdfxxzxzzxcvbn}. Hence the function ${ v_{\widetilde{\omega}} \in L^2(\mathbb{R}^{n+1}_T) }$ 
does not depend on ${ \widetilde{\omega} \in {\Omega} }$. Moreover, denoting by $v$ the unique solution of the problem~ \eqref{567yt65trftdfxxzxzzxcvbn}, 
we have that the sequence ${ \{v_\varepsilon(t,x,\widetilde{\omega})\}_{\varepsilon > 0} \subset L^2(\mathbb{R}^{n+1}_T) }$ $\Phi_{\omega}-$two-scale converges to the function 
$$
     v(t,x) \, \Psi(z,\omega,\theta^\ast) \equiv v(t,x) \, \psi {\left(\Phi^{-1}(z,\omega),\omega,\theta^\ast \right)}.
$$

\medskip
5.({\it\bf A Corrector-type Result.}) 
Finally, we show the following corrector type result, that is, for a.e. ${ \widetilde{\omega} \in \Omega }$
$$
\lim_{\varepsilon \to 0} \iint_{\mathbb{R}^{n+1}_T} \big|v_\varepsilon (t,x,\widetilde{\omega}) - v(t,x) \, 
\psi{\left( \Phi^{-1} {\left(\frac{x}{\varepsilon},\widetilde{\omega} \right)}, \widetilde{\omega}, \theta^\ast \right)} \big|^2 dx \, dt= 0.
$$
We begin by the simple observation 
\begin{equation}
\label{86576567tjhghjgnbmnb}
\begin{aligned}
&\iint_{\mathbb{R}^{n+1}_T} | v_\varepsilon (t,x,\widetilde{\omega}) - v(t,x) \, 
\psi{\left( \Phi^{-1} {\left(\frac{x}{\varepsilon},\widetilde{\omega} \right)}, \widetilde{\omega}, \theta^\ast \right)} |^2 dx dt 
\\
& \quad = \iint_{\mathbb{R}^{n+1}_T} {\left\vert v_\varepsilon (t,x,\widetilde{\omega}) \right\vert}^2 dx \, dt  
\\
& \quad - \iint_{\mathbb{R}^{n+1}_T} 
v_\varepsilon(t,x,\widetilde{\omega}) \, \overline{ v(t,x) \, \psi{\left( \Phi^{-1} {\left(\frac{x}{\varepsilon},\widetilde{\omega} \right)}, \widetilde{\omega}, \theta^\ast \right)} } \, dx dt  
\\
& \quad - \iint_{\mathbb{R}^{n+1}_T} \overline{v_\varepsilon(t,x,\widetilde{\omega})} \, v(t,x) \, \psi{\left( \Phi^{-1} {\left(\frac{x}{\varepsilon},\widetilde{\omega} \right)}, \widetilde{\omega}, \theta^\ast \right)} \, dx \, dt  
\\
& \quad + \iint_{\mathbb{R}^{n+1}_T} {\left\vert v(t,x) \, \psi{\left( \Phi^{-1} {\left(\frac{x}{\varepsilon},\widetilde{\omega} \right)}, \widetilde{\omega}, \theta^\ast \right)} \right\vert}^2 dx \, dt.
\end{aligned}
\end{equation}
From Lemma~\ref{63457rf2wertgh} we see that, the first integral of the right hand side of the above equation satisfies,
for all ${ t\in [0,T] }$ and a.e. ${ \widetilde{\omega} \in \Omega }$
		\begin{eqnarray*}
			\int_{\mathbb{R}^n} {\left\vert v_\varepsilon (t,x,\widetilde{\omega}) \right\vert}^2 dx & = & \int_{\mathbb{R}^n} {\left\vert u_\varepsilon (t,x,\widetilde{\omega}) \right\vert}^2 dx 
			= \;\, \int_{\mathbb{R}^n} {\left\vert v_\varepsilon^0 (x,\widetilde{\omega}) \right\vert}^2 dx \\
			& = & \int_{\mathbb{R}^n} {\left\vert v^0(x) \, \psi{\left( \Phi^{-1} {\left(\frac{x}{\varepsilon},\widetilde{\omega} \right)}, \widetilde{\omega}, \theta^\ast \right)} \right\vert}^2 dx.
		\end{eqnarray*}
Using the elliptic regularity theory (see E. De Giorgi \cite{Giorgi}, G. Stampacchia \cite{Stampacchia}), 
it follows that $\psi(\theta) \in L^\infty(\mathbb{R}^n; L^2(\Omega))$ and we can apply the Ergodic Theorem to obtain
$$
\begin{aligned}
\lim_{\varepsilon \to 0} & \iint_{\mathbb{R}^{n+1}_T} {\left\vert v_\varepsilon (t,x,\widetilde{\omega}) \right\vert}^2 dx \, dt 
\\
& = \lim_{\varepsilon \to 0} \iint_{\mathbb{R}^{n+1}_T} {\left\vert v^0(x) \right\vert}^2 {\left\vert \psi{\left( \Phi^{-1} {\left(\frac{x}{\varepsilon},\widetilde{\omega} \right)}, \widetilde{\omega}, \theta^\ast \right)} \right\vert}^2 dx dt 
\\
& = c_\Phi^{-1} \iint_{\mathbb{R}^{n+1}_T} \! \int_\Omega \int_{\Phi([0,1)^n, \omega)} {\left\vert v^0(x) \,
\psi{\left( \Phi^{-1}(z,\omega),\omega,\theta^\ast \right)} \right\vert}^2  dz \, d\mathbb{P} \, dx dt. 
\end{aligned}
$$
Similarly, we have 	 
\begin{multline*}
    \lim_{\varepsilon \to 0} \iint_{\mathbb{R}^{n+1}_T} {\left\vert v(t,x) \, \psi{\left( \Phi^{-1} {\left(\frac{x}{\varepsilon},\widetilde{\omega} \right)}, \widetilde{\omega}, \theta^\ast \right)} \right\vert}^2 dx dt
\\
    = c_\Phi^{-1} \iint_{\mathbb{R}^{n+1}_T} \! \int_\Omega \int_{\Phi([0,1)^n, \omega)} {\left\vert v(t,x) \, \psi{\left( \Phi^{-1}(z,\omega),\omega,\theta^\ast \right)} \right\vert}^2  dz \, d\mathbb{P} \, dx dt.
\end{multline*}
Moreover, seeing that for a.e. ${ \widetilde{\omega} \in \Omega }$	
$$
\begin{aligned}
\lim_{\varepsilon \to 0} & \iint_{\mathbb{R}^{n+1}_T}  v_\varepsilon(t,x,\widetilde{\omega}) \, \overline{ v(t,x) \, \psi{\left( \Phi^{-1} {\left(\frac{x}{\varepsilon},\widetilde{\omega} \right)}, \widetilde{\omega}, \theta^\ast \right)} } \, dx dt
\\
& = c_\Phi^{-1} \iint_{\mathbb{R}^{n+1}_T} \! \int_\Omega \int_{\Phi([0,1)^n, \omega)} 
	\!\!\! v(t,x) \, \psi{\left( \Phi^{-1}(z,\omega),\omega,\theta^\ast \right)} \, 
\\
& \qquad \qquad \qquad \qquad  \times \overline{v(t,x) \, \psi{\left( \Phi^{-1}(z,\omega),\omega,\theta^\ast \right)}} \, dz \, d\mathbb{P} \, dx dt,
\end{aligned}
$$
we can make Â­${ \varepsilon \to 0 }$ in the equation~\eqref{86576567tjhghjgnbmnb} to find 
\begin{eqnarray*}
&&\lim_{\varepsilon \to 0} \iint_{\mathbb{R}^{n+1}_T} {\left\vert v_\varepsilon (t,x,\widetilde{\omega}) - v(t,x) \, \psi{\left( \Phi^{-1} {\left(\frac{x}{\varepsilon},\widetilde{\omega} \right)}, \widetilde{\omega}, \theta^\ast \right)} \right\vert}^2 dx dt 
\\
&&\qquad=c_\Phi^{-1}{\big( \int_\Omega \int_{\Phi([0,1)^n, \omega)} {\left\vert \psi{\left( \Phi^{-1}(z,\omega),\omega,\theta^\ast \right)} \right\vert}^2  dz \, d\mathbb{P}(\omega) \big)}
\\
&&\qquad\qquad\qquad\qquad \times  \big({\iint_{\mathbb{R}^{n+1}_T} {\left\vert v^0(x) \right\vert}^2 dx \, dt}-{\iint_{\mathbb{R}^{n+1}_T} {\left\vert v(t,x) \right\vert}^2 dx \, dt}\big),
\end{eqnarray*}
for a.e. ${ \widetilde{\omega} \in \Omega }$. Therefore, using the energy conservation of 
the homogenized Schr\"odinger equation~\eqref{HomSchEqu}, that is, for all ${ t\in [0,T] }$
		\begin{equation*}
			\int_{\mathbb{R}^n} {\left\vert v(t,x) \right\vert}^2 dx = \int_{\mathbb{R}^n} {\left\vert v^0(x) \right\vert}^2 dx,
		\end{equation*}
we obtain that, for a.e. ${ \widetilde{\omega} \in \Omega }$
\begin{equation*}
			\lim_{\varepsilon \to 0} \iint_{\mathbb{R}^{n+1}_T} |v_\varepsilon (t,x,\widetilde{\omega}) - v(t,x) \, 
			\psi{( \Phi^{-1} {\left(\frac{x}{\varepsilon},\widetilde{\omega} \right)}, \widetilde{\omega}, \theta^\ast)} |^2 dx dt= 0,
		\end{equation*}
completing the proof of the theorem. 		
\end{proof}

\subsection{Radom Perturbations of the Quasiperiodic Case}

In this section, we shall give a nice application of the framework introduced in 
this paper, which can be used to homogenize a model beyond 
the periodic settings considered by Allaire and Piatnitski in~\cite{AllairePiatnitski}. 
%

\medskip
Let $n,m\ge 1$ be integers numbers and $\lambda_1,\cdots,\lambda_m$ be vectors in 
$\R^n$ linearly independent over the set $\mathbb{Z}$ satisfying 
the condition
\begin{equation}
\label{7863948tyfedf}
\big\{k\in\mathbb{Z}^m;\,|k_1\lambda_1+\cdots+k_m\lambda_m|<d\big\}
\end{equation}
is a finite set for any $d>0$.  Therefore, given a stochastic deformation 
$\Phi$, we have $\mathcal{H}_\Phi \subset \! \subset \mathcal{L}_\Phi$
under condition \eqref{7863948tyfedf}, 
and it follows a solution of the Bloch's spectral cell equation, 
see \cite{VCWNJS}. 

Let $\left(\Omega_0,\mathcal{F}_0,\mathbb{P}_0\right)$ be a probability space and 
$\tau_0:\mathbb{Z}^n\times \Omega_0\to\Omega_0$ 
be a discrete ergodic dynamical system and $\R^m/{\mathbb{Z}^m}$ be the $m-$dimensional torus which can be identified with the cube $[0,1)^m$.  
For $\Omega:=\Omega_0\times [0,1)^m$, consider the following 
continuous dynamical system $T:\R^n\times \Omega\to \Omega$, defined by 
$$
T(x)(\omega_0,s):=\Big(\tau_{\left\lfloor s+Mx \right\rfloor}\omega_0,s+Mx-\left\lfloor s+Mx\right\rfloor\Big),
$$
where $M$ is the matrix $M=\Big(\lambda_i\cdot e_j{\Big)}_{i=1,j=1}^{m,n}$ and 
$\left\lfloor y\right\rfloor$ denotes the unique element in $\mathbb{Z}^m$ such that $y-\left\lfloor y\right\rfloor\in [0,1)^m$. Now, we consider $[0,1)^m-$periodic functions 
$A_{\rm per}:\R^m\to\R^{n^2},\,V_{\rm per}:\R^m\to\R$ and $U_{\rm per}:\R^m\to\R$ such that 
\begin{itemize}
\item There exists $a_0,a_1>0$ such that for all $\xi\in\R^n$ and for a.e $y\in\R^m$ we have 
$$
a_0|\xi|^2\le A_{\rm per}(y)\xi\cdot \xi\le a_1|\xi|^2.
$$
\item $V_{\rm per},\,U_{\rm per}\in L^{\infty}(\R^m)$.
\end{itemize}
Let $B_{\rm per}:\R^m\to\R^{n^2}$ be a $[0,1)^m-$periodic matrix and $\Upsilon:\R^n\times [0,1)^m\to\R^n$ be any stochastic diffeomorphism
satisfying
$$
\nabla \Upsilon (x,s)=B_{\rm per}\Big(T(x)(\omega_0,s)\Big).
$$ 
Thus, we define the following stochastic deformation $\Phi:\R^n\times \Omega\to \R^n$ by 
$\Phi(x,\omega)=\Upsilon(x,s)+{\bf X}(\omega_0)$, 
where we have used the notation $\omega$ for the pair $(\omega_0,s)\in\Omega$ and ${\bf X}:\Omega_0\to\R^n$ is a random vector.  Now, taking 
$$
   A(x,\omega):=A_{\rm per}\left(T(x)\omega\right),\,V(x,\omega):=V_{\rm per}\left(T(x)\omega\right), 
   \; U(x,\omega):= U_{\rm per}\left(T(x)\omega\right)
$$ 
in the equation~\eqref{jhjkhkjhkj765675233}, it can be seen after some computations that the spectral equation correspondent is 
\begin{equation}\label{ApHom}
			\left\{
			\begin{array}{l}
				-{\Big( {\rm div}_{\rm {QP}} + 2i\pi \theta \Big)} {\left[ A _{\rm per}{\left(\cdot \right)} {\Big( \nabla^{\rm {QP}} + 2i\pi\theta \Big)} {\Psi}_{\rm per}(\cdot) \right]}
\\ [7.5pt]
 \hspace{2.0cm}+ V_{\rm per}{\left(\cdot \right)} {\Psi}_{\rm per}(\cdot) = \lambda {\Psi}_{\rm per}(\cdot) \;\; \text{in} \,\; [0,1)^m, \\ [7.5pt]
				\hspace{1.5cm} {\Psi}_{\rm per}(\cdot) \;\;\; \psi \;\, \text{is a $[0,1)^m-$periodic function},
			\end{array}
			\right.
		\end{equation}
where the operators ${\rm div}_{\rm {QP}}$ and $\nabla^{\rm{QP}}$ are defined as 
\begin{itemize}
\item $\left(\nabla^{\rm {QP}}u_{\rm per}\right)(y):=B_{\rm per}^{-1}(y)M^{\ast}\left(\nabla u_{\rm per}\right)(y)$;
\item $\left(\rm{div}_{\rm{QP}}\,a\right)(y):=\rm{div}\left(M B_{\rm per}^{-1}(\cdot)a(\cdot)\right)(y)$.
\end{itemize}

Assume that for some $\theta^{\ast}\in\R^n$, the spectral equation~\eqref{ApHom} admits a solution 
$\big(\lambda(\theta^{\ast}),\Psi_{\rm per}(\theta^{\ast})\big)\in \R\times H^1\left([0,1)^m\right)$, such that \eqref{conds} holds.
Then, we consider the problem~\eqref{jhjkhkjhkj765675233} with new coefficients  as highlighted above and with well-prepared initial data, that is, 
$$
u_{\varepsilon}(x,\omega):=e^{2\pi i \frac{\theta^{\ast}\cdot x}{\varepsilon}}\,{\Psi}_{\rm per}\Big(T\left(\Phi^{-1}\left(\frac{x}{\varepsilon},\omega\right)\right)\omega,\theta^{\ast}\Big)
v^0(x),
$$ 
for $(x,\omega)\in \R^n\times \Omega$ and $v^0\in C^{\infty}_c(\R^n)$. Applying Theorem~\ref{876427463tggfdhgdfgkkjjlmk}, the function 
\begin{equation*}
			v_\varepsilon(t,x,\omega) := e^{ -{\left( i \frac{\lambda(\theta^\ast) t}{\varepsilon^2} + 2i\pi \frac{\theta^\ast \! \cdot x}{\varepsilon} \right)} } u_\varepsilon(t,x,\omega), \;\, (t,x) \in \mathbb{R}^{n+1}_T, \; \omega \in \Omega, 
		\end{equation*}
$\Phi_\omega-$two-scale converges strongly to ${ v(t,x) \, {\Psi}_{\rm per}\Big({T\left( \Phi^{-1}(z,\omega)\right)\omega, \theta^\ast } }\Big)$,
where 
$v \in C([0,T], L^2(\mathbb{R}^n))$ is the unique solution of the homogenized Schr\"odinger equation 
\begin{equation*}
			\left\{
			\begin{array}{c}
				i \displaystyle\frac{\partial v}{\partial t} - {\rm div} {\left( A^\ast \nabla v \right)} + U^\ast v = 0 \, , \;\, \text{em} \;\, \mathbb{R}^{n+1}_T, \\ [7,5pt]
				v(0,x) = v^0(x) \, , \;\, x\in \mathbb{R}^n,
			\end{array}
			\right.
\end{equation*}
with effective matrix  ${ A^\ast = D_\theta^2 \lambda(\theta^\ast) }$ and effective potential 
\begin{equation*}
			U^\ast =  c^{-1}_\psi \int_{[0,1)^m} U_{\rm per}{\left(y \right)}\,  {\left\vert {\Psi}_{\rm per} {\left(y, \theta^\ast \right)} \right\vert}^2 
			|\det \left(B_{\rm per}(y)\right)|\,dy,
		\end{equation*}
		where $$c_\psi = \int_{[0,1)^m} {\left\vert {\Psi}_{\rm per} {\left(y, \theta^\ast \right)} \right\vert}^2 \,|\det \left(B_{\rm per}(y)\right)|\,dy.$$
		
It is worth highlighting that this singular example encompasses the settings considered by Allaire-Piatnitski in~\cite{AllairePiatnitski}. For this, it is enough to take 
$$
   n= m,\,\lambda_j=e_j,\,\Upsilon(\cdot,s)\equiv I_{n \times n}, \; \text{and ${\bf X}(\cdot)\equiv 0$.}
$$ 

\section{\! \! \!Homogenization of Quasi-Perfect Materials} 
\label{6775765ff0090sds}

We consider in this final section an interesting context, which is the small random perturbation of the periodic setting. To begin, we recall that
the homogenization analysis (see Theorem \ref{876427463tggfdhgdfgkkjjlmk}) 
of the equation~\eqref{jhjkhkjhkj765675233} rely on the spectral study of the operator 
$L^{\Phi}(\theta)(\theta\in\R^n)$ posed in the dual space ${ \mathcal{H}^\ast }$ and with domain ­
$D(L^{\Phi}(\theta))=\mathcal{H}$, defined by
\begin{equation}\label{OperL}
\begin{array}{l}
	L^\Phi(\theta)[f] := - {\big({\rm div}_{\! z} + 2i\pi \theta \big)} {\Big[ A {\big( \Phi^{-1} (\cdot, {\cdot\cdot} ), {\cdot\cdot} \big)} {\big( \nabla_{\!\! z} + 2i\pi\theta \big)} f{\big( \Phi^{-1}(\cdot, {\cdot\cdot} ),{\cdot\cdot} \big)} \Big]} \\ [10pt]
	\hspace{4cm} + \, V{\big( \Phi^{-1} (\cdot, {\cdot\cdot} ), {\cdot\cdot} \big)} f{\big( \Phi^{-1}(\cdot, {\cdot\cdot} ), {\cdot\cdot} \big)}, 
\end{array}
\end{equation}
where $\Phi:\R^n\times\Omega\to\R^n$ is a stochastic deformation, $A:\R^n\times\Omega\to\R^{n^2}$ and $V:\R^n\times\Omega\to\R$ are stationary functions. 

\subsection{Perturbed Periodic Case: Spectral Analysis}
\label{PERTUSPECTANALY}

We shall study the spectral properties of the operator ${ L^\Phi(\theta) }$, when the diffeomorphism ${ \Phi }$ 
is a stochastic perturbation of the identity. This concept was introduced in \cite{BlancLeBrisLions2},
and well-developed by T. Andrade, W. Neves, J. Silva \cite{AndradeNevesSilva} for modelling quasi-perfect materials. 

\medskip
Let $(\Omega,\mathcal{F},\mathbb{P})$ be a probability space, 
$\tau:\mathbb{Z}^n\times\Omega\to\Omega$ a discrete 
dynamical system, and $Z$ any fixed stochastic deformation.
%
Then, we consider the concept of stochastic perturbation of the identity given by the following
\begin{definition}
\label{37285gdhddddddddddd}
Given $\eta \in (0,1)$, let $\Phi_\eta: \mathbb{R}^n \times \Omega \to \mathbb{R}^n$ be a stochastic deformation.
Then $\Phi_\eta$ is said a stochastic perturbation of the identity, when 
it can be written as  
\begin{equation}
\label{DefPertIden}
\Phi_\eta(y,\omega) = y + \eta \, Z(y,\omega) + \mathrm{O}(\eta^2), 
\end{equation}
for some stochastic deformation $Z$.  
\end{definition}
We emphasize that the equality~\eqref{DefPertIden} is understood in the 
${\rm Lip}_{\loc}(\R^n; L^2(\Omega))$
sense, i.e. 
for each bounded open subset ${ \mathcal{O} \subset \mathbb{R}^n }$, 
there exist $\delta, C > 0$, such that for all ${ \eta \in (0,\delta) }$
\begin{eqnarray*}
&&\underset{y \in \mathcal{O}}{\rm sup} \, {\left\Vert \Phi_\eta(y,\cdotp) - y - \eta Z(y,\cdotp) \right\Vert}_{L^2(\Omega)}\\
&&\qquad +\,\underset{y \in \mathcal{O}}{\rm ess \, sup} \, {\left\Vert \nabla_{\!\! y} \Phi_\eta(y,\cdotp) - I 
- \eta \, \nabla_{\!\! y}  Z(y,\cdotp) \right\Vert}_{L^2(\Omega)}
\leqslant C \, \eta^2.
\end{eqnarray*}
Moreover, after some computations, we have
\begin{equation}
\label{654367ytr6tfclmlml}
\left \{
\begin{aligned}
	\nabla_y^{-1} \Phi_{\eta}&= I-\eta\,\nabla_y Z+O(\eta^2), 
	\\[5pt]
	\det \big(\nabla_y\Phi_{\eta}\big)&= 1+\eta\, {\rm div}_yZ +O(\eta^2).
\end{aligned}
\right.
\end{equation}

Now, we consider the $[0,1)^n-$periodic functions 
$A_{\rm per}:\R^n\to\R^{n^2},\,V_{\rm per}:\R^n\to\R$ and $U_{\rm per}:\R^n\to\R$, such that 
\begin{itemize}
\item There exists $a_0,a_1>0$ such that for all $\xi\in\R^n$ and for a.e $y\in\R^n$ we have 
$$
a_0|\xi|^2\le A_{\rm per}(y)\xi\cdot \xi\le a_1|\xi|^2.
$$
\item $V_{\rm per},\,U_{\rm per}\in L^{\infty}(\R^n)$.
\end{itemize}
The following lemma is well-known and it is stated explicitly here only for reference. 

\begin{lemma}
\label{7836565etyd43tre56rt3e54redgh}
For $\theta \in \mathbb{R}^n$ and  $f \in H_{\rm per}^1([0,1)^n)$, let ${ L_{\rm per}(\theta) }$ be the operator defined by
\begin{equation}
\label{753e6735827tdygetydr5de4se45se5}
L_{\rm per}(\theta){[f]} := -({\rm div}_{\! y} + 2i\pi \theta) {\big[ A_{\rm per} (y) {(\nabla_{\!\! y} + 2i\pi\theta)} f(y) \big]} + V_{\rm per}(y) f(y),
\end{equation}
with variational formulation 
\begin{equation*}
\begin{array}{c}
\displaystyle {\left\langle L_{\rm per}(\theta){\big[ f \big]}, g \right\rangle} := \int_{[0,1)^n} A_{\rm per}(y) {\left( \nabla_{\!\! y} + 2i\pi \theta \right)} f(y) \cdot \overline{ {\left( \nabla_{\!\! y} + 2i\pi \theta \right)} g(y) } \, dy \\ [10pt]
\displaystyle \hspace{1.7cm} + \int_{[0,1)^n} V_{\rm per}(y) \,  f(y) \, \overline{ g(y) } \, dy, 
\end{array}
\end{equation*}
for ${ f,g \in H_{\rm per}^1({[0,1)^n}) }$. Then ${ L_{\rm per}(\theta) }$ has the following properties:
		\begin{enumerate}
			\item[(i)] There exist ${ \gamma_0, b_0 > 0 }$, such that ${ L_{\gamma_0} := L_{\rm per}(\theta) + {\gamma_0}I }$ satisfies
			 for all $f \in H_{\rm per}^1({[0,1)^n})$, 
			\begin{equation*}
				{\langle L_{\gamma_0} {\big[ f \big]}, f \rangle} \geq b_0 {\Vert f \Vert}_{H_{\rm per}^1({[0,1)^n})}^2.
			\end{equation*}
			\item[(ii)] The point spectrum of ${ L_{\rm per}(\theta) }$ is not empty and their eigenspaces have finite dimension, that is, the set
			\begin{equation*}
				\sigma_{\rm point} {\big( L_{\rm per}(\theta) \big)} = \{ \lambda \in \mathbb{C} \; ; \; \lambda \; \text{an eigenvalue of} \; L_{\rm per}(\theta) \}
			\end{equation*}
			is not empty and for all ${ \lambda \in \sigma_{\rm point} {\big( L_{\rm per}(\theta) \big)} }$ fixed,
			\begin{equation*}
				{\rm dim} {\big\{ f \in H^1_{\rm per}({[0,1)^n}) \; ; \; L_{\rm per}(\theta){\big[ f \big]} = \lambda f \big\}} < \infty.
			\end{equation*}
			
			\item[(iii)] Every point in ${ \sigma_{\rm point}\big( L_{\rm per}(\theta) \big) }$ is isolated. 
		\end{enumerate}
	\end{lemma}

In what follows, we are interested in the study of spectral properties of the operator ${ L^{\Phi_\eta}(\theta) }$ whose variational formulation is given by 
	\begin{equation}\label{VarFor1}
	\begin{split}
		& {\left\langle L^{\Phi_\eta}(\theta)[f], g \right\rangle} :=  \int_\Omega \int_{\Phi_\eta ([0,1)^n, \omega)} A_{\rm per} {\left( \Phi_\eta^{-1} ( z, \omega) \right)} {\big( \nabla_{\!\! z} + 2i\pi \theta \big)} f{\left( \Phi_\eta^{-1}(z,\omega),\omega \right)} \cdot \\
		& \hspace{7cm} \overline{ {\big( \nabla_{\!\! z} + 2i\pi \theta \big)} g{\left( \Phi_\eta^{-1}(z,\omega),\omega \right)} } \, dz \, d\mathbb{P}(\omega) \\
		& + \int_\Omega \int_{\Phi_\eta ([0,1)^n, \omega)} V_{\rm per} {\left( \Phi_\eta^{-1} ( z, \omega) \right)} f{\left( \Phi_\eta^{-1}(z,\omega),\omega \right)} \, \overline{ g{\left( \Phi_\eta^{-1}(z,\omega),\omega \right)} } \, dz \, d\mathbb{P}(\omega),
	\end{split}
	\end{equation}
	for ${ f,g \in \mathcal{H} }$.  As we shall see in the next theorem, some of the spectral properties of the operator ${ L^{\Phi_\eta}(\theta) }$ are inherited from the periodic case.

\begin{theorem}
\label{4087865567576ghghj}
		Let ${ \Phi_\eta }$, ${ \eta \in (0,1) }$ be a stochastic perturbation of identity and ${ \theta_0 \in \mathbb{R}^n }$. If ${ \lambda_0 }$ is an eigenvalue of ${ L_{\rm per}(\theta_0) }$ with multiplicity ${ k_0 \in \mathbb{N} }$, that is,
		\begin{equation*}
			{\rm dim} {\big\{ f \in H^1_{\rm per}([0,1)^n) \; ; \; L_{\rm per}(\theta_0){\big[ f \big]} = \lambda_0 f \big\}} = k_0,				
		\end{equation*}
		then there exist a neighbourhood ${ \mathcal{U} }$ of ${ (0,\theta_0) }$, ${ k_0 }$ real analytic functions
		\begin{equation*}
			(\eta,\theta) \in \mathcal{U} \; \mapsto \; \lambda_k(\eta,\theta) \in \mathbb{R}, \;\; k\in \{1,\ldots,k_0\},
		\end{equation*}
		and ${ k_0 }$ vector-value analytic maps 
		\begin{equation*}
			(\eta,\theta) \in \mathcal{U} \; \mapsto \; \psi_k(\eta,\theta) \in \mathcal{H} \setminus \{0\}, \;\; k\in \{1,\ldots,k_0\},
		\end{equation*}
		such that, for all ${ k\in\{1,\ldots,k_0\} }$, 
		\begin{itemize}
			\item[(i)] ${ \lambda_k(0,\theta_0) = \lambda_0 }$,
			\item[(ii)] ${ L^{\Phi_\eta}(\theta) {\big[ \psi_k(\eta,\theta) \big]} = \lambda_k(\eta,\theta) \, \psi_k(\eta,\theta) }$, ${ \forall (\eta,\theta) \in \mathcal{U} }$,
			\item[(iii)] ${ {\rm dim}{\big\{ f \in \mathcal{H} \; ; \; L^{\Phi_\eta}(\theta){\big[ f \big]}=\lambda_k(\eta,\theta) f \big\}} \leqslant k_0 }$, ${ \forall (\eta,\theta) \in \mathcal{U} }$.
		\end{itemize}
\end{theorem}

\begin{proof}
1. The aim of this step is to rewrite the operator ${ L^{\Phi_\eta}(\theta) \in \mathcal{B}(\mathcal{H},\mathcal{H}^\ast) }$, for ${ \eta\in (0,1) }$ and ${ \theta\in\mathbb{R}^n }$ 
as an expansion in the variable ${ (\eta,\theta) }$ of operators in ${ \mathcal{B}(\mathcal{H},\mathcal{H}^\ast) }$ around the point ${ (\eta,\theta)=(0,\theta_0) }$. For this, 
using the variational formulation~\eqref{VarFor1}, a change of variables, and the expansions \eqref{654367ytr6tfclmlml} we obtain 		
	\begin{equation}\label{iy87678yhghj354g}
		L^{\Phi_\eta}(\theta) = L_{\rm per}(\theta_0) + \sum_{{\vert (\alpha,\beta) \vert} = 1}^{3} ((\eta,\theta)-(0,\theta_0))^{(\alpha,\beta)}L_{(\alpha,\beta)} + \mathrm{O}(\eta^2),
	\end{equation}
	in ${ \mathcal{B}(\mathcal{H},\mathcal{H}^\ast) }$ as ${ \eta \to 0 }$, where ${ L_{(\alpha,\beta)} \in \mathcal{B}(\mathcal{H},\mathcal{H}^\ast) }$ and ${ {\vert (\alpha,\beta) \vert} = \alpha + \sum_{k=1}^n \beta_k }$.
Here, for ${ (\alpha,\beta) \in \mathbb{N} \times \mathbb{N}^n }$ and ${ \beta=(\beta_1,\ldots,\beta_n) }$, 
we are using the multi-index notation ${ ((\eta, \theta)-(0,\theta_0))^{(\alpha,\beta)}= \eta^\alpha \prod_{k=1}^n (\theta_k-\theta_{0k})^{\beta_k} }$.				
Clearly, we can consider the parameters ${ (\eta,\theta) }$ in the set $B(0,1) \times \mathbb{C}^n$.

2.  In this step, we shall modify the expansion \eqref{iy87678yhghj354g} conveniently in order to obtain an holomorphic invertible operator in the variable ${ (\eta,\theta) }$.  For this, 
remember that according to the item ${ (i) }$ in Lemma \ref{7836565etyd43tre56rt3e54redgh}, there exists $\gamma_0>0$ such that the operator ${ L_{\rm per}(\theta_0) + {\gamma_0} I }$ is invertible. Then there exists ${ \delta>0 }$ such that the expansion
\begin{equation}\label{87tyrtdfdcccdasxzsaxzsa}
			L^{\Phi_\eta}(\theta) + {\gamma_0} I= 
			L_{\rm per}(\theta_0) + {\gamma_0} I + \!\!\!\sum_{{\vert (\alpha,\beta) \vert}= 1}^{3} ((\eta,\theta)-(0,\theta_0))^{(\alpha,\beta)}L_{(\alpha,\beta)}+ \mathrm{O}(\eta^2)
\end{equation}
in ${ \mathcal{B}(\mathcal{H},\mathcal{H}^\ast) }$ as ${ \eta \to 0 }$, is invertible for all ${ (\eta,\theta) \in B(0,\delta) \times B(\theta_0,\delta) }$, since the set of invertible bounded operators ${ GL(\mathcal{H},\mathcal{H}^\ast) }$ is an open subset of ${ \mathcal{B}(\mathcal{H},\mathcal{H}^\ast) }$. Now, we denote by ${ S(\eta,\theta) }$ the inverse operator of ${ L^{\Phi_\eta}(\theta) + {\gamma_0} I }$, ${ (\eta, \theta) \in B(0,\delta) \times B(\theta_0,\delta) }$. Since the map ${ L \in GL(\mathcal{H},\mathcal{H}^\ast) \mapsto L^{-1} \in \mathcal{B}(\mathcal{H}^\ast,\mathcal{H}) }$ is continuous, the map
		\begin{equation*}
			(\eta, \theta) \in B(0,\delta) \times B(\theta_0,\delta) \mapsto S(\eta,\theta) \in \mathcal{B}(\mathcal{H}^\ast, \mathcal{H})
		\end{equation*}
		is continuous. As a consequence of this, for ${ (\widetilde{\eta}, \widetilde{\theta}) \in B(0,\delta) \times B(\theta_0,\delta) }$ fixed, the limit of
		\begin{equation*}
			\frac{S(\eta, \widetilde{\theta}) - S(\widetilde{\eta}, \widetilde{\theta})}{\eta - \widetilde{\eta}} = -S(\eta, \widetilde{\theta}) {\left[ \frac{(L^{\Phi_\eta}(\widetilde{\theta}) + {\gamma_0} I) - (L^{\Phi_{\widetilde{\eta}}}(\widetilde{\theta}) + {\gamma_0} I)}{\eta - \widetilde{\eta}} \right]} S(\widetilde{\eta}, \widetilde{\theta}),
		\end{equation*}
		as ${ \widetilde{\eta} \not= \eta \to 0 }$, exists. Thus, ${ \eta \in B(0,\delta) \mapsto S(\eta,\widetilde{\theta}) }$ is an holomorphic map. In analogy with it, for ${ j\in{\{1,\ldots,n\}} }$, we can prove that
		\begin{equation*}
			\theta_j \mapsto S(\widetilde{\eta}, \widetilde{\theta}_1, \ldots, \widetilde{\theta}_{j-1}, \theta_j, \widetilde{\theta}_{j+1}, \ldots, \widetilde{\theta}_n) 
		\end{equation*}
		is an holomorphic map. Therefore, by Osgood's Lemma, see for instance \cite{GunningRossi}, we conclude that
		\begin{equation}\label{9867967689ndyfh}
			(\eta, \theta) \in B(0,\delta) \times B(\theta_0,\delta) \mapsto S(\eta,\theta) \in \mathcal{B}(\mathcal{H}^\ast, \mathcal{H})
		\end{equation}
		is a holomorphic function.
		
3.  Finally, we are in conditions to prove items ${ (i) }$, ${ (ii) }$ and ${ (iii) }$ (the spectral analysis of the operator ${ S(\eta, \theta) }$). First, we shall note that for 
${ (\eta,\theta) }$ in a neighbourhood of ${ (0,\theta_0) }$, the map ${ (\eta, \theta) \mapsto S(\eta, \theta) }$ satisfies the assumptions of the Theorem \ref{768746hughjg576}. 
We begin recalling that the restriction operator ${ T \in \mathcal{B}(\mathcal{H}^\ast, \mathcal{H}) \mapsto T\big\vert_\mathcal{L} \in \mathcal{B}(\mathcal{L}, \mathcal{L}) }$ is continuous and it satisfies
		\begin{equation}\label{78687326tygd53tegdcx}
			{\Vert T \Vert}_{\mathcal{B}(\mathcal{L}, \mathcal{L})} \leqslant {\Vert T \Vert}_{\mathcal{B}(\mathcal{H}^\ast, \mathcal{H})} \, , \;\, \forall T \in \mathcal{B}(\mathcal{H}^\ast, \mathcal{H}).
		\end{equation}
		Then, by \eqref{9867967689ndyfh}, the map ${ (\eta, \theta) \in B(0,\delta) \times B(\theta_0,\delta) \mapsto S(\eta, \theta) \in \mathcal{B}(\mathcal{L}, \mathcal{L}) }$ is holomorphic. Since holomorphic maps are, locally, analytic maps there exists a neighbourhood ${ \mathcal{U} }$ of ${ (0,\theta_0) }$, ${ (0, \theta_0) \in \mathcal{U} \subset \mathbb{C} \times \mathbb{C}^n }$, and a family ${ \{S_{\sigma}\}_{\sigma \in \mathbb{N} \times \mathbb{N}^n} }$ contained in ${ \mathcal{B}(\mathcal{L}, \mathcal{L}) }$, such that
		\begin{equation}\label{rtfgrffcfdfdfdfdssdadssss}
			S(\eta, \theta) = S_{0} + \sum_{\substack{\sigma \in \mathbb{N} \times \mathbb{N}^n \\ {\vert \sigma \vert} \neq 0}} (\eta, \theta)^\sigma S_\sigma \, , \;\, \forall (\eta, \theta) \in \mathcal{U}.
		\end{equation}
		
			Using \eqref{87tyrtdfdcccdasxzsaxzsa} and \eqref{rtfgrffcfdfdfdfdssdadssss}, it is easy to see that 
${ S_0 = (L_{\rm per}(\theta_0) + {\gamma_0} I)^{-1} \big\vert_\mathcal{L}}$.  Notice also that ${ \mu_0 := {\left( \lambda_0+\gamma_0 \right)}^{-1} }$ is an eigenvalue of ${ S_0 }$ if and only if ${ \lambda_0 }$ is an eigenvalue of ${ L_{\rm per}(\theta_0) }$ that is
		\begin{equation*}
			g \in {\{ f \in \mathcal{L} \; ; \; S_0 {\big[ f \big]} = \mu_0 f \}} \; \Leftrightarrow \; g \in {\{ f \in \mathcal{L} \; ; \; L_{\rm per}(\theta_0) {\big[ f \big]} =\lambda_0 f \}}.
		\end{equation*}
	
\medskip
		
		The final part of the proof is a direct application of the Theorem~\ref{768746hughjg576}. Due to our assumption, $\mu_0$ is a real eigenvalue of the operator $S_0$ with 
multiplicity $k_0$. Hence, by the Theorem~\ref{768746hughjg576},  there exists a neighbourhood ${ \widetilde{\mathcal{U}} }$ of ${ (0,\theta_0) }$, with ${ \widetilde{\mathcal{U}} \subset \mathcal{U} }$ and analytic maps
		\begin{equation*}
		\begin{array}{l}
			(\eta, \theta) \in \widetilde{\mathcal{U}} \; \longmapsto \; \mu_{0 1}(\eta,\theta), \mu_{0 2}(\eta,\theta), \ldots, \mu_{0 k_0}(\eta, \theta) \in (0,\infty), \\ [5pt]
			(\eta, \theta) \in \widetilde{\mathcal{U}} \; \longmapsto \; \psi_{0 1}(\eta,\theta), \psi_{0 2}(\eta,\theta), \ldots, \psi_{0 k_0}(\eta,\theta) \in \mathcal{L}-\{0\},
		\end{array}
		\end{equation*}
		such that
		\begin{itemize}
			\item ${ \mu_{0 \ell} (0,\theta_0) = \mu_0 }$, 
			\item ${ S(\eta, \theta) {\big[ \psi_{0 \ell}(\eta, \theta) \big]} = \mu_{0 \ell}(\eta, \theta) \psi_{0 \ell}(\eta, \theta) }$, \; ${ \forall (\eta, \theta) \in \widetilde{\mathcal{U}} }$,
			\item ${ {\rm dim}{\{ f \in \mathcal{L} \; ; \; S(\eta,\theta){\big[ f \big]} = \mu_{0 \ell}(\eta,\theta) f \}}\leqslant k_0 }$, \; ${ \forall (\eta, \theta) \in \widetilde{\mathcal{U}} }$, 
		\end{itemize}		
		for all ${ \ell \in \{1, \ldots, k_0\} }$. Thus, the proof of the item ${ (i) }$ is clear.

		Using the second equality above, we obtain 
		\begin{eqnarray*}
			(L^{\Phi_\eta}(\theta) + {\gamma_0} I) {\big[ \psi_{0 \ell}(\eta, \theta) \big]} & = & \frac{1}{\mu_{0 \ell}(\eta, \theta)} (L^{\Phi_\eta}(\theta) + {\gamma_0} I){\left\{ S(\eta, \theta) {\big[ \psi_{0 \ell}(\eta, \theta) \big]} \right\}} \\ [5pt]
			& = & \frac{1}{\mu_{0 \ell}(\eta, \theta)} \psi_{0 \ell}(\eta, \theta), 
		\end{eqnarray*}
which implies that ${ L^{\Phi_\eta}(\theta) {\big[ \psi_{0 \ell}(\eta, \theta) \big]} = \lambda_{0\ell}(\eta,\theta) \psi_{0 \ell}(\eta, \theta) }$, for ${ (\eta, \theta) \in \widetilde{\mathcal{U}} }$,  $\ell\in \{1, \ldots, m_0\}$ 
and $\lambda_{0 \ell}(\eta,\theta) := [\mu_{0 \ell}(\eta, \theta)]^{-1} - {\gamma_0}$. This finish the proof of the item $(ii)$.
		
\medskip 

		Finally, note that ${ S(\eta,\theta) \big[ \mathcal{L} \big] \subset \mathcal{H}}$ and 
		\begin{equation*}
			g \! \in \! \big\{ f \in \mathcal{H} ; S(\eta,\theta) {\big[ f \big]} \! = \! \mu_{0\ell}(\eta,\theta) f \big\} \Leftrightarrow g \! \in \! \big\{ f \in \mathcal{H} \; ; \; L^{\Phi_\eta}(\theta) {\big[ f \big]} \! = \! \lambda_{0\ell}(\eta,\theta) f \big\},
		\end{equation*}
		which concludes the proof of the item ${ (iii) }$. Hence the proof is completed. 
\end{proof}

\subsection{Homogenization Analysis of the Perturbed Model}

In this section, we shall investigate in which way the stochastic perturbation of 
the identity characterize the form of the coefficients, during the 
asymptotic limit of the Schr\"odinger equation 
	\begin{multline}\label{765tdyyuty67tsss}
		\left\{
		\begin{array}{l}
			i\displaystyle\frac{\partial u_{\eta\varepsilon}}{\partial t} - {\rm div} {\bigg( A_{\rm per} {\left( \Phi_\eta^{-1} {\left( \frac{x}{\varepsilon}, \omega \right)} \right)} \nabla u_{\eta\varepsilon} \bigg)} \\ [14pt]
			+ {\bigg( \displaystyle\frac{1}{\varepsilon^2} V_{\rm per} {\left( \Phi_\eta^{-1} {\left( \displaystyle\frac{x}{\varepsilon}, \omega \right)} \right)} +  U_{\rm per} {\left( \Phi_\eta^{-1} {\left( \displaystyle\frac{x}{\varepsilon}, \omega \right)} \right)} \bigg)} u_{\eta\varepsilon} = 0 \quad \text{in} \;\, \mathbb{R}^{n+1}_T \! \times \! \Omega, \\ [14pt]
			u_{\eta\varepsilon} (0,x,\omega)=u_{\eta\varepsilon}^0(x,\omega), \;  \; (x,\omega) \in \mathbb{R}^n \! \times \! \Omega,
		\end{array}
		\right.
	\end{multline}
where ${ 0 < T < \infty }$, ${ \mathbb{R}^{n+1}_T = (0,T) \times \mathbb{R}^n }$.  The coefficients are accomplishing of the periodic functions ${ A_{\rm per}(y) }$, ${ V_{\rm per}(y) }$, ${ U_{\rm per}(y) }$ (as defined in the last subsection) with a stochastic perturbation of identity ${ \Phi_\eta }$, ${ \eta \in (0,1) }$, presenting an rate of oscillation ${ \varepsilon^{-1} }$, ${ \varepsilon>0 }$.  
The function ${ u_{\eta\varepsilon}^0(x,\omega) }$ is a well prepared initial data (see~\eqref{well-prep.I}) and this well-preparedness is triggered by natural periodic conditions
on the existence of a pair 
${ \big( \theta^\ast, \lambda_{\rm per}(\theta^\ast) \big) \in \mathbb{R}^n \times \mathbb{R} }$ such that 
	\begin{equation}\label{7t8drtys65edsrt3xcvvcxcvb}
		\begin{split}
			(i) & \;\;\, \lambda_{\rm per}(\theta^\ast) \; \text{is a simple eigenvalue of} \; L_{\rm per}(\theta^\ast), \\
			(ii) & \;\;\, \theta^\ast \; \text{is a critical point of} \; \lambda_{\rm per}(\cdot), \, \text{that is}, \nabla_{\!\! \theta} \lambda_{\rm per}(\theta^\ast)=0.
		\end{split}
	\end{equation}

	By the condition ${ (i) }$ and the Theorem \ref{4087865567576ghghj}, there exists a neighborhood ${ \mathcal{U} }$ of ${ (0,\theta^{\ast}) }$ and the analytic maps
	\begin{equation}\label{67ty3uhrjefd67tgrefdcx8ur7u}
		\begin{split}
			(i) & \;\;\, (\eta,\theta) \in \mathcal{U} \; \mapsto \; \lambda(\eta,\theta) \in \mathbb{R}, \\
			(ii) & \;\;\, (\eta,\theta) \in \mathcal{U} \; \mapsto \; \psi(\eta,\theta) \in \mathcal{H}\setminus\{0\}, 
		\end{split}
	\end{equation}
	such that ${ \lambda(0,\theta^\ast) = \lambda_{\rm per}(\theta^\ast) }$, ${ L^{\Phi_\eta}(\theta) \big[ \psi(\eta,\theta) \big] = \lambda(\eta,\theta) \, \psi(\eta,\theta) }$ and
	\begin{equation*}
		{\rm dim} \big\{ f \in \mathcal{H} \; ; \; L^{\Phi_\eta} (\theta) = \lambda(\eta,\theta) \, f \big\} = 1, \; \forall (\eta,\theta) \in \mathcal{U}.
	\end{equation*}
	
		Thus,
	\begin{equation}\label{7rter44}
		\lambda(\eta,\theta) \; \text{is a simple eigenvalue of} \; L^{\Phi_\eta}(\theta),  \forall (\eta,\theta) \in \mathcal{U}.
	\end{equation}

Additionally, as ${ \lambda(0,\theta^\ast)=\lambda_{\rm per}(\theta^\ast) }$ is an isolated point of ${ \sigma_{\rm point} \big(L_{\rm per}(\theta^\ast) \big) }$ (any point has this property	), ${ \lambda(\eta,\theta) }$ is an isolated point of ${ \sigma_{\rm point} \big(L^{\Phi_\eta}(\theta^\ast) \big) }$ for each ${ (\eta,\theta) \in \mathcal{U} }$. Thus, we have 
${ \lambda(0,\cdot) = \lambda_{\rm per}(\cdot) }$ in a neighbourhood of ${ \theta^\ast }$. We now denote ${ \psi_{\rm per}(\cdot) := \psi(0,\cdot) }$. Without loss of generality, we assume ${ \int_{[0,1)^n} {\vert \psi_{\rm per}(\theta^\ast) \vert}^2 dy = 1 }$. Moreover, we shall assume that the homogenized (periodic) matrix 
${ A_{\rm per}^\ast = D_{\! \theta}^2 \lambda_{\rm per}(\theta^\ast) }$ is invertible which happens if $\theta=\theta^{\ast}$ is a point of local minimum or local maximum strict of 
$\R^n\ni \theta\mapsto \lambda_{\rm per}(\theta)$. Thus, an immediate application of the Implicit Function Theorem gives us the following lemma:

\begin{lemma}\label{6487369847639gfhdghjdftrtrtfgcbvbv}
		Let the condition \eqref{7t8drtys65edsrt3xcvvcxcvb} be satisfied and ${ A_{\rm per}^\ast}$ be an invertible matrix. Then, there exists a neighborhood ${ \mathcal{V} }$ of ${ 0 }$, ${ 0 \in \mathcal{V} \subset \mathbb{R} }$, and a ${ \mathbb{R}^n }$-value analytic map
		\begin{equation*}
			\theta (\cdot) : \eta \in \mathcal{V} \mapsto \theta(\eta) \in \mathbb{R}^n, 
		\end{equation*}
		such that ${ \theta(0)=\theta^\ast }$ and
		\begin{equation}\label{67trsdasdsoktig}
			\nabla_{\!\! \theta} \lambda \big( \eta,\theta(\eta) \big) = 0, \;\; \forall \eta \in \mathcal{V}.
		\end{equation}
\end{lemma}

By the analytic structure of the functions in \eqref{67ty3uhrjefd67tgrefdcx8ur7u} and the Lemma \ref{6487369847639gfhdghjdftrtrtfgcbvbv}, there exists a neighborhood ${ \mathcal{V} }$ of ${ 0 }$, ${ 0 \in \mathcal{V} \subset \mathbb{R} }$, such that 
	\begin{equation}\label{563gdc}
		\begin{split}
			(i) & \;\;\, \eta \in \mathcal{V} \; \mapsto \; \lambda \big( \eta, \theta(\eta) \big) \in \mathbb{R}, \\
			(ii) & \;\;\, \eta \in \mathcal{V} \; \mapsto \; \psi \big( \eta, \theta(\eta) \big) \in \mathcal{H} \setminus \{0\}, \\
			(iii) & \;\;\, \eta \in \mathcal{V} \; \mapsto \; \xi_k \big( \eta, \theta(\eta) \big) \in \mathcal{H}, \forall \{1,\ldots,n\},
		\end{split}
	\end{equation}
	are analytic functions, where ${ \xi_k(\eta,\theta) := (2i \pi)^{-1}{\partial_{\theta_k} \psi} (\eta,\theta) }$, for ${ k\in\{1,\ldots,n\} }$. We also consider ${ \xi_{k,{\rm per}}(\cdot) = \xi_k(0,\cdot) }$. Furthermore, by \eqref{7rter44} and \eqref{67trsdasdsoktig}, for each fixed ${ \eta \in \mathcal{V} }$ we have that the pair ${ \big( \theta(\eta),\lambda \big( \eta, \theta(\eta) \big) \big) \in \mathbb{R}^n \times \mathbb{R} }$ satisfies: 
	\begin{equation}\label{674tyghd}
		\begin{split}
			(i) & \;\;\, \lambda(\eta,\theta(\eta)) \; \text{is a simple eigenvalue of} \; L^{\Phi_\eta}\big( \theta(\eta) \big), \\
			(ii) & \;\;\, \theta(\eta) \; \text{is a critical point of} \; \lambda(\eta,\cdot), \, \text{that is}, \nabla_{\!\! \theta} \lambda(\eta, \theta(\eta)) = 0.
		\end{split}
	\end{equation}
	This means that the Theorem \ref{876427463tggfdhgdfgkkjjlmk} can be used. Before, we establish a much simplified notations for the functions in \eqref{563gdc} as follows:
	\begin{equation}\label{798723678rtyd5rdftrgdfdfsdssss}
		\begin{split}
			(i) & \;\;\, \theta_\eta := \theta(\eta), \\
			(ii) & \;\;\, \lambda_\eta := \lambda \big( \eta,\theta(\eta) \big), \\
			(iii) & \;\;\, \psi_\eta := \psi \big( \eta,\theta(\eta) \big), \\
			(iv) & \;\;\, \xi_{k,\eta} := \xi_k \big( \eta,\theta(\eta) \big), \, k\in\{1,\ldots,n\}.
		\end{split}
	\end{equation}

	Finally, from \eqref{674tyghd}, for each fixed ${ \eta \in \mathcal{V} }$, the notion of well-preparedness for the initial data $u_{\eta\varepsilon}^0$ is given as below. 
	
	\begin{equation}\label{well-prep.I}
		u_{\eta\varepsilon}^0(x,\omega) = e^{2i\pi \frac{\theta_\eta \cdot x}{\varepsilon}} \,  v^0(x) \, \psi_\eta {\left( \Phi_{\eta}^{-1} {\left( \frac{x}{\varepsilon}, \omega \right)}, \omega \right)}, \; (x,\omega) \in \mathbb{R}^n \times \Omega,
	\end{equation}
where ${ v^0 \in C_{\rm c}^\infty(\mathbb{R}^n) }$. Thus, applying the Theorem \ref{876427463tggfdhgdfgkkjjlmk}, if ${ u_{\eta\varepsilon} }$ is solution of \eqref{765tdyyuty67tsss}, the sequence in ${ \varepsilon>0 }$
	\begin{equation*}
		v_{\eta\varepsilon}(t,x,\widetilde{\omega}) = e^{ -{\left( i \frac{\lambda_\eta t}{\varepsilon^2} + 2i\pi \frac{\theta_\eta \cdot x}{\varepsilon} \right)} } u_{\eta\varepsilon}(t,x,\widetilde{\omega}), \;\, (t,x,\widetilde{\omega}) \in \mathbb{R}^{n+1}_T \times \Omega, 
	\end{equation*}
	$\Phi_{\omega}-$two-scale converges to the limit ${ v_{\eta}(t,x) \, \psi_{\eta}{\big( \Phi^{-1}(z,\omega),\omega \big)} }$ with
	\begin{equation*}
		\lim_{\varepsilon \to 0} \iint_{\mathbb{R}^{n+1}_T} \! {\left\vert v_{\eta \varepsilon} (t,x,\widetilde{\omega}) - v_{\eta}(t,x) \, \psi_{\eta}{\left( \Phi^{-1}_{\eta} {\left(\frac{x}{\varepsilon},\widetilde{\omega} \right)}, \widetilde{\omega} \right)} \right\vert}^2 dx \, dt \, = \, 0,
	\end{equation*}
for a.e. ${ \widetilde{\omega} \in \Omega }$, where ${ v_{\eta} \in C \big( [0,T], L^2(\mathbb{R}^n) \big) }$ is the unique solution of the homogenized Schr\"odinger equation 
	\begin{equation}\label{askdjfhucomojkfdfd}
		\left\{
		\begin{array}{c}
			i \displaystyle\frac{\partial v_\eta}{\partial t} - {\rm div} {\left( A^\ast_{\eta} \nabla v_\eta \right)} + U_{\! \eta}^\ast v_\eta = 0 \, , \;\, \text{in} \;\, \mathbb{R}^{n+1}_T, \\ [7,5pt]
			v_\eta(0,x) = v^0(x) \, , \;\, x\in \mathbb{R}^n,
		\end{array}
		\right.
	\end{equation}
with effective coefficients ${ A^\ast_{\eta} = D_{\! \theta}^2 \lambda \big( \eta,\theta(\eta) \big) }$ and 
	\begin{equation}\label{783874tgffffg}
		U^\ast_{\! \eta} =  c^{-1}_{\eta} \int_\Omega \int_{\Phi_{\eta}([0,1)^n, \omega)} U_{\rm per}{\big( \Phi^{-1}_{\eta} (z, \omega) \big)}  {\left\vert \psi_{\eta} {\big( \Phi^{-1}_{\eta} (z,\omega), \omega \big)} \right\vert}^2 dz \, d\mathbb{P}(\omega),
	\end{equation}
	where
	\begin{equation}\label{874326984yghedf}
		c_{\eta} = \int_\Omega \int_{\Phi_{\eta}([0,1)^n, \omega)} {\left\vert \psi_{\eta} {\big( \Phi^{-1}_{\eta} (z,\omega), \omega \big)} \right\vert}^2 dz \, d\mathbb{P}(\omega).
	\end{equation}
	
\begin{remark} 
		We remember that, using the equality~\eqref{7358586tygfjdshfbvvcc}, we have for each ${ \eta }$ fixed that the matrix ${ B_\eta \in \mathbb{R}^{n \times n} }$ must satisfy for ${ k,\ell \in \{1,\ldots,n\} }$
		\begin{equation}\label{786587tdyghs7rsdfxsdfsdf}
		\begin{split}
			& (B_\eta)_{k\ell} := c_\eta^{-1} \bigg[ \int_\Omega\int_{\Phi_\eta([0,1)^n,\omega)} A_{\rm per}{\left( \Phi_\eta^{-1}(z,\omega) \right)} {\left( e_\ell \, \psi_\eta{\left( \Phi_\eta^{-1}(z,\omega),\omega \right)}  \right)} \cdot \\
			& \hspace{7.85cm} \overline{\left( e_k \, \psi_\eta{\left( \Phi_\eta^{-1}(z,\omega),\omega \right)} \right)} \, dz \, d\mathbb{P}(\omega) \\
			& + \int_\Omega\int_{\Phi_\eta([0,1)^n,\omega)} A_{\rm per}{\left( \Phi_\eta^{-1}(z,\omega) \right)} {\left( e_\ell \, \psi_\eta{\left( \Phi_\eta^{-1}(z,\omega),\omega \right)}  \right)} \cdot \\
			& \hspace{6cm} \overline{{\left( \nabla_{\!\! z} + 2i\pi\theta_\eta \right)} {\left( \xi_{k,\eta}{\left( \Phi_\eta^{-1}(z,\omega),\omega \right)} \right)}} \, dz \, d\mathbb{P}(\omega) \\
			& - \int_\Omega\int_{\Phi_\eta([0,1)^n,\omega)} A_{\rm per}{\left( \Phi_\eta^{-1}(z,\omega) \right)} {\left( \nabla_{\!\! z} + 2i\pi\theta_\eta \right)} {\left( \psi_\eta{\left( \Phi_\eta^{-1}(z,\omega),\omega \right)} \right)} \cdot \\
			& \hspace{7.8cm} \overline{{\left( e_\ell \, \xi_{k,\eta}{\left( \Phi_\eta^{-1}(z,\omega),\omega \right)}  \right)}} \, dz \, d\mathbb{P}(\omega) \bigg],
		\end{split}
		\end{equation}
		and the homogenized matrix can be written as 
		${ A_\eta^\ast = 2^{-1} {\big( B_\eta + B_\eta^t \big)} }$. 	
\end{remark}
	
\subsubsection{Expansion of the effective coefficients}
As a consequence of the formula of the effective coefficients of the homogenized equation~\eqref{askdjfhucomojkfdfd}, we have the following proposition:

	\begin{proposition}\label{jnchndhbvgfbdtegdferfer}
		The maps ${ \eta \mapsto A_\eta^\ast, B_\eta \in \mathbb{R}^{n \times n} }$ and ${ \eta \mapsto U_\eta^\ast \in \mathbb{R} }$ are analytics in a neighbourhood of ${ \eta=0 }$.
	\end{proposition}
\begin{proof}
		Let us assume ${ \mathcal{U} }$ and ${ \mathcal{V} }$ as in \eqref{67ty3uhrjefd67tgrefdcx8ur7u} and \eqref{563gdc}, respectively. For each ${ \eta \in \mathcal{V} }$, the above arguments give us the formula ${ A^\ast_{\eta} = D_{\! \theta}^2 \lambda \big( \eta,\theta(\eta) \big) }$. Thus, as ${ (\eta,\theta) \in \mathcal{U} \mapsto D_{\! \theta}^2\lambda(\eta,\theta) \in \mathbb{R}^{n \times n} }$ and ${ \eta \in \mathcal{V} \mapsto \theta(\eta) \in \mathbb{R}^n }$ are analytic maps, we conclude that ${ \eta \in \mathcal{V} \mapsto D_\theta^2 \lambda(\eta,\theta(\eta)) \in \mathbb{R}^{n \times n} }$ is also an analytic map. This means that ${ \eta \mapsto A_\eta^\ast }$ is an analytic map. 

\medskip

		From \eqref{783874tgffffg} and \eqref{874326984yghedf}, making a change of variables, we have
		\begin{equation*}
			U^\ast_{\! \eta} =  c^{-1}_{\eta} \int_\Omega \int_{[0,1)^n} U_{\rm per}(y)  {\left\vert \psi_{\eta}(y,\omega) \right\vert}^2 {\rm det} [\nabla_{\!\! y} \Phi_\eta (y,\omega)] \, dz \, d\mathbb{P}(\omega)
		\end{equation*}
		and
		\begin{equation*}
			c_\eta =  \int_\Omega \int_{[0,1)^n} {\left\vert \psi_{\eta}(y,\omega) \right\vert}^2 {\rm det} [\nabla_{\!\! y} \Phi_\eta (y,\omega)] \, dz \, d\mathbb{P}(\omega) \not= 0.
		\end{equation*}
		Then, as the map ${ \eta \mapsto \psi_\eta \in \mathcal{H} \setminus\{0\} }$ is analytic, the map ${ \eta \mapsto c_\eta \not= 0 }$ is also analytic. Hence the map ${ \eta \mapsto c_\eta^{-1} }$ is analytic. Therefore, ${ \eta \mapsto U_\eta^\ast }$ is analytic.
	\end{proof}
	
As a consequence of this proposition, there exist ${ \{ A^{(j)},\,B^{(j)} \}_{j \in \mathbb{N}} \subset \mathbb{R}^{n \times n} }$ and ${ \{ U^{(j)} \}_{j \in \mathbb{N}} \subset \mathbb{R} }$ such that 
	\begin{equation}\label{678435}
		\left\{
		\begin{array}{lll}
			A_\eta^\ast & = & A^{(0)} + \eta A^{(1)} + \eta^2 A^{(2)} +\ldots, \\ [5pt]
			U_\eta^\ast & = & U^{(0)} + \eta U^{(1)} + \eta^2 U^{(2)} + \ldots,\\ [5pt]
			B_\eta^\ast &= & B^{(0)} + \eta B^{(1)} + \eta^2 B^{(2)} +\ldots.
		\end{array}
		\right.
	\end{equation}

Now, the object of our interest is determine the terms of order ${ \eta^0 }$ and ${ \eta }$ of these homogenized coefficients. For this purpose, guided by~\ref{654367ytr6tfclmlml} and 
by the formulas \eqref{786587tdyghs7rsdfxsdfsdf}, \eqref{783874tgffffg} and \eqref{874326984yghedf}, we shall analyse the expansion of the analytic functions in \eqref{798723678rtyd5rdftrgdfdfsdssss}. By analytic property, there exist the sequences ${ {\{ \theta^{(j)} \}}_{j \in \mathbb{N}} \subset \mathbb{R}^n }$, ${ {\{ \lambda^{(j)} \}}_{j \in \mathbb{N}} \subset \mathbb{R} }$, ${ {\{ \psi^{(j)} \}}_{j \in \mathbb{N}} \subset \mathcal{H} }$ and ${ {\{ \xi_k^{(j)} \}}_{j \in \mathbb{N}} \subset \mathcal{H} }$, ${ k\in\{1,\ldots,n\} }$, such that, for ${ k\in\{1,\ldots,n\} }$
	\begin{eqnarray}
		\theta_\eta & = & \theta^{(0)} + \eta \theta^{(1)} + \eta^2 \theta^{(2)} + \ldots = \theta^{(0)} + \eta \theta^{(1)} + \mathrm{O}(\eta^2), \label{9789789794r6ttrtrtr} \\
		\lambda_\eta & = & \lambda^{(0)} + \eta \lambda^{(1)} + \eta^2 \lambda^{(2)} + \ldots = \lambda^{(0)} + \eta \lambda^{(1)} + \mathrm{O}(\eta^2), \label{6t8y365873586edtygc} \\
		\psi_\eta & = & \psi^{(0)} + \eta\psi^{(1)} + \eta^2 \psi^{(2)} + \ldots = \psi^{(0)} + \eta \psi^{(1)} + \mathrm{O}(\eta^2), \label{099uiuiyujhjchhtydfty} \\
		\xi_{k,\eta} & = & \xi_k^{(0)} + \eta\xi_k^{(1)} + \eta^2 \xi_k^{(2)} + \ldots = \xi_k^{(0)} + \eta \xi_k^{(1)} + \mathrm{O}(\eta^2). \label{67tyuxcvbdfgoikjjhbhb}
	\end{eqnarray}

At first glance, in order to determine the coefficients of the expansions in~\eqref{678435} we should solve, a priori, auxiliary problems that involves both, the deterministic 
and stochastic variables. This can be a disadvantage from the point of view of numerical analysis. Our aim hereafter is to prove that, we can simplify the computations of this 
coefficients working in a periodic environment which is computationally cheaper. In order to do this, 
note that ${ \theta^{(0)} = \theta^\ast}$, ${ \lambda^{(0)}=\lambda_{\rm per}(\theta^\ast) }$, ${ \psi^{(0)} = \psi_{\rm per}(\theta^\ast) }$ and ${ \xi_k^{(0)} = \xi_{k,{\rm per}}(\theta^\ast) }$, ${ k\in\{1,\ldots,n\} }$, which satisfy
	\begin{eqnarray}
	    & & \left\{
		\begin{array}{l}
			{\left( L_{\rm per}(\theta^\ast) - \lambda_{\rm per}(\theta^\ast) \right)} {\big[ \psi_{\rm per}(\theta^\ast) \big]} = 0  \;\, \text{in} \;\, [0,1)^n, \\ [6pt]
			\hspace{1.5cm} \psi_{\rm per}(\theta^\ast) \;\; [0,1)^n\text{-periodic},
	\end{array}
		\right. \label{yhujtgvjnjnhnvfnvfshjbn} \\ [7.5pt]
		& & \left\{
		\begin{array}{l}
			{\left( L_{\rm per}(\theta^\ast) - \lambda_{\rm per}(\theta^\ast) \right)} {\big[ \xi_{k,{\rm per}}(\theta^\ast) \big]} =  \mathcal{X} {\big[ \psi_{\rm per}(\theta^\ast) \big]}  \;\, \text{in} \;\, [0,1)^n, \\ [6pt]
			\hspace{1.5cm} \xi_{k,{\rm per}}(\theta^\ast) \;\; [0,1)^n \text{-periodic},
		\end{array}
		\right. \label{yhfjsgfsfsdyhujtgvjnjnhnvfnvfshjbn}
	\end{eqnarray}
	where for ${ f\in \mathcal{H} }$
	\begin{equation*}
		\mathcal{X} {\big[ f \big]} := {\left( {\rm div}_{\! y} + 2i\pi \theta^\ast \right)} {\big\{ A_{\rm per} (y) {( e_k f )} \big\}} + {( e_k )} {\big\{ A_{\rm per}(y) {( \nabla_{\!\! y} + 2i\pi \theta^\ast )} f \big\}}.
	\end{equation*}
The equation~\eqref{yhujtgvjnjnhnvfnvfshjbn} is the spectral 
	cell equation and \eqref{yhfjsgfsfsdyhujtgvjnjnhnvfnvfshjbn} is the first auxiliary cell equation 
related to the periodic case (see Section~\ref{ACE}). 

\medskip

The following theorem show us that, 
the terms ${ \psi^{(1)} }$ and ${ \xi_k^{(1)} }$, $k\in\{1,\ldots,n\}$,
given by \eqref{099uiuiyujhjchhtydfty} and \eqref{67tyuxcvbdfgoikjjhbhb}
respectively, satisfy auxiliary type cell equations. 
	\begin{theorem}
	\label{THM58}
		Let ${ \psi^{(1)} }$ and ${ \xi_k^{(1)} }$, ${ k\in\{1,\ldots,n\} }$, be as above. Then these functions satisfy the following equations:  
		\begin{eqnarray}
			& & \left\{
			\begin{array}{l}
				{\left( L_{\rm per}(\theta^\ast) - \lambda_{\rm per}(\theta^\ast) \right)} {\big[ \psi^{(1)} \big]} = \mathcal{Y} {\big[ \psi_{\rm per}(\theta^\ast) \big]} \;\, \text{in} \;\, [0,1)^n \times \Omega, \\ [6pt]
				\hspace{1.5cm} \psi^{(1)} \; \text{stationary},
			\end{array}
			\right. \label{cvbnmdhfyhryfryhfiajbcjzx} \\ [7.5pt]
			& & \left\{
			\begin{array}{l}   
				{\left( L_{\rm per}(\theta^\ast) - \lambda_{\rm per}(\theta^\ast) \right)} {\big[ \xi_k^{(1)} \big]} = \mathcal{X}{\big[ \psi^{(1)} \big]}  \\ [6pt]
				\hspace{2cm} + \, \mathcal{Y}{\big[ \xi_{k,{\rm per}}(\theta^\ast) \big]} + \mathcal{Z}_k{\big[ \psi_{\rm per}(\theta^\ast) \big]} \;\,  \text{in} \;\, [0,1)^n \times \Omega, \\ [6pt]
				\hspace{1.5cm} \xi_k^{(1)} \; \text{stationary},
			\end{array} 
			\right. \label{cvbnmdhfdwadawdwayhryfryhfiajbcjzx}
		\end{eqnarray}
		where the operators ${ \mathcal{Y} }$ and ${ \mathcal{Z}_k }$, ${ k\in\{1,\ldots,n\} }$, are defined by
		\begin{eqnarray*}
			\mathcal{Y} {\big[ f \big]} & \!\! := \!\! & {\left( {\rm div}_{\! y} + 2i\pi \theta^\ast \right)} {\big\{ A_{\rm per} (y) {\big( -[\nabla_{\!\! y} Z](y,\omega) \nabla_{\!\! y} f + 2i\pi \theta^{(1)} f \big)} \big\}} \\ [1pt]
						& & - \, {\rm div}_{\! y} {\big\{ [\nabla_{\!\! y} Z]^t(y,\omega) A_{\rm per} (y) {(\nabla_{\!\! y} + 2i\pi\theta^\ast)} f \big\}} \\ [1pt]
						& & + {\left( 2i\pi\theta^{(1)} \right)} {\big\{ A_{\rm per}(y) {\left( \nabla_{\!\! y} + 2i\pi\theta^\ast \right)} f \big\}} \\ [1pt]
						& & + \, {\left( {\rm div}_{\! y} + 2i\pi \theta^\ast \right)}{\big\{ {\left[ {\rm div}_{\! y} Z (y,\omega) A_{\rm per} (y) \right]} {(\nabla_{\!\! y} + 2i\pi\theta^\ast)} f \big\}} + \lambda^{(1)} f \\ [1pt]
						& & + \, {\big\{ {\rm div}_{\! y} Z(y,\omega) \, {\left[ \lambda_{\rm per}(\theta^\ast) - V_{\rm per}(y) \right]} \big\}} f, \\ [6.5pt]
			\mathcal{Z}_k {\big[ f \big]} & \!\! := \!\! & {\left( {\rm div}_{\! y} + 2i\pi \theta^\ast \right)} {\big\{ {\left[ {\rm div}_{\! y} Z(y,\omega) A_{\rm per} (y) \right]} {( e_k f )} \big\}} \\ [1pt]
						& & - \, {\rm div}_{\! y} {\big\{ [\nabla_{\!\! y} Z]^t(y,\omega) A_{\rm per} (y) {( e_k f )} \big\}} + {\left( 2i\pi\theta^{(1)} \right)} {\left\{ A_{\rm per}(y) {( e_k f )} \right\}} \\ [1pt]
						& & + \, {\left( e_k \right)} {\big\{ {\big[ {\rm div}_{\! y} Z(y,\omega) A_{\rm per}(y) \big]} {\left( \nabla_{\!\! y} + 2i\pi\theta^\ast \right)} f \big\}} \\ [1pt]
						& & - \, {\left( e_k \right)} {\left\{ A_{\rm per}(y) {\left[ \nabla_{\!\! y} Z \right]}(y,\omega) \nabla_{\!\! y} f \right\}} + {\left( e_k \right)} {\big\{ A_{\rm per}(y) {( 2i\pi \theta^{(1)} f )} \big\}},
		\end{eqnarray*}
		for ${ f \in \mathcal{H} }$.
	\end{theorem}
For the proof of this theorem, we shall use essentially the structure of the spectral cell equation~\eqref{92347828454trfhfd4rfghjls} 
with periodic coefficients accomplished by stochastic deformation of identity ${ \Phi_\eta }$ together with the identities~\eqref{654367ytr6tfclmlml}.
\begin{proof}
1. For begining, let us consider the set ${ \mathcal{V} }$ as in \eqref{563gdc}. Then, making change of variables in the spectral cell equation~\eqref{92347828454trfhfd4rfghjls} adapted to this context, we find  
		\begin{eqnarray}\label{8365287erdtfrewxzqzazaazzaaz}
&&\hspace{-0.5cm} \int_{[0,1)^n} \int_\Omega \Big\{A_{\rm per}(y) \big( [\nabla_{\!\! y} \Phi_\eta]^{-1} \nabla_{\!\! y} \psi_\eta + 2i\pi \theta_\eta \psi_\eta \big) \cdot \overline{ \big( [\nabla_{\!\! y} \Phi_\eta]^{-1} 
\nabla_{\!\! y} \zeta + 2i\pi \theta_\eta \zeta \big)}  \, \nonumber\\
&& \qquad\qquad\qquad+  {\left( V_{\rm per}(y) - \lambda_\eta \right)} \, \psi_\eta \, \overline{\zeta} \,  \Big\} {\rm det} [\nabla_{\!\! y} \Phi_\eta] \, d\mathbb{P}(\omega) \, dy = 0,
		\end{eqnarray}
for all ${ \eta \in \mathcal{V} }$ and ${ \zeta \in \mathcal{H} }$. If we insert the equations~\eqref{654367ytr6tfclmlml}, \eqref{9789789794r6ttrtrtr}, \eqref{6t8y365873586edtygc} and \eqref{099uiuiyujhjchhtydfty} in equation~\eqref{8365287erdtfrewxzqzazaazzaaz} and compute the term $\eta$, we arrive at 

\begin{equation*}
		\begin{array}{l}
			\displaystyle { \int_{[0,1)^n} \! \int_\Omega \! \Big\{ A_{\rm per}(y) {\left( \nabla_{\!\! y} \psi^{(1)} + 2i\pi \theta^\ast \psi^{(1)} \right)} \cdot \overline{{\left( \nabla_{\!\! y} \zeta + 2i\pi \theta^\ast \zeta \right)}}+{\left( V_{\rm per}(y) - \lambda_{\rm per}(\theta^\ast) \right)} \psi^{(1)} \, \overline{\zeta}} \\ [15pt]
			\displaystyle \qquad\qquad\qquad+ \, \! A_{\rm per}(y) {\left( - [\nabla_{\!\! y} Z] \, \nabla_{\!\! y} \psi_{\rm per}(\theta^\ast) + 2i\pi\theta^{(1)}\psi_{\rm per}(\theta^\ast) \right)} \! \cdot \! \overline{{\left( \nabla_{\!\! y} \zeta + 2i\pi \theta^\ast \zeta \right)}}\\ [15pt]
			\displaystyle\qquad\qquad\qquad +  A_{\rm per}(y) {\left( \nabla_{\!\! y} \psi_{\rm per}(\theta^\ast) + 2i\pi \theta^\ast \psi_{\rm per}(\theta^\ast) \right)} \cdot \overline{{\left( -[\nabla_{\!\! y} Z] \nabla_{\!\! y} \zeta + 2i\pi\theta^{(1)} \zeta \right)}} \\ [15pt]
			\displaystyle\qquad\qquad\qquad + A_{\rm per}(y) {\left( \nabla_{\!\! y} \psi_{\rm per}(\theta^\ast) + 2i\pi \theta^\ast \psi_{\rm per}(\theta^\ast) \right)} \cdot \overline{{\left( \nabla_{\!\! y} \zeta + 2i\pi \theta^\ast \zeta \right)}} \, {\rm div}_{\! y} Z \\ [15pt]
			\displaystyle \qquad\qquad\qquad- \lambda^{(1)} \, \psi_{\rm per}(\theta^\ast) \, \overline{\zeta}+ {\left( V_{\rm per}(y) - \lambda_{\rm per}(\theta^\ast) \right)} \psi^{(0)} \, \overline{\zeta} \, {\rm div}_{\! y} Z\Big\}\,d\mathbb{P}(\omega) \, dy=0,
		\end{array}
		\end{equation*}
		for all ${ \eta \in \mathcal{V} }$ and ${ \zeta \in \mathcal{H} }$. This equation is the variational formulation of the equation \eqref{cvbnmdhfyhryfryhfiajbcjzx}, which concludes the first part of the proof. 	

\medskip
2. For the second part of the proof, we have
		\begin{equation}\label{8365wd}
		\begin{split}
			& \displaystyle \int_{[0,1)^n} \int_\Omega\Big\{ A_{\rm per}(y) \big( [\nabla_{\!\! y} \Phi_\eta]^{-1} \nabla_{\!\! y} \xi_{k,\eta} + 2i\pi \theta_\eta \xi_{k,\eta} \big) \cdot \overline{\big( [\nabla_{\!\! y} \Phi_\eta]^{-1} \nabla_{\!\! y} \zeta + 2i\pi \theta_\eta \zeta \big)} \\
			& \displaystyle\qquad\qquad\qquad + A_{\rm per}(y) {\left( e_k \, \psi_\eta \right)} \cdot \overline{\big( [\nabla_{\!\! y} \Phi_\eta]^{-1} \nabla_{\!\! y} \zeta + 2i\pi \theta_\eta \zeta \big)} \\
			& \displaystyle\qquad\qquad\qquad - A_{\rm per}(y) \big( [\nabla_{\!\! y} \Phi_\eta]^{-1} \nabla_{\!\! y} \psi_\eta + 2i\pi \theta_\eta \psi_\eta \big) \cdot \overline{{\left( e_k \, \zeta \right)}} \\
			& \displaystyle\qquad\quad + {\left( V_{\rm per}(y) - \lambda_\eta \right)} \, \xi_{k,\eta} \, \overline{\zeta} - \, \frac{1}{2i\pi} \frac{\partial \lambda}{\partial \theta_k}(\eta,\theta(\eta))\,
			\psi_\eta \, \overline{\zeta}\Big\}\,{\rm det} [\nabla_{\!\! y} \Phi_\eta] \, d\mathbb{P}(\omega) \, dy = 0,
		\end{split}
		\end{equation}
		for all ${ \eta \in \mathcal{V} }$, ${ \zeta \in \mathcal{H} }$ and ${ k \in \{1,\ldots,n\} }$. Hence, taking into account the Lemma \ref{6487369847639gfhdghjdftrtrtfgcbvbv} and 
inserting the equations \eqref{654367ytr6tfclmlml}, \eqref{9789789794r6ttrtrtr}, \eqref{6t8y365873586edtygc}, \eqref{099uiuiyujhjchhtydfty} and \eqref{67tyuxcvbdfgoikjjhbhb} 
in equation~\eqref{8365wd}, a computation of the term of order ${ \eta }$ lead us to

\begin{equation*}
		\begin{array}{l} 
			\displaystyle\hspace{-0.5cm} \int_{[0,1)^n} \int_\Omega \Big\{A_{\rm per}(y) \big( \nabla_{\!\! y}\xi_k^{(1)} + 2i\pi\theta^\ast \xi_k^{(1)} \big) \cdot \overline{ \big( \nabla_{\!\! y} \zeta + 2i\pi \theta^\ast \zeta \big)}+ {\left( V_{\rm per}(y) - \lambda_{\rm per}(\theta^\ast) \right)} \xi_k^{(1)} \, \overline{\zeta} \\ [15pt]
			\displaystyle\qquad\qquad + A_{\rm per}(y) {\left( e_k \, \psi^{(1)} \right)} \cdot \overline{ \big( \nabla_{\!\! y} \zeta + 2i\pi \theta^\ast \zeta \big)} -
	A_{\rm per}(y) \big( \nabla_{\!\! y} \psi^{(1)} + 2i\pi \theta^\ast \psi^{(1)} \big) \cdot \overline{{\left( e_k \, \zeta \right)}}\\ [15pt]
				\displaystyle\qquad\qquad + A_{\rm per}(y) \big( - [\nabla_{\!\! y} Z] \nabla_{\!\! y} \xi_{k,{\rm per}}(\theta^\ast) + 2i\pi\theta^{(1)} \xi_{k,{\rm per}}(\theta^\ast) \big) \cdot \overline{ \big( \nabla_{\!\! y} \zeta + 2i\pi \theta^\ast \zeta \big)} \\ [15pt]
			\displaystyle\qquad\qquad + A_{\rm per}(y) \big( \nabla_{\!\! y} \xi_{k,{\rm per}}(\theta^\ast) + 2i\pi \theta^\ast \xi_{k,{\rm per}}(\theta^\ast) \big) \cdot \overline{ \big( -[\nabla_{\!\! y} Z] \nabla_{\!\! y} \zeta + 2i\pi\theta^{(1)} \zeta \big)}  \\ [15pt]
			\displaystyle\qquad\qquad +  A_{\rm per}(y) \big( \nabla_{\!\! y} \xi_{k,{\rm per}}(\theta^\ast) + 2i\pi \theta^\ast \xi_{k,{\rm per}}(\theta^\ast) \big) \cdot \overline{\big( \nabla_{\!\! y} \zeta + 2i\pi \theta^\ast \zeta \big)} \, {\rm div}_{\! y} Z \\ [15pt]
			\displaystyle\qquad\qquad\qquad\qquad - \lambda^{(1)} \, \xi_{k,{\rm per}}(\theta^\ast) \, \overline{\zeta} + {\left( V_{\rm per}(y) - \lambda_{\rm per}(\theta^\ast) \right)} \xi_{k,{\rm per}}(\theta^\ast) \, \overline{\zeta} \, {\rm div}_{\! y} Z\\ [15pt]
			\displaystyle\qquad\qquad\qquad +\, A_{\rm per}(y) {\left( e_k \, \psi_{\rm per}(\theta^\ast) \right)} \cdot \Big(\overline{ \big( \nabla_{\!\! y} \zeta + 2i\pi \theta^\ast \zeta \big) \, {\rm div}_{\! y} Z -[\nabla_{\!\! y} Z] \nabla_{\!\! y} \zeta + 2i\pi\theta^{(1)} \zeta}\Big) \\ [15pt]
			\displaystyle\qquad\qquad\qquad\qquad\qquad\qquad -  A_{\rm per}(y) \big( \nabla_{\!\! y} \psi_{\rm per}(\theta^\ast) + 2i\pi \theta^\ast \psi_{\rm per}(\theta^\ast) \big) \cdot \overline{{\left( e_k \, \zeta \right)}} \, {\rm div}_{\! y} Z  \\ [15pt]
			\displaystyle\qquad\qquad -\, A_{\rm per}(y) \big( - [\nabla_{\!\! y} Z] \, \nabla_{\!\! y} \psi_{\rm per}(\theta^\ast) + 2i\pi\theta^{(1)} \psi_{\rm per}(\theta^\ast) \big) \cdot \overline{{\left( e_k \, \zeta \right)}}\Big\} \, d\mathbb{P}(\omega) \, dy =0,
		\end{array}
		\end{equation*}
		for all ${ \zeta \in \mathcal{H} }$. Noting that this is the variational formulation of the equation \eqref{cvbnmdhfdwadawdwayhryfryhfiajbcjzx}, we conclude the proof. 
\end{proof}

We remember the reader that if $f:\R^n\times\Omega\to\R$ is a stationary function, then we shall use the following notation 
$$
    \mathbb{E}[f(x,\cdot)]= \int_\Omega f(x,\omega) \,d\mathbb{P}(\omega),
$$
for any $x\in\R^n$. Roughly speaking, the theorem below tell us that the homogenized matrix of the problem~\eqref{765tdyyuty67tsss} can be obtained by solving periodic problems. 

\begin{theorem}
\label{873627yuhfdd}
		Let ${ A_\eta^\ast }$ be the homogenized matrix as in \eqref{askdjfhucomojkfdfd}. Then
		\begin{equation*}
			A_\eta^\ast = A_{\rm per}^\ast + \eta A^{(1)} + \mathrm{O}(\eta^2).
		\end{equation*}
Moreover, the term of order ${ \eta^0 }$ is given by the homogenized matrix of the periodic case, that is, ${ A_{\rm per}^\ast = 2^{-1}{\left( B^{(0)} + (B^{(0)})^t \right)} }$, where the matrix ${ B^{(0)} }$ is the term of order ${ \eta^0 }$ in \eqref{678435} and it is defined by
		\begin{equation*}
		\begin{array}{r}
			\displaystyle (B^{(0)})_{k\ell} := \int_{[0,1)^n} A_{\rm per}(y) {(e_\ell \, \psi_{\rm per}(\theta^\ast))} \cdot {(e_k \, \overline{\psi_{\rm per}(\theta^\ast)})} \, dy \hspace{4cm} \\ [10pt]
			\displaystyle + \, \int_{[0,1)^n} A_{\rm per}(y) {(e_\ell \, \psi_{\rm per}(\theta^\ast))} \cdot \overline{\left( \nabla_{\!\! y} \xi_{k,{\rm per}}(\theta^\ast) + 2i\pi \theta^\ast \xi_{k,{\rm per}}(\theta^\ast) \right)} \, dy \\ [10pt]
			\displaystyle - \, \int_{[0,1)^n} A_{\rm per}(y) {\Big( {\left( \nabla_{\!\! y} \psi_{\rm per}(\theta^\ast) + 2i\pi \theta^\ast \psi_{\rm per}(\theta^\ast) \right)} \Big)} \cdot \overline{{(e_\ell \, \xi_{k,{\rm per}}(\theta^\ast))}} \, dy .
		\end{array}    
		\end{equation*}
		The term of order ${ \eta }$ is given by ${ A^{(1)} = 2^{-1} {\left( B^{(1)}+(B^{(1)})^t \right)} }$, where the matrix ${ B^{(1)} }$ is the term of order ${ \eta }$ in \eqref{678435} and it is defined by
		\begin{equation*}
		\begin{array}{l}
			\displaystyle (B^{(1)})_{k\ell} = { {\bigg[ \int_{[0,1)^n} A_{\rm per}(y) {(e_\ell \, \psi_{\rm per}(\theta^\ast))} \cdot {(e_k \, \overline{\psi_{\rm per}(\theta^\ast)})} \, \mathbb{E}\Big[{\rm div}_{\! y} Z(y,\cdot)\Big] \, dy }} \\ [11pt]
			\displaystyle + \, \int_{[0,1)^n} A_{\rm per}(y) {(e_\ell \, \psi_{\rm per}(\theta^\ast))} \cdot {( e_k \, \overline{ \mathbb{E}\big[ \psi^{(1)}(y,\cdot)\big]} )} \, dy \\ [11pt]
			\displaystyle + \int_{[0,1)^n} A_{\rm per}(y) {\left( e_\ell \, {\mathbb{E}\big[ \psi^{(1)}(y,\cdot)\big]} \right)} \cdot {(e_k \, \overline{\psi_{\rm per}(\theta^\ast)})} \, dy \\ [11pt]
			\displaystyle + \int_{[0,1)^n} A_{\rm per}(y) {(e_\ell \, \psi_{\rm per}(\theta^\ast))} \cdot 
			\overline{\left( \nabla_{\!\! y} \xi_{k,{\rm per}}(\theta^\ast) + 2i\pi \theta^\ast \xi_{k,{\rm per}}(\theta^\ast) \right)} \, {\mathbb{E}\big[{\rm div}_{\! y} Z(y,\cdot)\big]} \, dy \\ [11pt]
			\displaystyle + \int_{[0,1)^n} A_{\rm per}(y) {(e_\ell \, \psi_{\rm per}(\theta^\ast))} \cdot \overline{( \nabla_{\!\! y} {\mathbb{E}\big[ \xi_k^{(1)}(y,\cdot)
			\big]} + 2i\pi\theta^\ast {\mathbb{E}\Big[ \xi_k^{(1)}(y,\cdot)\Big]}} \\ [11pt]
			\hspace{4cm} \overline{ + \, 2i\pi\theta^{(1)} \xi_{k,{\rm per}}(\theta^\ast)-{\mathbb{E}\Big[[\nabla_{\!\! y} Z](y,\cdot)\Big]} \nabla_{\!\! y} \xi_{k,{\rm per}}(\theta^\ast)}) \, dy \\ [11pt]
			\displaystyle + \int_{[0,1)^n} \!\!\! A_{\rm per}(y) {\left( e_\ell \, {\mathbb{E}\Big[ \psi^{(1)}(y,\cdot)\Big]} \right)} \cdot \overline{\left( \nabla_{\!\! y} \xi_{k,{\rm per}}(\theta^\ast) + 2i\pi \theta^\ast \xi_{k,{\rm per}}(\theta^\ast) \right)} \, dy \\ [11pt]
			\displaystyle - \int_{[0,1)^n} \!\!\! A_{\rm per}(y) {( {\left( \nabla_{\!\! y} \psi_{\rm per}(\theta^\ast) + 2i\pi \theta^\ast \psi_{\rm per}(\theta^\ast) \right)})} \cdot \overline{{(e_\ell \, \xi_{k,{\rm per}}(\theta^\ast))}} \, {\mathbb{E}\big[{\rm div}_{\! y} Z(y,\cdot)\big]} \, dy \\ [11pt]
			\displaystyle - \int_{[0,1)^n} \!\!\! A_{\rm per}(y) {\Big( {\left( \nabla_{\!\! y} \psi_{\rm per}(\theta^\ast) + 2i\pi \theta^\ast \psi_{\rm per}(\theta^\ast) \right)} \Big)} \cdot \overline{{\Big(e_\ell \, {\mathbb{E}\Big[ \xi_k^{(1)}(y,\cdot)\Big]} \Big)}} \, dy
		\end{array}
		\end{equation*}
		\begin{equation*}
		\begin{array}{l}
			\displaystyle \hspace{0.75cm} - \int_{[0,1)^n} A_{\rm per}(y) {\left( \nabla_{\!\! y} {\mathbb{E}\Big[ \psi^{(1)}(y,\cdot)\Big]} + 
			2i\pi\theta^\ast {\mathbb{E}\Big[ \psi^{(1)}(y,\cdot)\Big]}\right.} \\ [10pt]
			\hspace{2cm} {{\left. + \, 2i\pi\theta^{(1)} \psi_{\rm per}(\theta^\ast) -{\mathbb{E}\Big[[\nabla_{\!\! y} Z](y,\cdot)\Big]} \nabla_{\!\! y} \psi_{\rm per}(\theta^\ast)\right)} \cdot \overline{{(e_\ell \, \xi_{k,{\rm per}}(\theta^\ast))}} \, dy \bigg]} \\ [11pt]
			\displaystyle - {\bigg[ \int_{[0,1)^n} {\vert \psi_{\rm per}(\theta^\ast) \vert}^2 {\mathbb{E}\Big[ {\rm div}_{\! y} Z(y,\cdot)\Big]} \, dy + \int_{[0,1)^n} \psi_{\rm per}(\theta^\ast) \, \overline{ \mathbb{E}\Big[\psi^{(1)}(y,\cdot) \Big]} } \, dy  \\ [11pt]
			\displaystyle \hspace{0.75cm} { + \int_{[0,1)^n} {\mathbb{E}\Big[\psi^{(1)}(y,\cdot) \Big]} \, \overline{\psi_{\rm per}(\theta^\ast)} \, dy \bigg]} \cdot {\bigg[ \int_{[0,1)^n}
			A_{\rm per}(y) {(e_\ell \, \psi_{\rm per}(\theta^\ast))} \cdot {(e_k \, \overline{\psi_{\rm per}(\theta^\ast)})} \, dy } \\ [11pt]
			\displaystyle \hspace{0.75cm} + \int_{[0,1)^n} A_{\rm per}(y) {(e_\ell \, \psi_{\rm per}(\theta^\ast))} \cdot \overline{\left( \nabla_{\!\! y} \xi_{k,{\rm per}}(\theta^\ast) + 2i\pi \theta^\ast \xi_{k,{\rm per}}(\theta^\ast) \right)} \, dy \\ [11pt]
			\displaystyle \hspace{0.75cm} - \, {{ \int_{[0,1)^n} A_{\rm per}(y) {\left[ {\left( \nabla_{\!\! y} \psi_{\rm per}(\theta^\ast) + 2i\pi \theta^\ast \psi_{\rm per}(\theta^\ast) \right)} \right]} \cdot \overline{{(e_\ell \, \xi_{k,{\rm per}}(\theta^\ast))}} \, dy \bigg]}}.
		\end{array}
		\end{equation*}
	\end{theorem}
\begin{proof}
1.  Taking into account ${ \mathcal{V} }$ as in \eqref{563gdc}, we get from~\eqref{786587tdyghs7rsdfxsdfsdf}, for ${ \eta \in \mathcal{V} }$, that the homogenized matrix 
is given by ${ A_\eta^\ast = 2^{-1}(B_\eta + B_\eta^t) }$. Thus, in order to describe the terms of the expansion of ${ A_\eta^\ast }$, we only need to determine the terms in the expansion of ${ B_\eta }$. 

2. Using the equations~\eqref{654367ytr6tfclmlml} and \eqref{099uiuiyujhjchhtydfty}, the map ${ \eta \mapsto c_\eta \in (0,+\infty) }$ has an expansion about ${ \eta=0 }$. Remembering 
that ${ \int_{[0,1)^n} {\vert \psi_{\rm per}(\theta^\ast) \vert}^2 dy = 1 }$, we have

\begin{equation*}\label{47647342rfedrdrefrdfer} 
		\begin{split}
			c_\eta^{-1} = & 1 - \eta {\left[ \int_\Omega \int_{[0,1)^n} {\vert \psi_{\rm per}(\theta^\ast) \vert}^2 {\rm div}_{\! y} Z(y,\omega) \, dy \, d\mathbb{P} \right.} \\
			& {\left. + \int_\Omega \int_{{[0,1)^n}} \psi_{\rm per}(\theta^\ast) \overline{\psi^{(1)}} \, dy \, d\mathbb{P} + \int_\Omega \int_{[0,1)^n} \psi^{(1)} \overline{\psi_{\rm per}(\theta^\ast)} \, dy \, d\mathbb{P} \right]} + \; \mathrm{O}(\eta^2),
		\end{split}
		\end{equation*}
in ${ \mathbb{C} }$ as ${ \eta \to 0 }$. Thus, using the expansions \eqref{654367ytr6tfclmlml}, \eqref{9789789794r6ttrtrtr}, \eqref{099uiuiyujhjchhtydfty} and \eqref{67tyuxcvbdfgoikjjhbhb} in the formula \eqref{786587tdyghs7rsdfxsdfsdf}, the computation of the resulting term of order ${ \eta^0 }$ of ${ B_\eta }$ give us the desired expression for $(B^{(0)})_{k,\ell}$.  
The same reasoning with a little more computations, which is an exercise that we leave to the reader, allow us to obtain the expression for $(B^{(1)})_{k\ell}$. 
\end{proof}

\begin{remark} Next we stress that, the computation of the coefficients of ${ A_{\rm per}^\ast }$ is performed by solving the equations \eqref{yhujtgvjnjnhnvfnvfshjbn} and \eqref{yhfjsgfsfsdyhujtgvjnjnhnvfnvfshjbn}, which are equations with periodic boundary conditions. In order to compute the coefficients of ${ A^{(1)} }$, we need to know the functions 
${ \psi^{(1)} }$ and ${ \xi_k^{(1)} }$, ${ k\in\{1,\ldots,n\} }$, which are a priori stochastic in nature (see the equations \eqref{cvbnmdhfyhryfryhfiajbcjzx} and \eqref{cvbnmdhfdwadawdwayhryfryhfiajbcjzx}, respectively). 
But Theorem \ref{873627yuhfdd} shows that, we only need their expectation values, 
${ \mathbb{E}\Big[ \psi^{(1)}(y,\cdot)\Big] }$ and ${ \mathbb{E}\Big[ \xi_k^{(1)}(y,\cdot) \Big] }$, ${ k \in \{1,\ldots,n\} }$, which are ${ [0,1)^n }$-periodic functions and, respectively, solutions of the following equations
		\begin{eqnarray*}
			& & \left\{
			\begin{array}{l}
				{\Big( L_{\rm per}(\theta^\ast) - \lambda_{\rm per}(\theta^\ast) \Big)}\, {\mathbb{E}\Big[ \psi^{(1)}(y,\cdot)\Big]} = \mathcal{Y}_{\rm per} {\big[ \psi_{\rm per}(\theta^\ast) \big]} \;\, \text{in} \;\, [0,1)^n, \\ [6pt]
				\hspace{2cm} {\mathbb{E}\Big[ \psi^{(1)}(y,\cdot)\Big]} \;\text{is $[0,1)^n$-periodic},
			\end{array}
			\right. \\ [7.5pt]
			& & \left\{
			\begin{array}{l}   
				{\left( L_{\rm per}(\theta^\ast) - \lambda_{\rm per}(\theta^\ast) \right)}{\mathbb{E}\Big[ \xi_k^{(1)}(y,\cdot) \Big]} = 
				\mathcal{X}\Big[{\mathbb{E}\Big[ \psi^{(1)}(y,\cdot)\Big]}\Big] \\ [6pt]
				\hspace{1.75cm} + \, \mathcal{Y}_{\rm per} {\big[ \xi_{k,{\rm per}}(\theta^\ast) \big]} + \mathcal{Z}_{k,{\rm per}} {\big[ \psi_{\rm per}(\theta^\ast) \big]} \;\, \text{in} \;\, 
				[0,1)^n, \\ [6pt]
				\hspace{2cm} {\mathbb{E}\Big[ \xi_k^{(1)}(y,\cdot) \Big]} \;\text{is $[0,1)^n$-periodic},
			\end{array}
			\right.
		\end{eqnarray*}
		where for ${ f\in H^1_{\rm per}([0,1)^n) }$
		\begin{eqnarray*}
			\mathcal{Y}_{\rm per} {\big[ f \big]} & \!\! := \!\! & {\left( {\rm div}_{\! y} + 2i\pi \theta^\ast \right)} {\Big\{ A_{\rm per} (y)
			 {\big( - \mathbb{E}\Big[[\nabla_{\!\! y} Z](y,\cdot)\Big] \nabla_{\!\! y} f + 2i\pi \theta^{(1)} f \big)} \Big\}} \\ [1pt]
						& & - \, {\rm div}_{\! y} {\Big\{ \mathbb{E}\Big[[\nabla_{\!\! y} Z](y,\cdot)\Big]^t A_{\rm per} (y) {(\nabla_{\!\! y} + 2i\pi\theta^\ast)} f \Big\}} \\ [1pt]
						& & + {\left( 2i\pi\theta^{(1)} \right)} {\big\{ A_{\rm per}(y) {\left( \nabla_{\!\! y} + 2i\pi\theta^\ast \right)} f \big\}} \\ [1pt]
						& & + \, {\left( {\rm div}_{\! y} + 2i\pi \theta^\ast \right)}{\Big\{ {\left[ \mathbb{E}\Big[{\rm div}_{\! y} Z(y,\cdot)\Big] A_{\rm per} (y) \right]} {(\nabla_{\!\! y} + 2i\pi\theta^\ast)} f \Big\}} + \lambda^{(1)} f \\ [1pt]
						& & + \, {\Big\{ \mathbb{E}\Big[{\rm div}_{\! y} Z(y,\cdot)\Big] \, {\left[ \lambda_{\rm per}(\theta^\ast) - V_{\rm per}(y) \right]} \Big\}} f, 
		\end{eqnarray*}
		\begin{eqnarray*}
			\mathcal{Z}_{k,{\rm per}} {\big[ f \big]} & \!\! := \!\! & {\left( {\rm div}_{\! y} + 2i\pi \theta^\ast \right)} {\Big\{ {\left[ \mathbb{E}\Big[{\rm div}_{\! y} Z(y,\cdot)\Big] A_{\rm per} (y) \right]} {( e_k f )} \Big\}} \\ [1pt]
						& & - \, {\rm div}_{\! y} {\Big\{ \mathbb{E}\Big[[\nabla_{\!\! y} Z](y,\cdot)\Big]^t A_{\rm per} (y) {( e_k f )} \Big\}} + {\left( 2i\pi\theta^{(1)} \right)} {\left\{ A_{\rm per}(y) {( e_k f )} \right\}} \\ [1pt]
						& & + \, {\left( e_k \right)} {\Big\{ {\Big[ \mathbb{E}\Big[{\rm div}_{\! y} Z(y,\cdot)\Big] A_{\rm per}(y) \Big]} {\left( \nabla_{\!\! y} + 2i\pi\theta^\ast \right)} f 
						\Big\}} \\ [1pt]
						& & - \, {\left( e_k \right)} {\left\{ A_{\rm per}(y) \mathbb{E}\Big[[\nabla_{\!\! y} Z](y,\cdot)\Big] \nabla_{\!\! y} f \right\}} + {\left( e_k \right)} {\big\{ A_{\rm per}(y) {( 2i\pi \theta^{(1)} f )} \big\}}.
		\end{eqnarray*}
	\end{remark}
Summing up, the determination of the homogenized coefficients for~\eqref{jhjkhkjhkj765675233} is a stochastic problem in nature.  However, when we consider the 
interesting context of materials which have small deviation from perfect ones (modeled by periodic functions), this problem, in the specific case~\eqref{37285gdhddddddddddd} 
reduces, at the first two orders in $\eta$, to the simpler solution to the two periodic problems above. Both of them are of the same nature. Importantly, note that 
$Z$ in~\eqref{37285gdhddddddddddd} is only present through $\mathbb{E}\Big[{\rm div}_{\! y} Z(y,\cdot)\Big]$ and $\mathbb{E}\Big[[\nabla_{\!\! y} Z](y,\cdot)\Big]$.

\medskip
In the theorem below, we assume that the homogenized matrix of the periodic case satisfies the uniform coercive condition, that is, 
$$
 A_{\rm per}^\ast \xi \cdot \xi\ge \Lambda |\xi|^2,
$$
for some $\Lambda>0$ and for all $\xi\in\R^n$, which has experimental evidence for metals and semiconductors. 
Therefore, due to Theorem~\ref{873627yuhfdd} the homogenized matrix of the perturbed case ${ A_\eta^\ast }$ has similar 
property for $\eta\sim 0$. 
	
\begin{theorem}
\label{THM511}
Let ${ v_\eta }$ be the solution of homogenized equation \eqref{askdjfhucomojkfdfd}. Then
		\begin{equation*}
			v_\eta\Big(t,\sqrt{A_\eta^\ast}\,x\Big) = v_{\rm per}\Big(t,\sqrt{A_{\rm per}^\ast}\,x\Big) + \eta\, v^{(1)}\Big(t,\sqrt{A_{\rm per}^\ast}\,x\Big) + \mathrm{O}(\eta^2),
		\end{equation*}
weakly in ${ L^2(\mathbb{R}^n_T) }$ as ${ \eta \to 0 }$, that means,
		\begin{eqnarray*}
			&&  \int_{\mathbb{R}^n_T} {\Bigg( v_\eta\Big(t,\sqrt{A_\eta^\ast}\,x\Big)-v_{\rm per}\Big(t,\sqrt{A_{\rm per}^\ast}\,x\Big) 
-\eta\, v^{(1)}\Big(t,\sqrt{A_{\rm per}^\ast}\,x\Big) \Bigg)} \, h(t,x) \, dx \, dt\\
&&\qquad\qquad\qquad= \mathrm{O}(\eta^2),
		\end{eqnarray*}
for each ${ h \in L^2(\mathbb{R}^n_T) }$, where ${ v_{\rm per} }$ is the solution of the periodic homogenized problem  
		\begin{equation}\label{987tr7tef76756g7rg5467g546r7g5}
			\left\{
			\begin{array}{c}
				i \displaystyle\frac{\partial v_{\rm per}}{\partial t} - {\rm div} {\left( A_{\rm per}^\ast \nabla v_{\rm per} \right)} + U_{\! \rm per}^\ast v_{\rm per} = 0 , \;\, \text{in} \;\, \mathbb{R}^{n+1}_T, \\ [7,5pt]
				v_{\rm per}(0,x) = v_0(x) \, , \;\, x\in \mathbb{R}^n,
			\end{array}
			\right.
		\end{equation}
		and ${ v^{(1)} }$ is the solution of
		\begin{equation}\label{73547tr764tr63tr4387tr8743tr463847}
			\left\{
			\begin{array}{c}
				i \displaystyle\frac{\partial v^{(1)}}{\partial t} - {\rm div} {\left( A_{\rm per}^\ast \nabla v^{(1)} \right)} + U_{\! \rm per}^\ast v^{(1)} = {\rm div} {\left( A_{\rm per}^\ast \nabla v_{\rm per} \right)} - U^{(1)}v_{\rm per}, \;\, \text{in} \;\, \mathbb{R}^{n+1}_T, \\ [7,5pt]
				v^{(1)}(0,x) = v^{1}_0(x) \, , \;\, x\in \mathbb{R}^n,
			\end{array}
			\right.
		\end{equation}
		where ${ U^{(1)} }$ is the coefficient of the term of order ${ \eta }$ of the expansion ${ U_\eta^\ast }$ and $v_0^{1}\in C_c^{\infty}(\R^n)$ is given by the limit
$$
v_0^1\Big(\sqrt{A_{\rm per}^\ast}\,x\Big):= \lim_{\eta\to 0}\frac{v_0\Big(\sqrt{A_\eta^\ast}\,x\Big)-v_0\Big(\sqrt{A_{\rm per}^\ast}\,x\Big)}{\eta}.
$$
\end{theorem}

\begin{proof}
1. Taking into account the set ${ \mathcal{V} }$ as in \eqref{563gdc}, we have for ${ \eta \in \mathcal{V} }$ and from the conservation of energy of the homogenized Schr\"odinger equation \eqref{askdjfhucomojkfdfd}, that the solution ${ v_\eta : \mathbb{R}^n_T \to \mathbb{C} }$ satisfies

		\begin{equation*}
			{\Vert v_\eta \Vert}_{L^2(\mathbb{R}^{n+1}_T)} = T {\Vert v_0 \Vert}_{L^2(\mathbb{R}^n)}, \;\, \forall \eta \in \mathcal{V}.
		\end{equation*}
Thus, after possible extraction of a subsequence, we have the existence of a function ${ v^{(0)} \in L^2(\mathbb{R}^{n+1}_T) }$ such that 
		\begin{equation}\label{tvdfvcfvfvfvfcdcdcdcdc}
			v_{\eta} \; \xrightharpoonup[\eta \to 0]{} \; v^{(0)} \; \text{em} \; L^2(\mathbb{R}^{n+1}_T).
		\end{equation}

By the variational formulation of the equation \eqref{askdjfhucomojkfdfd}, we find
		\begin{equation}\label{yhnygbtgbrvftgbdfervd}
		\begin{array}{l}
			0 = \displaystyle i \int_{\mathbb{R}^n} v_0(x) \, \overline{\varphi}(0,x) \, dx - i  \int_{\mathbb{R}^n_T} v_\eta(t,x) \, \frac{\partial \overline{\varphi}}{\partial t} (t,x) \, dx \, dt \\ [15pt]
			\displaystyle \qquad\qquad+ \int_{\mathbb{R}^n_T}\Bigg\{- {\left\langle A_\eta^\ast v_\eta(t,x), D^2 {\varphi}(t,x) \right\rangle} + U_\eta^\ast v_\eta(t,x) \, \overline{\varphi}(t,x)\Bigg\} \, dx \, dt,
		\end{array}
		\end{equation}
for all ${ \varphi \in C_{\rm c}^1((-\infty,T)) \otimes C_{\rm c}^2(\mathbb{R}^n) }$. Recall that ${ {\left\langle P,Q \right\rangle} := {\rm tr}(P \overline{Q}^t) }$, for ${ P,Q }$ in 
${ \mathbb{C}^{n \times n} }$. Then, using~\eqref{tvdfvcfvfvfvfcdcdcdcdc}, the Theorem~\ref{jnchndhbvgfbdtegdferfer}, making ${ \eta \to 0 }$ and invoking the 
uniqueness property of the equation~\eqref{987tr7tef76756g7rg5467g546r7g5}, we conclude that ${ v^{(0)}=v_{\rm per} }$.

2. Now, using that ${ U_\eta^\ast = U_{\rm per}^\ast + \eta\, U^{(1)} + \mathrm{O}(\eta^2) }$ as ${ \eta \to 0 }$, defining $V_{\eta}(t,x):=v_\eta\Big(t,\sqrt{A_\eta^\ast}\,x\Big)$ and 
using the homogenized equation \eqref{askdjfhucomojkfdfd}, we arrive at 
\begin{equation}\label{PertCase1}
		\left\{
		\begin{array}{c}
			i \displaystyle\frac{\partial V_\eta}{\partial t} - \Delta V_{\eta} + U_{\rm per}^\ast V_\eta = -\Big(\eta\,U^{(1)} + \mathrm{O}(\eta^2)\Big)\,V_{\eta}\, , \;\, \text{in} \;\, \mathbb{R}^{n+1}_T, \\ [7,5pt]
			V_\eta(0,x) = v^0\Big(\sqrt{A_\eta^\ast}\,x\Big) \, , \;\, x\in \mathbb{R}^n,
		\end{array}
		\right.
	\end{equation}
Proceeding similarly with respect to $V(t,x):=v_{\rm per}\Big(t,\sqrt{A_{\rm per}^\ast}\,x\Big)$, we obtain

\begin{equation}\label{PertCase2}
		\left\{
		\begin{array}{c}
			i \displaystyle\frac{\partial V}{\partial t} - \Delta V + U_{\rm per}^\ast V_\eta = 0 \, , \;\, \text{in} \;\, \mathbb{R}^{n+1}_T, \\ [7,5pt]
			V(0,x) = v^0\Big(\sqrt{A_{\rm per}^\ast}\,x\Big) \, , \;\, x\in \mathbb{R}^n.
		\end{array}
		\right.
	\end{equation}
Now, the difference between the equations~\eqref{PertCase1} and~\eqref{PertCase2} yields, 
\begin{equation}\label{PertCase3}
		\left\{
		\begin{array}{c}
			i \displaystyle\frac{\partial (V_\eta-V)}{\partial t} - \Delta (V_{\eta}-V) + U_{\rm per}^\ast (V_\eta-V) = -\Big(\eta\,U^{(1)} + \mathrm{O}(\eta^2)\Big)\,V_{\eta}\, , \;\, \text{in} \;\, \mathbb{R}^{n+1}_T, \\ [7,5pt]
			(V_\eta-V)(0,x) = v^0\Big(\sqrt{A_\eta^\ast}\,x\Big)-v^0\Big(\sqrt{A_{\rm per}^\ast}\,x\Big) \, , \;\, x\in \mathbb{R}^n.
		\end{array}
		\right.
	\end{equation}
Hence, multiplying the last equation by $\overline{V_{\eta}-V}$, integrating over $\R^n$ and taking the imaginary part yields
$$
\frac{d}{dt}\|V_{\eta}-V{\|}_{L^2(\mathbb{R}^n)}\le \mathrm{O}(\eta), \quad \text{for $\eta\in \mathcal{V}$.}
$$
Defining 
\begin{equation*}
			W_\eta(t,x) := \frac{V_\eta(t,x)-V(t,x)}{\eta}, \;\, \eta \in \mathcal{V},
		\end{equation*}
the last inequality provides,
$
\sup_{\eta\in \mathcal{V}} \| W_{\eta}{\|}_{L^2(\mathbb{R}^{n+1}_T)}< +\infty.
$
Thus, taking a subsequence if necessary, there exists ${ v^{(1)} \in L^2(\mathbb{R}^{n+1}_T) }$ such that
		\begin{equation}\label{67dtguystrt6756456rt3yd}
			W_{\eta}(t,x) \; \xrightharpoonup[\eta \to 0]{} \; v^{(1)}\Big(t,\sqrt{A_{\rm per}^\ast}\,x\Big), \; \text{in} \; L^2(\mathbb{R}^{n+1}_T).
		\end{equation}
Hence, multiplying the equation~\eqref{PertCase3} by $\eta^{-1}$, letting $\eta\to 0$ and performing a change of variables, we reach the 
equation~\eqref{73547tr764tr63tr4387tr8743tr463847} finishing the proof of the theorem.

\end{proof}

\section*{Conflict of Interest}
Author Wladimir Neves has received research grants from CNPq
through the grant  308064/2019-4, and also by FAPERJ 
(Cientista do Nosso Estado) through the grant E-26/201.139/2021. 
Author Jean Silva has received research grants from CNPq through the Grant 303533/2020-0.

\end{document}